\documentclass{article} % default font size is 10pt

\usepackage[T1]{fontenc}
\usepackage{emptypage}
\usepackage{centernot}
\usepackage{array}
\usepackage{pst-grad}
\usepackage{pst-arrow}
\usepackage{pst-text}

\usepackage{latexsym}
\usepackage{amsfonts}
\usepackage{amssymb}
\usepackage{amsmath}
\usepackage{amscd}
\usepackage{eucal}
\usepackage{amsthm}
\usepackage{makeidx}

\usepackage{pstricks}
\usepackage{pst-node}
\usepackage{verbatim} % allows comments
\usepackage{graphics}
\usepackage{graphicx}
\usepackage{color}   
\usepackage{framed}

\usepackage{bbm}
\usepackage{pst-plot}
\usepackage{pstricks-add}

\usepackage{xcolor}
\usepackage{mdframed}

\def\1{\ensuremath{\mathbbm{1}}}%
\def\2{\ensuremath{\mathbbm{2}}}%

\newcommand{\E}{\ensuremath{\bb{E}}}

\newcommand{\cC}{{\cl{C}}}

\newcommand{\cG}{{\cl{G}}}

\newcommand{\cV}{{\cl{V}}}
\newcommand{\cW}{{\cl{W}}}

\newcommand{\mysdot}{\psset{unit=1mm}\pscircle*(0,0){0.3}}

\newcommand{\UP}{\ensuremath{U\hs{-0.15}P}}

\newcommand{\UX}{\ensuremath{U\hs{-0.1}X}}

\newcommand{\VH}{\ensuremath{V\hs{-0.1}H}}

\newcommand{\eps}{\endpspicture}
\newcommand{\ps}{\pspicture}

\newcommand{\dl}{distributive law}
\newcommand{\dls}{distributive laws}

\newcommand{\Dls}{Distributive laws}

\newcommand{\eh}{Eckmann--Hilton}

\newcommand{\bmc}{braided monoidal category}
\newcommand{\bmcs}{braided monoidal categories}

\newcommand{\dd}{doubly-degenerate}

\newcommand{\catequiv}{\simeq} 
\newcommand{\scr}{\scriptsize}

\renewcommand{\:}{\colon}

\newcommand{\bs}{\bigskip}

\newcommand{\bicats}{\ensuremath{\cat{Bicat}_s}}

\newcommand{\VoneGph}{\cat{$\cV_1$-Gph}}

\newcommand{\fcvpp}{\ensuremath{\fc_{(\cV,P)}}}

\newcommand{\fcvp}{\ensuremath{\fc_{P}}}
\newcommand{\fcvonepone}{\ensuremath{\fc_{P_1}}}
\newcommand{\fcvtwoptwo}{\ensuremath{\fc_{P_2}}}

\newcommand{\vgph}{\cat{$\cV$-Gph}}
\newcommand{\vonegph}{\cat{$\cV_1$-Gph}}
\newcommand{\vtwogph}{\cat{$\cV_2$-Gph}}
\newcommand{\vpcat}{\ensuremath{\cV\cat{-Cat}_P}}
\newcommand{\voneponecat}{\ensuremath{\cV_1\cat{-Cat}_{P_1}}}
\newcommand{\voneponecatgph}{\cat{\voneponecat-Gph}}

\newcommand{\vtwoptwocat}{\ensuremath{\cV_2\cat{-Cat}_{P_2}}}
\newcommand{\vnpncat}{\ensuremath{\cV_n\cat{-Cat}_{P_n}}}

\newcommand{\vonegphgph}{\ensuremath{\cV_1\cat{-Gph-Gph}}}

\newcommand{\vone}{\ensuremath{\cV_1}}
\newcommand{\vtwo}{\ensuremath{\cV_2}}
\newcommand{\vthree}{\ensuremath{\cV_3}}

\newcommand{\tonealg}{\cat{$T_1$-Alg}}

\newcommand{\ddbicatscat}{doubly-degenerate \bicats-category}
\newcommand{\ddbicatscats}{doubly-degenerate \bicats-categories}

\newcommand{\bicatscats}{\bicats-categories}

\newcommand{\kcat}{\twocat{$K$-Cat}}
\newcommand{\kicon}{\twocat{$K$-Icon}}
\newcommand{\kpcat}{\ensuremath{\twocat{$K$-Cat}_P}}
\newcommand{\fckp}{\ensuremath{\fc_{P}}}

\newcommand{\kgpho}{\twocatt{$K$-Gph}}

\newcommand{\kpcato}{\ensuremath{\twocatt{$K$-Cat}_P}}

\newcommand{\catgphgph}{\cat{Cat-Gph-Gph}}

\newcommand{\noi}{\noindent}

\newcommand{\cat}[1]{\ensuremath{\textrm{\bfseries {\upshape {#1}}}}}

\newcommand{\set}{\cat{Set}}
\newcommand{\Set}{{\cat{Set}}}
\newcommand{\Cat}{{\cat{Cat}}}

\newcommand{\cl}[1]{\ensuremath{\mathcal {#1}}}
\newcommand{\bb}[1]{\ensuremath{\mathbb {#1}}}
\newcommand{\ed}{\end{document}}
\newcommand{\bq}{\begin{quote}}
\newcommand{\eq}{\end{quote}}
\newcommand{\bc}{\begin{center}}
\newcommand{\ec}{\end{center}}
\newcommand{\bmp}{\noi\begin{minipage}}
\newcommand{\emp}{\end{minipage}}

\newcommand{\bfr}{\begin{flushright}}
\newcommand{\efr}{\end{flushright}}

\renewcommand{\mapsto}{\stmapsto}

\newcommand{\demph}[1]{{\bfseries #1}}
\newcommand{\twocat}[1]{\makebox{\psset{unit=1mm}\pnode(-0.7,0){a}\cat{#1}\pnode(0.7,0){b}\ncline[offset=-2.5pt,linewidth=0.6pt,doubleline=true,doublesep=0.2]{a}{b}}}

\newcommand{\twocatt}[1]{\makebox{\psset{unit=1mm}\pnode(-0.7,0){a}\cat{#1}\pnode(0.7,0){b}}}

\newcommand{\fc}{\cat{fc}}

\newcommand{\ol}{\overline}

\newcommand{\vs}[1]{\vspace*{#1em}}
\newcommand{\hs}[1]{\hspace*{#1em}}

\newcommand{\tra}{{\psset{unit=0.1cm,nodesep=0pt} \pspicture(8,0)
\pcline{->}(1,1.1)(7,1.1) \endpspicture}}

\newcommand{\mtra}{{\psset{unit=0.1cm,nodesep=0pt} \pspicture(10,0)
\pcline{->}(1,1.1)(9,1.1) \endpspicture}}

\newcommand{\ltra}{{\psset{unit=0.1cm,nodesep=0pt} \pspicture(15,0)
\pcline{->}(1.5,1.4)(13.5,1.4) \endpspicture}}

\newcommand{\tramap}[1]{{\psset{unit=0.1cm,nodesep=0pt,labelsep=1pt} \pspicture(8,4)
\pcline{->}(1,1)(7,1)\naput[npos=0.45]{\ensuremath{\scriptstyle{#1}}} \endpspicture}}

\newcommand{\tmap}{\tramap}

\newcommand{\tmapiso}[1]{{\psset{unit=0.1cm,nodesep=0pt,labelsep=1pt} \pspicture(8,4)
\pcline{->}(1,1)(7,1)\naput[npos=0.45]{\ensuremath{\scriptstyle{#1}}}
\nbput[npos=0.45]{\scr $\sim$}\endpspicture}}

\newcommand{\mtmap}[1]{{\psset{unit=0.1cm,nodesep=0pt,labelsep=1pt} \pspicture(10,4)
\pcline{->}(1,1)(9,1)\naput[npos=0.45]{\ensuremath{\scriptstyle{#1}}} \endpspicture}}

\newcommand{\mtmapiso}[1]{{\psset{unit=0.1cm,nodesep=0pt,labelsep=1pt} \pspicture(10,4)
\pcline{->}(1,1)(9,1)\naput[npos=0.45]{\ensuremath{\scriptstyle{#1}}}
\nbput[npos=0.45]{\scr $\sim$}\endpspicture}}

\newcommand{\ltramap}[1]{{\psset{unit=0.1cm,nodesep=0pt,labelsep=1pt} \pspicture(15,4)
\pcline{->}(1.5,1.1)(13.5,1.1)\naput{\ensuremath{\scriptstyle{#1}}} \endpspicture}}

\newcommand{\ltmap}{\ltramap}

\newcommand{\stmapsto}
{{\psset{unit=0.1cm,nodesep=0pt} \pspicture(6.4,0) %shorter arrow
\pcline{|->}(0.8,1.1)(5.8,1.1) \endpspicture}}

\newcommand{\tmapsto}{{\psset{unit=0.1cm,nodesep=0pt} \pspicture(8,0) %shorter arrow
\pcline{|->}(1,1.2)(7,1.2) \endpspicture}}

\newcommand{\trta}{{\psset{unit=0.1cm,nodesep=0pt} \pspicture(8,0)
\pcline[doubleline=true,arrowinset=0.7,arrowlength=0.8, arrowsize=3.5pt 1.5]{->}(1,1.1)(7,1.1) \endpspicture}}

\newcommand{\Tra}{\trta}

\newcommand{\Tmap}[1]{{\psset{unit=0.1cm,nodesep=0pt,labelsep=2pt} \pspicture(8,4)
\pcline[doubleline=true, arrowinset=0.7, arrowlength=0.8, arrowsize=3.5pt 1.5]{->}(1,1.1)(7,1.1)\naput[npos=0.45]{\ensuremath{\scriptstyle{#1}}} \endpspicture}}

\newcommand{\msTmap}[1]{{\psset{unit=0.1cm,nodesep=0pt,labelsep=2pt} \pspicture(7,4)
\pcline[doubleline=true, arrowinset=0.7, arrowlength=0.8, arrowsize=3.5pt 1.5]{->}(1,1.1)(6,1.1)\naput[npos=0.45]{\ensuremath{\scriptstyle{#1}}} \endpspicture}}

\newcommand{\myps}{
\begin{small}
\pspicture
}

\newcommand{\emyps}{
\endpspicture
\end{small}
}

\newcommand{\inv}{^{-1}}

\renewcommand{\|}{\hs{0.15}|\hs{0.15}}

\newcommand\Tstrut{\rule{0pt}{2.4ex}}         % = `top' strut
\newcommand{\myhline}{\hline\Tstrut}  %put at beginning of line in array

\newcommand{\vtp}[2]{\ensuremath{\setlength{\arraycolsep}{0.1ex}\hs{-0.1}\begin{array}{c}#1\\ \myhline #2 \end{array}\hs{-0.1}}}

\newcommand{\htp}[2]{\ensuremath{\setlength{\arraycolsep}{0.6ex}\hs{-0.2}\begin{array}{c|c}#1 & #2\end{array}\hs{-0.2}}}

\newcommand{\fourtp}[4]{\ensuremath{\setlength{\arraycolsep}{0.6ex}\hs{-0.2}\begin{array}{c|c}#1 & #2 \\ \myhline #3 & #4\end{array}\hs{-0.2}}}

\newcommand{\eighttp}[8]{\ensuremath{\hs{-0.2}\begin{array}{c|c}#1 & #2 \\ \myhline #3 & #4 \\ \myhline #5 & #6 \\ \myhline #7 & #8 \end{array}\hs{-0.2}}}

\newcommand{\onefouronetp}[6]{\ensuremath{\hs{-0.2}\begin{array}{c}#1 \\
\myhline \hs{-0.4} \fourtp{#2}{#3}{#4}{#5} \hs{-0.4} \\
\myhline #6\end{array}\hs{-0.2}}}

\newcommand{\fourvtp}[4]{\ensuremath{\setlength{\arraycolsep}{0.1ex}\hs{-0.2}\begin{array}{c}#1 \\ \myhline #2 \\ \myhline #3 \\ \myhline #4\end{array}\hs{-0.2}}}

\newcommand{\threevtp}[3]{\ensuremath{\setlength{\arraycolsep}{0.1ex}\begin{array}{c}#1 \\ \myhline #2 \\ \myhline #3 \end{array}}}

\newcommand{\threevtpp}[3]{\ensuremath{\setlength{\arraycolsep}{0.1ex}\begin{array}{c}#1 \\ \hline \\[-10pt] \raisebox{-0.5em}{\ensuremath{#2}} \\[14pt] \hline \\[-10pt] #3 \end{array}}}

\newlength\myfntht

\newcommand{\trimbat}{\cite{che16}}

\newcommand{\chp}{\cite{chp1}}
\newcommand{\wvc}{\cite{cc1}}

\usepackage[linktocpage=true]{hyperref}

\renewcommand{\cat}[1]{\ensuremath{\textrm{\bfseries {\sf {#1}}}}}
\newlength{\currentindent}
\newlength{\mpt}
\makeindex

\makeatletter
\def\footnoterule{\relax%
  \kern 3pt
  \hbox to \columnwidth{\vrule width 0.4\columnwidth height 0.3pt}
  \kern6pt}
\makeatother

\numberwithin{equation}{section}

\theoremstyle{plain}

\newtheorem{theorem}{Theorem}[section]

\newtheorem{proposition}[theorem]{Proposition}

\newtheorem{cor}[theorem]{Corollary}

\theoremstyle{definition}

\newtheorem{definition}[theorem]{Definition}

\newtheorem{example}[theorem]{Example}
\newtheorem{eg}[theorem]{Example}
\newtheorem{examples}[theorem]{Examples}
\newtheorem{nonexample}[theorem]{Non-example}
\newtheorem{remark}[theorem]{Remark}
\newtheorem{remarks}[theorem]{Remarks}
\newtheorem{exercise}[theorem]{Exercise}
\newtheorem{note}[theorem]{Note}
\newtheorem{question}[theorem]{Question}
\newtheorem{questions}[theorem]{Questions}
\newtheorem{algorithm}[theorem]{Algorithm}
\newtheorem{method}[theorem]{Method}
{ \end{sf}\end{framed}\end{minipage}
\end{center}}

\psset{
unit=1mm,
linewidth=0.7pt,
arrowlength=1.1, 
arrowsize=3pt 2,
arrowinset=0.7,
doublesep=1pt,
labelsep=2pt,
nodesep=2pt, 
dash=2pt 2pt}
  
\begin{document}

%%%%%%%%%%%%%%%%%%%%%%%%%%%%%%%%%%%%%%%%%%%%%%%%%  title info

\title{Weak vertical composition II: totalities}

\author{Eugenia Cheng \\  School of the Art Institute of Chicago \\E-mail: info@eugeniacheng.com \\[12pt]
Alexander S. Corner\\
Sheffield Hallam University\\
E-mail: alex.corner@shu.ac.uk
}

%\date{August 2009}

\maketitle

%%%%%%%%%%%%%%%%%%%%%%%%%%%%%%%%%%%%%%%%%%%%%%%%% abstract

\begin{abstract}
We continue our study of semi-strict tricategories in which the only weakness is in vertical composition.  We assemble the doubly-degenerate such tricategories into a 2-category, defining weak functors and transformations. We exhibit a biadjoint biequivalence between this 2-category and the 2-category of braided (weakly) monoidal categories, braided (weakly) monoidal functors, and monoidal transformations.
\end{abstract}

%%%%%%%%%%%%%%%%%%%%%%%%%%%%%%%%%%%%%%%%%%%%%%%%% toc

\setcounter{tocdepth}{2}
\tableofcontents

%%%%%%%%%%%%%%%%%%%%%%%%%%%%%%%%%%%%%%%%%%%%%%%%% 
% Intro
%%%%%%%%%%%%%%%%%%%%%%%%%%%%%%%%%%%%%%%%%%%%%%%%%

\section*{Introduction}
\addcontentsline{toc}{section}{Introduction}

In this paper we continue the study of weak vertical composition begun in \cite{cc1}. We study semi-strict tricategories in which everything is strict except vertical composition, that is, composition along bounding 1-cells. These tricategories can be conveniently constructed as categories enriched in \bicats, the category of bicategories and strict functors, with monoidal structure given by cartesian product. 

In \wvc\ we showed that any \ddbicatscat\ $X$ has an underlying braided monoidal category $\UX$, and that given any \bmc\ $B$ there is a \ddbicatscat\ $\Sigma B$ such that $U \Sigma B$ is braided monoidal equivalent to $B$. This shows that weak vertical composition is ``enough'' to achieve \bmcs\ in the \dd\ case, a typical test case for whether a theory of tricategories is fully weak, and a special case of the study of $k$-degenerate $n$-categories \cite{bd3, cg1, cg3, cg4}.  

That work followed on from \cite{jk1} which proved an analogous result for semistrict tricategories in which everything is strict except weak horizontal units.  However, in both cases the totalities of the structures in question were not studied. 

In this paper we extend the comparison to totalities. That is, we assemble \ddbicatscats\ into a 2-category and exhibit a biequivalence with the 2-category of \bmcs; this extends the object-level comparison of \wvc.

The first task then is to construct a suitable 2-category of \ddbicatscats.  In order to make an equivalence with the 2-category of \bmcs\ we need to consider weak maps, so the first step is to make that definition. Note that as in \cite{cg3,cg4} we do not simply take homomorphisms, transformations and modifications of tricategories as this gives the ``wrong'' structure in the \dd\ case.  One issue is that this would not be expected to form a 2-category; \emph{a priori} tricategories and their higher morphisms assemble into a tetracategory that does not truncate a coherent 2-dimensional structure.  Another issue is that fully weak homomorphisms and transformations of tricategories give too much extraneous structure in the doubly-degenerate case, in the form of distinguished invertible elements arising as higher-dimensional constraint cells relating to degenerate dimensions; the idea is that degenerate dimensions should not give rise to constraint cells, but rather, we should start with some semi-strict versions of weak functors and transformations that are strict with respect to dimensions that are going to become degenerate.

To address both of these issues we follow \cite{cg3, cg4} and use Lack's icons in a higher-dimensional generalisation \cite{lac2} to ensure a coherent 2-category totality and the ``correct'' functors and transformations for the doubly-degenerate structures.  The idea behind icons is that they are ``identity component oplax natural transformations'', but the key is that the identity components are ignored and replaced by an assertion that the source and target homomorphisms agree on 0-cells. This means that the only components are 2-cells and thus icons compose strictly, so bicategories, homomorphisms and icons form a strict 2-category. The process can be iterated \cite{cg4} to give 2-dimensional totalities of weak $n$-categories where restricting to the $(n-1)$-degenerate $n$-categories then results in an appropriate 2-category of categories with extra structure (monoidal, braided monoidal, or symmetric monoidal).  We refer to these higher dimensional iterated versions generally  as ``icon-like'' or ``iconic''; the first step in this work is to make an iconic 2-category of \ddbicatscats.

In \wvc\ we characterised \ddbicatscats\ as a semi-strict form of 2-monoidal category \cite{am1} (that is a category with two monoidal structures and interchange) in which one tensor product is weak but the other tensor product and interchange are strict.  This suggests a characterisation of weak functor as a weak monoidal functor with respect to each monoidal structure, together with some interaction axiom(s).  To put this on a secure footing we will proceed abstractly via monads and distributive laws.  In Section~\ref{one} we construct \bicatscats\ as algebras for a 2-monad on the 2-category \twocat{Cat-2-Gph} of 2-graphs enriched in \Cat\ (equivalently graphs enriched in \cat{Cat-Gph}). This 2-category has iconic 2-cells, so all further constructions are then automatically iconic. The 2-monad in question is a composite of a 2-monad $V$ for vertical composition and a 2-monad $H$ for horizontal composition, composed via a strict distributive law $VH \Tra HV$ coming from strict interchange. We then have a 2-category of algebras for the composite 2-monad, weak maps of algebras, and transformations.  This is the 2-category of \bicatscats\ that we want, but we need to unravel the definitions somewhat in order to compare it with maps of braided monoidal categories.

In Section~\ref{two} we do some preliminary examination of structures arising from a 2-dimensional distributive law of general 2-monads $S$ over $T$. We characterise strict $TS$-algebras via a $T$-algebra and $S$-algebra structure together with an interaction axiom; we characterise a weak map of $TS$-algebras as a weak map with respect to the $T$-algebra structure and to the $S$-algebra structure, together with an interaction axiom.  Transformations are just transformations of the $T$-structure and the $S$-structure, with no further interaction axiom required.

In Section~\ref{three} we unravel those definitions in our case of interest. We re-characterise \ddbicatscats\ as braided monoidal categories in steps:

\begin{enumerate}

\item First we express them as $HV$-algebras.

\item We then re-express them as an $H$-algebra and $V$-algebra structure satisfying an interaction axiom coming from the distributive law (that is, a horizontal and vertical monoidal structure with interchange).

\item Finally we express them as a $V$-algebra (monoidal category) with a braiding coming from a weak Eckmann--Hilton argument.

\end{enumerate}

\noi We re-characterise a weak map of \ddbicatscats\ as a braided monoidal functor via the corresponding steps:

\begin{enumerate}

\item We start with a weak map of $HV$-algebras.

\item We re-express it as a weak map of $H$-algebras and a weak map of $V$-algebras, with an interaction condition relating to the distributive law.

\item Finally we re-express it as just a weak map of $V$-algebras plus a braiding condition.

\end{enumerate}

\noi Our overall aim is to relate (1) to (3), and (2) mediates between those steps for us.  Section~\ref{two} takes us from (1) to (2), and Section~\ref{three} takes us from (2) to (3) in our specific case.

We use a weak \eh\ argument to show that a weak map of \dd\ $HV$-algebras in our case can be characterised as just a weak map of the $V$-structures interacting well with the braiding (which itself comes from an \eh\ argument); conversely such a weak map of \dd\ $V$-algebras can be given the structure of a weak map of $HV$-algebras.  We characterise transformations similarly, and show that a transformation of the $V$-structures is automatically a transformation of the $H$-structures.  

We are then ready to construct a biadjoint biequivalence.  In Section~\ref{five} we extend the assignation $U$ defined in \wvc\ to a 2-functor
\[U\: \twocat{dd\bicats-Cat} \tra \twocat{BrMonCat}.\]
Biessential surjectivity was shown in \wvc; local essential surjectivity on 1-cells follows from Section~\ref{three}, and local full and faithfulness on 2-cells follows from Section~\ref{four}. Then by \cite{gur3} we have the main theorem:

\subsubsection*{Main theorem}

The 2-functor $\twocat{dd\bicats-Cat} \tmap{U} \twocat{BrMonCat}$ is part of a biadjoint biequivalence of 2-categories.

\bs

\noi Note that constructing a pseudo-inverse is non-trivial and we defer it to a sequel.

Finally it is worth noting that weak vertical units are much easier to deal with than weak horizontal units, as it is the weak horizontal 1-cell units that make the weak \eh\ argument tricky in a general tricategory (see for example \cite{cg3}). Using weak vertical units but strict horizontal ones avoids that technical issue.

\subsubsection*{How to read this paper quickly}

\begin{enumerate}

\item Section~\ref{examplebicatscat} gives the 2-category totality of \bicatscats\ with strict maps.

\item Theorem~\ref{weakTSmap} gives the characterisation of weak maps of $TS$-algebras via a $T$-structure and $S$-structure.

\item Section 3 contains the main content of the comparison with braided monoidal categories.

\item The Main Theorem (\ref{maintheorem}) follows immediately.

\end{enumerate}

\subsubsection*{Terminology and notation conventions}

\begin{itemize}

\item We use ``strict'' when axioms hold on the nose and ``weak'' when axioms hold up to specified constraint isomorphisms.

\item Section~\ref{twocattot} is concerned with a careful construction of 2-categorical structures, so in that section we adopt a double-underline notation for 2-category totalities to distinguish them from 1-category totalities.

\item The Appendix contains various proofs using string diagrams.

\end{itemize}

%%%%%%%%%%%%%%%%%%%%%%%%%%%%%%%%%%%%%%%%%%%%%%%%% 
% Section 1
%%%%%%%%%%%%%%%%%%%%%%%%%%%%%%%%%%%%%%%%%%%%%%%%%

\section{\bicats-categories via distributive laws of 2-monads}
\label{one}

% Write \bicats\ for the category of bicategories and strict functors between them.  

% The main difficulty in comparing totalities of weak higher categories is in constructing suitably weak maps. In this work this project is simplified by us only studying \dd\ 3-categories. The work of \cite{cg2} showed that the ``correct'' notion of weak map between \dd\ higher categories is via iterated icons, rather than via the fully weak maps.

In this section we construct a suitable totality of \bicatscats\ for our comparison.  The two subtle features are that it needs to have appropriately weak maps (not fully weak), and that it needs to be 2-dimensional, rather than fully 4-dimensional which is what we might otherwise expect for semi-strict tricategories.  The work of \cite{cg4} showed that the ``correct'' approach for totalities of doubly-degenerate higher categories is via iterated icons rather than via the fully weak maps.  

%aaa

A fully weak map of \ddbicatscats\ would include functoriality constraints for composition of 1-cells.  However, \cite{cg3} showed that this is the ``wrong'' notion for doubly-degenerate situations, because even if there is only one 0-cell and one 1-cell any constraint 2-cell would remain as a distinguished invertible element; that is, when we perform the dimension shift, the old 2-cells become 0-cells of a new lower-dimensional structure (in our case a putative braided monoidal category), and the constraint 2-cells would become distinguished invertible 0-cells, adding unwanted extra structure to our braided monoidal category.

We eliminate that issue by using stricter functors, specifically, functors that are strictly functorial on any dimension of cell that is going to become degenerate. In our case, that means we want functors that are strictly functorial with respect to 1-cell composition, and only have functoriality constraints with respect to 2-cell composition.  This is effected technically by our use of icons, because the existence of an iconic 2-cell constraint encodes an assertion that the source and target agree on objects.

Moreover, we wish to express \ddbicatscats\ via a distributive law between 2-monads, as this gives us a convenient framework for defining weak maps between them; weak maps are difficult to define in generality for higher categories.

Both of these issues are resolved by making the construction via 2-monads on the 2-category \twocat{Cat-Gph-Gph}. The algebras for the resulting 2-monad immediately form a 2-category, with 2-cells that are immediately iconic.  That is the content of this section. 

The 2-dimensional distributive law in question is just a 2-dimensional extension of a 1-dimensional distributive law that is already established \cite{che16,et1} so we will begin with an overview of the 1-dimensional version.

\subsection{Preliminaries on distributive laws}

\Dls\ were introduced by Beck in \cite{bec1} and are a way of combining two algebraic structures in a coherent way. We first recall the definitions and main results that we will be building on.  

\begin{definition}\cite{bec1} Let $S$ and $T$ be monads on a 
category \cl{C}.  A \emph{distributive law of $S$ over $T$} consists of a 
natural transformation $\lambda\: ST \Tra TS$ such that the following 
diagrams commute.

\[\psset{unit=0.1cm,labelsep=2pt,nodesep=3pt,linewidth=0.8pt}
\pspicture(80,19)

%  c 
% a b

\rput(0,0){
\rput(0,0){\rnode{a}{$ST$}}
\rput(20,0){\rnode{b}{$TS$}}
\rput(10,15){\rnode{c}{$T$}}

\ncline{->}{a}{b}\nbput{\scriptsize $\lambda$}
\ncline{->}{c}{a}\nbput{\scriptsize $\eta^ST$}
\ncline{->}{c}{b}\naput{\scriptsize $T\eta^S$}

}

% a b c
% d   e

\rput(40,0){
\rput(0,15){\rnode{a}{$S^2T$}}
\rput(20,15){\rnode{b}{$STS$}}
\rput(40,15){\rnode{c}{$TS^2$}}
\rput(0,0){\rnode{d}{$ST$}}
\rput(40,0){\rnode{e}{$TS$}}

\ncline{->}{a}{b}\naput{\scriptsize $S\lambda$}
\ncline{->}{b}{c}\naput{\scriptsize $\lambda S$}
\ncline{->}{a}{d}\nbput{\scriptsize $\mu^ST$}
\ncline{->}{c}{e}\naput{\scriptsize $T\mu^S$}
\ncline{->}{d}{e}\nbput{\scriptsize $\lambda$}

}

\endpspicture\]

\[\psset{unit=0.1cm,labelsep=2pt,nodesep=3pt,linewidth=0.8pt}
\pspicture(0,-3)(80,18)

%  c 
% a b

\rput(0,0){
\rput(0,0){\rnode{a}{$ST$}}
\rput(20,0){\rnode{b}{$TS$}}
\rput(10,15){\rnode{c}{$S$}}

\ncline{->}{a}{b}\nbput{\scriptsize $\lambda$}
\ncline{->}{c}{a}\nbput{\scriptsize $S\eta^T$}
\ncline{->}{c}{b}\naput{\scriptsize $\eta^T S$}

}

% a b c
% d   e

\rput(40,0){
\rput(0,15){\rnode{a}{$ST^2$}}
\rput(20,15){\rnode{b}{$TST$}}
\rput(40,15){\rnode{c}{$T^2 S$}}
\rput(0,0){\rnode{d}{$ST$}}
\rput(40,0){\rnode{e}{$TS$}}

\ncline{->}{a}{b}\naput{\scriptsize $\lambda T$}
\ncline{->}{b}{c}\naput{\scriptsize $T\lambda$}
\ncline{->}{a}{d}\nbput{\scriptsize $S\mu^T$}
\ncline{->}{c}{e}\naput{\scriptsize $\mu^TS$}
\ncline{->}{d}{e}\nbput{\scriptsize $\lambda$}

}

\endpspicture\]

% \begin{equation}\label{dl1}
% \xymatrix{
% & T \ar[dl]_{\eta^S T} \ar[dr]^{T\eta^S} && S^2 T \ar[r]^<<<<{S\lambda} \ar[d]_{\mu^S T}
% & STS \ar[r]^{\lambda S }
% & TS^2 \ar[d]^{T \mu^S} \\
% ST \ar[rr]_{\lambda} && TS & ST \ar[rr]_{\lambda}
% && TS}
% \end{equation}
% 
% \begin{equation}\label{dl2}
% \xymatrix{
% & S \ar[dl]_{S\eta^T} \ar[dr]^{\eta^T S} && ST^2 \ar[r]^<<<<{\lambda T} \ar[d]_{S\mu^T}
% & TST \ar[r]^{T\lambda}
% & T^2 S \ar[d]^{\mu^T S} \\
% ST \ar[rr]_{\lambda} && TS & ST \ar[rr]_{\lambda}
% && TS }
% \end{equation}

\end{definition}

The main theorem about distributive laws tells us about new monads that arise canonically as a result of the distributive law.  In this work we will mostly be interested in the composite monad.

\begin{theorem}\cite{bec1}\label{barrwells}\label{basicdlthm}
Write $S\cat{-Alg}$ for the category of algebras for $S$, and $\cat{Kl}\hs{0.2}T$ for the Kleisli category of $T$. The following are equivalent:
% \numarabic
\begin{itemize} \setlength{\itemsep}{1pt}
\item A distributive law of $S$ over $T$.
\item A lifting of the monad $T$ to a monad $T'$ on $S\cat{-Alg}$.
\item An extension of the monad $S$ to a monad $\tilde{S}$ on $\cat{Kl}\hs{0.2}T$.
\end{itemize}

It follows that $TS$ canonically acquires the structure of a monad, whose category of algebras coincides with that of the lifted monad $T'$, and whose Kleisli category coincides with that of $\tilde{S}$.
\end{theorem}

% Note that the monad structure on $TS$ is given as follows.
% 
% \begin{itemize}
%  \item Multiplication: $\pspicture(-5,3.7)(20,6)
% \psset{arrowlength=0.8, arrowsize=3.5pt 1.5}
%  \rput(0,5){\rnode{a}{$TSTS$}}
%  \rput(23,5){\rnode{b}{$TTSS$}}
%  \rput(42,5){\rnode{c}{$TS$}}
% \ncline[doubleline=true]{->}{a}{b} \naput{\scr $T\lambda S$}
%  \ncline[doubleline=true]{->}{b}{c} \naput[labelsep=1pt]{\scr $\mu^T \mu^S$}
%  \endpspicture$
% 
%  \item Unit: $\pspicture(-3,4)(20,7)
% \psset{arrowlength=0.8, arrowsize=3.5pt 1.5}
%  \rput(0,5){\rnode{a}{$1$}}
%  \rput(16,5){\rnode{b}{$TS$}}
% %  \rput(42,5){\rnode{c}{$TS$}}
% \ncline[doubleline=true]{->}{a}{b} \naput[labelsep=1pt]{\scr $\eta^T \eta ^S$}
% %  \ncline[doubleline=true]{->}{b}{c} \naput{\scr $\mu^T \mu^S$}
%  \endpspicture$
%  
% \end{itemize}
% 
% 
% 
% \noi The action of the monad $T'$ is given by:
% \[\psset{npos=0.4}
% \pspicture(0,8)(50,28)
% 
% 
% \rput(0,10){\rnode{x}{
% \rput(0,10){\rnode{a}{$SA$}}
% \rput(0,0){\rnode{b}{$A$}}
% \ncline{->}{a}{b} \naput{\scr $\theta$}
% } 
% }
% 
% \rput(30,10){\rnode{y}{
% \rput(0,-3){
% \rput(0,20){\rnode{a}{$STA$}}
% \rput(0,10){\rnode{b}{$TSA$}}
% \rput(0,0){\rnode{c}{$TA$}}
% \ncline{->}{a}{b} \naput{\scr $\lambda_A$}
% \ncline{->}{b}{c} \naput{\scr $T\theta$}
% }
% 
% } 
% }
% 
% 
% \ncline[nodesep=30pt,offset=20pt]{|->}{x}{y}
% 
% 
% \endpspicture
% \]
%  

\noi For the proof we refer the reader to \cite{bec1}.

\begin{eg}\label{ring}

{\bfseries (Rings)}

\begin{itemize}\setlength{\itemsep}{-1pt}
\item $\cl{C} = \cat{Set}$

\item $S = \mbox{free monoid monad}$

\item $T = \mbox{free abelian group monad}$

\item $\lambda = \mbox{the usual distributive law for multiplication and addition}$ e.g. \[(a+b)(c+d) \mapsto ac+bc+ad+bd.\]
\end{itemize}
\noi Then the composite monad $TS$ is the free ring monad.
\end{eg}
% 
% \begin{eg}\label{monoidpointing}
% 
% {\bfseries (Monoids)}
% 
% $\cl{C}=\cat{Set}$
% 
% $S = \mbox{monad for non-unital associative multiplication}$
% 
% $T = \mbox{monad for pointed sets}$, that is,$TA = A \coprod \{\ast\}$
% 
% $\lambda$ ensures that $\ast$ acts as a unit for multiplication:
% %
% \[\begin{array}{ccc}
% S(A \coprod \{\ast\}) & \lra & SA \coprod \{\ast\} \\
% a_1 . \ldots . a_i . \ast . a_{i+1} . \ldots . a_n & \mapsto & a_1 . \ldots . a_n \end{array}\]
% 
% Then the composite monad $TS$ is the free monoid monad.
% \end{eg}

\begin{eg}\label{egtwopointfour}

{\bfseries (2-categories)}

\begin{itemize}\setlength{\itemsep}{-1pt}

\item $\cl{C}=\cat{2-GSet}$, the category of 2-globular sets.

\item $S = $ monad for vertical composition of 2-cells (1- and 0-cells are unchanged)

\item $T=$ monad for horizontal composition of 2-cells and 1-cells (0-cells are unchanged)

\item $\lambda$ is given by the interchange law, for example:
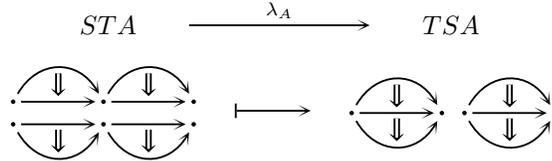
\[
\psset{unit=0.06cm,labelsep=2pt,nodesep=3pt}
\pspicture(0,2)(120,35)

% a1 a2 

\rput(21,30){\rnode{c1}{$STA$}}
\rput(97,30){\rnode{c2}{$TSA$}}

\ncline[nodesep=20pt]{->}{c1}{c2} \naput{\scr $\lambda_A$}

\rput(0,3){

\rput(21,8){\rnode{d1}{}}
\rput(97,8){\rnode{d2}{}}

\ncline[nodesepA=48pt,nodesepB=53pt]{|->}{d1}{d2}

%%%%%%%%%%%%%%%%%%%%%%%%%%%%%%%%%%%%%%%%%%%%%%%%%%%%%%%%%% interchange diag

\rput[l](0,0){

\rput(0,0){

\rput(0,10){\rnode{a1}{$\mysdot$}}  % named node, with something placed there
\rput(20,10){\rnode{a2}{$\mysdot$}}  % named node, with something placed there
\rput(40,10){\rnode{a3}{$\mysdot$}}

{
\rput[c](10,11.3){\psset{unit=1mm,doubleline=true,arrowinset=0.6,arrowlength=0.5
, arrowsize=0.5pt 2.1,nodesep=0pt,labelsep=2pt}
\pcline{->}(0,3)(0,0) \naput{{\scriptsize $$}}}}

{
\rput[c](30,11.3){\psset{unit=1mm,doubleline=true,arrowinset=0.6,arrowlength=0.5
, arrowsize=0.5pt 2.1,nodesep=0pt,labelsep=2pt}
\pcline{->}(0,3)(0,0) \naput{{\scriptsize $$}}}}

\nccurve[angleA=60,angleB=120,ncurv=1]{->}{a1}{a2}\naput{{\scriptsize $$}}
\ncline{->}{a1}{a2}

\nccurve[angleA=60,angleB=120,ncurv=1]{->}{a2}{a3}\naput{{\scriptsize $$}}
\ncline{->}{a2}{a3}

}

\rput[l](0,-5){

\rput(0,10){\rnode{b1}{$\mysdot$}}  % named node, with something placed there
\rput(20,10){\rnode{b2}{$\mysdot$}}  % named node, with something placed there
\rput(40,10){\rnode{b3}{$\mysdot$}}

\ncline{->}{b1}{b2}
\nccurve[angleA=-60,angleB=-120,ncurv=1]{->}{b1}{b2}\nbput{{\scriptsize $$}}

\ncline{->}{b2}{b3}
\nccurve[angleA=-60,angleB=-120,ncurv=1]{->}{b2}{b3}\nbput{{\scriptsize $$}}

{
\rput[c](10,3.8){\psset{unit=1mm,doubleline=true,arrowinset=0.6,arrowlength=0.5
, arrowsize=0.5pt 2.1,nodesep=0pt,labelsep=2pt}
\pcline{->}(0,3)(0,0) \naput{{\scriptsize $$}}}}

{
\rput[c](30,3.8){\psset{unit=1mm,doubleline=true,arrowinset=0.6,arrowlength=0.5
, arrowsize=0.5pt 2.1,nodesep=0pt,labelsep=2pt}
\pcline{->}(0,3)(0,0) \naput{{\scriptsize $$}}}}

}
}

\rput(75,-2.5){

\rput(0,0){
\rput(0,10){\rnode{a1}{$\mysdot$}}  % named node, with something placed there
\rput(20,10){\rnode{a2}{$\mysdot$}}  % named node, with something placed there

\nccurve[angleA=60,angleB=120,ncurv=1]{->}{a1}{a2}\naput{{\scriptsize $$}}
\nccurve[angleA=-60,angleB=-120,ncurv=1]{->}{a1}{a2}\nbput{{\scriptsize $$}}

\ncline{->}{a1}{a2}

\rput[c](10,11.3){
\psset{unit=1mm,doubleline=true,arrowinset=0.6,arrowlength=0.5
, arrowsize=0.5pt 2.1,nodesep=0pt,labelsep=2pt}
\pcline{->}(0,3)(0,0) \naput{{\scriptsize $$}}}

\rput[c](10,3.8){
\psset{unit=1mm,doubleline=true,arrowinset=0.6,arrowlength=0.5
, arrowsize=0.5pt 2.1,nodesep=0pt,labelsep=2pt}
\pcline{->}(0,3)(0,0) \naput{{\scriptsize $$}}}

}

\rput(25,0){
\rput(0,10){\rnode{a1}{$\mysdot$}}  % named node, with something placed there
\rput(20,10){\rnode{a2}{$\mysdot$}}  % named node, with something placed there

\nccurve[angleA=60,angleB=120,ncurv=1]{->}{a1}{a2}\naput{{\scriptsize $$}}
\nccurve[angleA=-60,angleB=-120,ncurv=1]{->}{a1}{a2}\nbput{{\scriptsize $$}}

\ncline{->}{a1}{a2}

\rput[c](10,11.3){
\psset{unit=1mm,doubleline=true,arrowinset=0.6,arrowlength=0.5
, arrowsize=0.5pt 2.1,nodesep=0pt,labelsep=2pt}
\pcline{->}(0,3)(0,0) \naput{{\scriptsize $$}}}

\rput[c](10,3.8){
\psset{unit=1mm,doubleline=true,arrowinset=0.6,arrowlength=0.5
, arrowsize=0.5pt 2.1,nodesep=0pt,labelsep=2pt}
\pcline{->}(0,3)(0,0) \naput{{\scriptsize $$}}}

}

}

}

\endpspicture
\]
\end{itemize}
\end{eg}

This is all we need about \dls\ for this section; in Section~\ref{two} we will revisit \dls\ to examine more closely how to express $TS$-algebra structures via $T$- and $S$-algebra structures interacting.

\subsection{The \dl\ in the 1-dimensional setting}

We will now lay out the 1-dimensional version of the distributive law for operadic weak $n$-categories; see \trimbat\ for full details.  This is a higher level of generality than we need in this work, but we choose to work at this level as it can then also be used for Trimble 3-categories, which we will study in a sequel. So we will parametrise all our composition by the action of operads. When composition is strict it will simply be the terminal operad.

The 1-dimensional \dl\ in question is a generalisation of the one given by Leinster for $n$-categories in \cite{lei8}. It was further studied in an iterated version in \cite{che17} and with operad actions in \cite{che16,et1}. The formulae are somewhat complicated by the operad actions, but the ideas and the proofs are the same.

Our basic setup is an iterative operadic theory of $n$-categories, a generalisation of Trimble's definition \cite{tri1} given in \trimbat. The iteration is by ``operadic enrichment'', where we form categories enriched in $\cV$ with composition parametrised by an operad $P$ in $\cV$.

\begin{definition}
Let $\cV$ be a category with finite products, and $P$ an operad in it.  A (small) $(\cV, P)$-category $A$ is given by:

\begin{itemize}
\item an underlying $\cV$-graph, that is, a set $A_0$ of objects and for all $a, b \in A_0$ a hom-object $A(a,b) \in \cV$.

\item For all $k \geq 0$ and $a_0, \ldots, a_k \in A_0$ a composition morphism
\[P(k) \times A(a_{k-1}, a_k) \times \cdots \times A(a_0, a_1) \tra A(a_0, a_k)\]
\end{itemize}

\noi compatible with the composition and identities of the operad in the usual way (as for algebras). Morphisms are defined in the obvious way, giving a category \cat{$(\cV, P)$-Cat}, which also has products.   We will also write this category \vpcat, to facilitate the notation for iteration.

\end{definition}

\begin{definition} An \emph{iterative operadic theory of $n$-categories} is a series $(\cV_n, P_n)$ where each $\cV_n$ is a category with finite products and each $P_n$ is an operad in $\cV_n$, with

\[\cV_{n+1} = \vnpncat\]

\noi Then $\cV_n$ is the category of $n$-categories and strict functors, according to this theory.
\end{definition}

In this work we will start this induction with $\cV_0 = \set$ and $P_0 = 1$ and thus $\cV_1 = \Cat$. We then have

\begin{itemize}

\item an operad $P_1 \in \Cat$, and $\cV_2 = \voneponecat$
\item an operad $P_2 \in \cV_2$, and $\cV_3 = \vtwoptwocat = \cV_1\cat{-Cat}_{P_1}\cat{-Cat}_{P_2}$

\end{itemize}

One convenient aspect of this form of definition is that many results about strict iterated enrichment generalise straightforwardly, with just some notational complication as we must put operad actions everywhere.  The first step is the free category construction.

In this work we will always be enriching with respect to the cartesian monoidal structure. To make the free category construction we also invoke small coproducts, and the products must distribute over them; we call this an ``infinitely distributive'' category.  Shulman \cite{shul2} works with general monoidal structures and so uses the terminology ``$\otimes$-distributive categories''. The following result is from \cite{che16,et1}, and gives us the general free $(\cV,P)$-category monad we need. Elsewhere this monad is written as $\fcvpp$, but we will write it as $\fcvp$ to streamline the subscripts.

% ***check infinitely distributive and give definition: finite products and infinite coproducts, products distribute over coproducts, Shulman calls is $\otimes$-distributive: small sums, preserved on both sides by $\otimes$**

\begin{proposition}\label{prop1}
Given any infinitely distributive category $\cV$ and operad $P$ in it, \vpcat\ is monadic over \cat{$\cV$-Gph} via a monad $\fcvp$, and the category \vpcat\ is in turn infinitely distributive.
\end{proposition}

So we know that $\cV_3$ is monadic over \cat{$\cV_2$-Gph}; our aim now is to show that $\cV_3$ is monadic over \cat{$\cV_1$-Gph-Gph}, and construct the monad via a distributive law.  The idea is that if we think of $\cV_3$ as $\cV_1\cat{-Cat}_{P_1}\cat{-Cat}_{P_2}$ we can do a free category construction on the middle component (with the $P_1$ subscript) and on the last component (with the $P_2$ subscript) separately.  The middle one constructs 1-composites and the last one constructs $0$-composites.  The following two forgetful functors produce our two monads for the distributive law.

% ***insert explanation about * construction**
% ***insert diagram showing two categories and forgetful functors**
% 

\[\ps(-20,0)(20,20)

%  tl   tr
%     b
\rput(0,0){\rnode{b}{$\cat{$\cV_1$-Gph-Gph}$}}  
\rput(-20,16){\rnode{tl}{\cat{$\cat{$\cV_1$-Cat}_{P_1}$-Gph}}}  
% \rput(20,20){\rnode{tr}{$\cat{$\V_1$-Gph-Cat}_{\UP_1}$}}  
  \rput(20,16){\rnode{tr}{$\cat{$\cV_1$-Gph-Cat}_{\UP_2}$}}

\ncline{->}{tl}{b} \naput{{\scriptsize $$}}
\ncline{->}{tr}{b} \naput{{\scriptsize $$}}

\eps\]

In order to construct the monad for 1-composition we use the following 2-functor.

\begin{definition}\label{astone} The assignation $\cV \tmapsto \cV\cat{-Gph}$ extends to a 2-functor 
\[\twocat{Cat}\tra \twocat{Cat}\]  

\noi We write the action on functors as $(\ )_\ast$, so given a functor $\cV \tmap{F} \cW$ we induce a functor $\cV\cat{-Gph} \tmap{F_\ast} \cW\cat{-Gph}$ in the obvious way; likewise for natural transformations. Thus a monad $T$ on $\cV$ induces a monad $T_\ast$ on $\cV\cat{-Gph}$.
\end{definition}

\begin{remark}
Note that $T_\ast$ acts on a $\cV$-graph by leaving the objects unchanged, and just acting as $T$ on the homs.
\end{remark}

% \[\pspicture(-20,0)(20,45)
% 
% 
% \rput(-20,40){\rnode{a1}{$\twocat{Cat}$}}  
% \rput(20,40){\rnode{a2}{$\twocat{Cat}$}}  
% 
% \rput(-20,30){\rnode{b1}{$\cV$}}  
% \rput(20,30){\rnode{b2}{$\vgph$}}  
% 
% \rput(-20,20){\rnode{c1}{$\cV$}}  
% \rput(-20,10){\rnode{d1}{$\cW$}}  
% 
% \rput(20,20){\rnode{e1}{$\vgph$}}  
% \rput(20,10){\rnode{f1}{$\cat{$\cW$-Gph}$}}  
% 
% \ncline[nodesep=6pt]{->}{a1}{a2}
% \ncline[nodesep=16pt]{|->}{b1}{b2}
% \ncline{->}{c1}{d1}\naput{\scr $F$}
% \ncline{->}{e1}{f1}\naput{\scr $F_*$}
% 
% \ncline[nodesep=16pt,offset=-10pt]{|->}{c1}{e1}
% 
% % \rput(-20,){\rnode{}{$$}}  
% % \rput(20,){\rnode{}{$$}}  
% 
% 
% % \rput(,){\rnode{}{$$}}  
% 
% 
% 
% \eps\]

\begin{definition}[Monad for vertical composition]\label{verticalone}

We know that $\cV_2 = \cV_1\cat{-Cat}_{P_1}$ is monadic over \VoneGph\ with monad \fcvonepone. Write $T_1 = (\fcvonepone)_*$ for the induced monad on \vonegphgph, with $\cat{$T_1$-Alg} \cong \vtwogph$.

\end{definition}

\begin{remark}
The monad $T_1$ acts on a $\cV_1$-graph-graph by leaving the 0-cells unchanged and then forming the free $(\cV_1,P_1)$-category on each hom $\cV_1$-graph, that is, it makes vertical composition freely.
\end{remark}

Next we construct the monad for horizontal composition. This comes from the monad for free $(\cV_2, P_2)$-categories, but that monad already encodes interaction with vertical composition. So we invoke the forgetful functor $\cV_2 \tra \cV_1\cat{-Gph}$ to forget that part, so that the monads for horizontal and vertical composition can be applied separately; the interaction will be encoded in the distributive law.

\begin{definition}[Monad for horizontal composition]\label{verticaltwo}
We know that \vthree\ is monadic over $\vtwogph = \voneponecatgph$. We also have a forgetful functor
\[\vtwo = \voneponecat \mtmap{U} \vonegph.\]
As $U$ preserves products, the operad $P_2 \in \vtwo$ becomes an operad 
$\UP_2 \in \vonegph$, and we can form $(\vonegph, \UP_2)$-categories.  Then by Proposition~\ref{prop1} we know that $\cat{\vonegph-Cat}_{\UP_2}$ is monadic over \vonegphgph\ with monad $\fc_{\UP_2}$. Call this monad $T_0$.

\end{definition}

Note that as we are enriching with respect to strict functors, we have strict interchange at all levels, parametrised by actions of the operads in question. In this case, vertical composition is parametrised by $P_1$ and horizontal composition is parametrised by $P_2$.  Note that as $P_2$ is an operad in \voneponecat\ there is also an action of $P_1$ on $P_2$, and this has to be invoked when performing parametrised interchange (see \trimbat).

\begin{proposition}\label{prop4}

There is a \dl\ $T_1T_0 \Tra T_0T_1$ with $\cat{$T_0T_1$-Alg} \cong \vthree$ and the lifted monad $\hat{T_0}$ on \tonealg\ is \fcvtwoptwo.

\end{proposition}

\begin{remark}

The monad for 3-categories can be decomposed further by also invoking the monad for composition along bounding 2-cells. However, as we do not need to weaken this composition, we do not need to reference this monad.

\end{remark}

Note that for this work we will be taking $P_1$ to be the operad for bicategories, that is, the free operad generated by one binary and one nullary operation.  We will take $P_2 = 1$ (the terminal operad in \cat{Cat}) so that horizontal composition is strict.

\subsection{The 2-category totalities}
\label{twocattot}

We now construct the iconic 2-category totalities we will be working with. Icons were introduced by Lack \cite{lac2}, giving a convenient 2-dimensional totality of bicategories, and were iterated in \cite{gg1} to give a convenient 2-dimensional totality of tricategories.  Regarding bicategories as weak \cat{Cat}-categories, the definition of icon can be generalised to weak $K$-categories for bicategories $K$ other than \cat{Cat}, yielding an iconic bicategory of weak $K$-categories. Under some mild conditions the generalised construction can then be iterated \cite{cg4}.  For this work we generalise in a slightly different direction, as we need to add in the action of an operad but we do not need $K$ to be weak.  As a result, we have strict 2-category totalities at every stage of the iteration, which we see as a great advantage of this operadic approach.

We will begin with the basic definitions, before introducing the operad actions.

\begin{definition}

Let $K$ be a 2-category.  Then a \emph{$K$-graph} $X$ is given by 
\begin{itemize}
\item a set $X_0$ of 0-cells, and
\item for all $a,b \in X_0$ a hom-object $X(a,b)$ which is a 0-cell of $K$.
\end{itemize}

\noi A \emph{morphism $X\tmap{F} Y$ of $K$-graphs} is given by
\begin{itemize}
\item a function $X_0 \tmap{F_0} Y_0$, and
\item for all $a,b \in X_0$ a 1-cell $X(a,b) \mtmap{F_{ab}} Y(Fa, Fb) \in K$.

\end{itemize}

\noi Given morphisms $F,G\: X \tra Y$ of $K$-graphs, a transformation $\alpha\: F \Tra G$ can exist only when $F$ and $G$ agree on 0-cells. In that case such a transformation is given by
\begin{itemize}

\item for all $a, b \in X_0$ a 2-cell in $K$

\[
\psset{unit=0.09cm,labelsep=2pt,nodesep=2pt}
\pspicture(-14,-14)(14,14)

\rput(-20,0){\rnode{a1}{$X(a,b)$}}  
\rput(20,4){\rnode{a2}{$Y(Fa,Fb)$}}  
\rput(20,0){\rotatebox{90}{$=$}}
\rput(20,-4){\rnode{a3}{$Y(Ga, Gb)$}}  

\ncarc[arcangle=45]{->}{a1}{a2}\naput[npos=0.55]{{\scriptsize $F$}}
\ncarc[arcangle=-45]{->}{a1}{a3}\nbput[npos=0.55]{{\scriptsize $G$}}

{
\rput[c](0,0){\psset{unit=1mm,doubleline=true,arrowinset=0.6,arrowlength=0.5,arrowsize=0.5pt 2.1,nodesep=0pt,labelsep=2pt}
\pcline{->}(0,2.5)(0,-2.5) \nbput{{\scriptsize $\alpha_{ab}$}}}}

\endpspicture
\]

\end{itemize}

\noi $K$-graphs, their morphisms and their 2-morphisms assemble into a 2-category \twocat{$K$-Gph}, where the composition and identities are inherited from $K$.

\end{definition}

\begin{remark}

Note that if $K$ has underlying 1-category $K_1$, then a $K$-graph is no different from a $K_1$-graph; likewise morphisms between them.  The only difference is that the 2-cells of $K$ enable us to put a 2-dimensional structure on the totality.  That is, the  underlying 1-category of \twocat{$K$-Gph} is the same as the 1-category \cat{$K_1$-Gph}.

\end{remark}

In what follows $K$ could be a tensor distributive monoidal 2-category as in \cite{shul2}.   However in all our cases we will use products, so we need $K$ to have small coproducts, and finite products that distribute over them (using the strictest 2-dimensional version of (co)limits); we will call such a 2-category ``infinitely distributive''.  The definition of $K$-icon is given in \cite{cg4} in a fully weak setting using monoidal bicategories; our setting is simplified by us using strict enrichment in a strict 2-category, and strict functors.

As with $K$-graphs, the concepts of $K$-category and $K$-functor are no different from $K_1$-category and $K_1$-functor; the difference is that we now have 2-dimensional structure with which to assemble $K$-categories into a 2-category.

\begin{definition} Let $K$ be a 2-category with products.  Then a category enriched in $K$ is just a category enriched in its underlying 1-category $K_1$, and a $K$-functor is just a $K_1$-functor.
\end{definition}

\begin{definition}[$K$-icons]

Let $X,Y$ be $K$-categories, and let $F,G \: X \tra Y$ be strict $K$-functors such that $Fa=Ga$ for all objects $a \in X$.  A \demph{$K$-icon}
\[
\psset{unit=0.1cm,labelsep=2pt,nodesep=2pt}
\pspicture(20,20)

\rput(0,10){\rnode{a1}{$X$}}  % named node, with something placed there
\rput(20,10){\rnode{a2}{$Y$}}  % named node, with something placed there

\ncarc[arcangle=45]{->}{a1}{a2}\naput{{\scriptsize $F$}}
\ncarc[arcangle=-45]{->}{a1}{a2}\nbput{{\scriptsize $G$}}

\pcline[linewidth=0.6pt,doubleline=true,arrowinset=0.6,arrowlength=0.8,arrowsize=0.5pt 2.1]{->}(10,13)(10,7)  \naput{{\scriptsize $\alpha$}}

% \ncarc[arcangle=45]{->}{a1}{a2}\naput{{\scriptsize $$}}
% \ncarc[arcangle=-45]{->}{a1}{a2}\nbput{{\scriptsize $$}}

\endpspicture
\]
is given by, for all pairs of objects $a,b \in X$ a 2-cell

\[
\psset{unit=0.1cm,labelsep=2pt,nodesep=2pt}
\pspicture(0,-4)(30,24)

\rput(0,10){\rnode{a1}{$X(a,b)$}}  % named node, with something placed there
\rput(30,13){\rnode{a2}{$Y(Fa,Fb)$}}  % named node, with something placed there
\rput(31,10){\rotatebox{90}{$=$}}
\rput(30,7){\rnode{a3}{$Y(Ga,Gb)$}}  % named node, with something placed there

\ncarc[arcangle=45]{->}{a1}{a2}\naput[npos=0.6]{{\scriptsize $F$}}
\ncarc[arcangle=-45]{->}{a1}{a3}\nbput[npos=0.6]{{\scriptsize $G$}}

{
\rput[c](17,8.5){\psset{unit=1mm,doubleline=true,arrowinset=0.6,arrowlength=0.5,arrowsize=0.5pt 2.1,nodesep=0pt,labelsep=2pt}
\pcline{->}(0,3)(0,0) \nbput{{\scriptsize $\alpha_{ab}$}}}}

% \ncarc[arcangle=45]{->}{a1}{a2}\naput{{\scriptsize $$}}
% \ncarc[arcangle=-45]{->}{a1}{a2}\nbput{{\scriptsize $$}}

\endpspicture
\]
satisfying the following axioms. Note that the following are all diagrams in $K$. We write $(a,b)$ for the hom-object of the appropriate $K$-category, and omit $\times$ signs. Here $m$ and $m'$ represent composition in the appropriate $K$-categories.

\begin{itemize}
 \item Composition:

\[\begin{small}
\psset{unit=0.09cm,labelsep=2pt,nodesep=2pt}
\pspicture(0,-10)(120,35)

% a1 a2
%    a3
%
% b1 b2

\rput(0,30){\rnode{a1}{$(b,c)(a,b)$}}  % named node, with something placed there
\rput(30,33){\rnode{a2}{\makebox[4em][l]{$(Fb,Fc)(Fa,Fb)$}}}  % named node, with something placed there
\rput(38,30){\rotatebox{90}{$=$}}
\rput(30,27){\rnode{a3}{\makebox[4em][l]{$(Gb,Gc)(Ga,Gb)$}}}  % named node, with something placed there

\rput(0,3){\rnode{b1}{$(a,c)$}}  
\rput(30,3){\rnode{b2}{\makebox[3em][l]{$(Ga,Gc)$}}}

\ncarc[arcangle=45]{->}{a1}{a2}\naput[npos=0.56]{{\scriptsize $FF$}}
\ncarc[arcangle=-45]{->}{a1}{a3}\nbput[npos=0.56]{{\scriptsize $GG$}}

\ncline{->}{a1}{b1} \nbput{{\scriptsize $m$}}
\ncline{->}{a3}{b2} \naput{{\scriptsize $m'$}}

\ncarc[arcangle=-45]{->}{b1}{b2}\nbput[npos=0.5]{{\scriptsize $G$}}

{
\rput[c](17,28.5){\psset{unit=1mm,doubleline=true,arrowinset=0.6,arrowlength=0.5,arrowsize=0.5pt 2.1,nodesep=0pt,labelsep=2pt}
\pcline{->}(0,3)(0,0) \nbput{{\scriptsize $\alpha\alpha$}}}}

{
\rput[c](15,9){\psset{unit=1mm,doubleline=true,arrowinset=0.6,arrowlength=0.5,arrowsize=0.5pt 2.1,nodesep=0pt,labelsep=2pt}
\pcline{-}(-1.5,-1.5)(1.5,1.5) \nbput{{\scriptsize $$}}}}

% \ncarc[arcangle=45]{->}{a1}{a2}\naput{{\scriptsize $$}}
% \ncarc[arcangle=-45]{->}{a1}{a2}\nbput{{\scriptsize $$}}

\rput(53,15){$=$}

\rput(70,0){

% a1 a2
%    
%
% b1 b2
%    b3

\rput(0,30){\rnode{a1}{$(b,c)(a,b)$}}  % named node, with something placed there
\rput(30,30){\rnode{a2}{\makebox[6em][l]{$(Fb,Fc)(Fa,Fb)$}}}  % named node, with something placed there

\rput(0,3){\rnode{b1}{$(a,c)$}}  
\rput(30,6){\rnode{b2}{\makebox[3em][l]{$(Fa,Fc)$}}}  

\rput(32,3){\rotatebox{90}{$=$}}
\rput(30,0){\rnode{b3}{\makebox[3em][l]{$(Ga,Gc)$}}}  % named node, with something placed there

\ncarc[arcangle=45]{->}{a1}{a2}\naput[npos=0.56]{{\scriptsize $FF$}}
%\ncarc[arcangle=-45]{->}{a1}{a3}\nbput[npos=0.56]{{\scriptsize $GG$}}

\ncline{->}{a1}{b1} \nbput{{\scriptsize $m$}}
\ncline{->}{a2}{b2} \naput{{\scriptsize $m'$}}

\ncarc[arcangle=45]{->}{b1}{b2}\naput[npos=0.56]{{\scriptsize $F$}}
\ncarc[arcangle=-45]{->}{b1}{b3}\nbput[npos=0.56]{{\scriptsize $G$}}

{
\rput[c](15,21){\psset{unit=1mm,doubleline=true,arrowinset=0.6,arrowlength=0.5,arrowsize=0.5pt 2.1,nodesep=0pt,labelsep=2pt}
\pcline{-}(-1.5,-1.5)(1.5,1.5) \nbput{{\scriptsize $$}}}}

{
\rput[c](17,0.5){\psset{unit=1mm,doubleline=true,arrowinset=0.6,arrowlength=0.5,arrowsize=0.5pt 2.1,nodesep=0pt,labelsep=2pt}
\pcline{->}(0,3)(0,0) \nbput{{\scriptsize $\alpha$}}}}

% \ncarc[arcangle=45]{->}{a1}{a2}\naput{{\scriptsize $$}}
% \ncarc[arcangle=-45]{->}{a1}{a2}\nbput{{\scriptsize $$}}

}

\eps
\end{small}
\]

\item Unit:

\[\begin{small}
\psset{unit=0.13cm,labelsep=2pt,nodesep=2pt}
\pspicture(15,0)(75,20)

%     a2
% a1
%     a3

\rput(2,10){\rnode{a1}{$\1$}}
\rput(25,18){\rnode{a2}{$(a,a)$}}
\rput(25,2){\rnode{a3}{$(Fa,Fa)=(Ga,Ga)$}}

\ncline{->}{a1}{a2} \naput{{\scriptsize $I$}}
\ncline{->}{a1}{a3} \nbput{{\scriptsize $I'$}}

\ncarc[arcangle=45]{->}{a2}{a3}\naput[npos=0.5]{{\scriptsize $G$}}
\ncarc[arcangle=-45]{->}{a2}{a3}\nbput[npos=0.5]{{\scriptsize $F$}}

{
\rput[c](14,9){\psset{unit=1mm,doubleline=true,arrowinset=0.6,arrowlength=0.5,arrowsize=0.5pt 2.1,nodesep=0pt,labelsep=-2pt}
\pcline{-}(0,0)(3,3) \naput{{\scriptsize $$}}}}

{
\rput[c](24,10){\psset{unit=1mm,doubleline=true,arrowinset=0.6,arrowlength=0.5,arrowsize=0.5pt 2.1,nodesep=0pt,labelsep=2pt}
\pcline{->}(0,0)(4,0) \nbput{{\scriptsize $\alpha$}}}}

\rput(40,10){$=$}

\rput(45,0){

%     a2
% a1
%     a3

\rput(2,10){\rnode{a1}{$\1$}}
\rput(25,18){\rnode{a2}{$(a,a)$}}
\rput(25,2){\rnode{a3}{$(Fa,Fa)=(Ga,Ga)$}}

\ncline{->}{a1}{a2} \naput{{\scriptsize $I$}}
\ncline{->}{a1}{a3} \nbput{{\scriptsize $I'$}}

\ncarc[arcangle=45]{->}{a2}{a3}\naput[npos=0.5]{{\scriptsize $G$}}

{
\rput[c](17,9){\psset{unit=1mm,doubleline=true,arrowinset=0.6,arrowlength=0.5,arrowsize=0.5pt 2.1,nodesep=0pt,labelsep=-2pt}
\pcline{-}(0,0)(3,3) \naput{{\scriptsize $$}}}}

% {
% \rput[c](24,10){\psset{unit=1mm,doubleline=true,arrowinset=0.6,arrowlength=0.5,arrowsize=0.5pt 2.1,nodesep=0pt,labelsep=2pt}
% \pcline{->}(0,0)(4,0) \nbput{{\scriptsize $\alpha$}}}}

}

\endpspicture
\end{small}\]

\end{itemize}

Write \kcat\ for the 2-category of $K$-categories, $K$-functors, and $K$-icons.

\end{definition}

% \begin{definition}[Icons]
% 
% Let $K$ be an infinitely distributive 2-category with products. ***really?**
% 
% ***somehow this definition is nowhere yet**
% 
% 
% 
% \end{definition}

\begin{remark}

This is more usually written \kicon\ to distinguish it from the full 3-dimensional totality of $K$-categories, but we will write it as \kcat\ as there will be no ambiguity, and we want to emphasise the $K$-categories as they are our main objects of study.

\end{remark}

We now need to weaken all this by the action of an operad.  As we are working somewhat strictly we do not need to invoke any 2-categorical structure for the operad. So an operad in a 2-category $K$ is just an operad in the underlying 1-category $K_1$.

\begin{definition}

Let $K$ be a 2-category with products, and $P$ an operad in $K$.  We define a 2-category \kpcat\ as follows.

\begin{itemize}

\item 0-cells are $(K,P)$-categories i.e. $(K_1, P)$-categories

\item 1-cells are (strict) $(K,P)$-functors

\item 2-cells are $(K,P)$-icons.

\end{itemize}

Note that the underlying data for a $(K,P)$-icon is the same as for a $K$-icon, but the composition axiom is now parametrised by $P$ as follows (where $m$ and $m'$ now represent parametrised composition in the appropriate $(K,P)$-category): 

\[\begin{small}
\psset{unit=0.09cm,labelsep=2pt,nodesep=2pt}
\pspicture(0,-10)(120,95)

% a1 a2
%    a3
%
% b1 b2

\rput(15,55){

\rput(0,30){\rnode{a1}{$P(k)(x_{k-1},x_k)\cdots(x_0,x_1)$}}  

\rput(60,0){
\rput(0,33){\rnode{a2}{\makebox[8em][l]{$P(k)(Fx_{k-1},Fx_k)\cdots(Fx_0,Fx_1)$}}}  
\rput(6,29.5){\rotatebox{90}{$=$}}
\rput(0,27){\rnode{a3}{\makebox[8em][l]{$P(k)(Gx_{k-1},Gx_k)\cdots(Gx_0,Gx_1)$}}}

\rput(0,3){\rnode{b2}{\makebox[3em][l]{$(Gx_0,Gx_k)$}}}

}

\rput(0,3){\rnode{b1}{$(x_0,x_k)$}}

\ncarc[arcangle=40,ncurvA=0.6,ncurvB=0.4]{->}{a1}{a2}\naput[npos=0.5]{{\scriptsize $1\times F^k$}}
\ncarc[arcangle=-40,ncurvA=0.6,ncurvB=0.4]{->}{a1}{a3}\nbput[npos=0.5]{{\scriptsize $1\times G^k$}}

\ncline{->}{a1}{b1} \nbput{{\scriptsize $m$}}
\ncline{->}{a3}{b2} \naput{{\scriptsize $m'$}}

\ncarc[arcangle=-40,ncurvA=0.6,ncurvB=0.6]{->}{b1}{b2}\nbput[npos=0.5]{{\scriptsize $G$}}

{
\rput[c](30,30){\psset{unit=1mm,doubleline=true,arrowinset=0.6,arrowlength=0.5,arrowsize=0.5pt 2.1,nodesep=0pt,labelsep=2pt}
\pcline{->}(0,2)(0,-2) \naput{{\scriptsize $1\times \alpha^k$}}}}

{
\rput[c](30,3){\psset{unit=1mm,doubleline=true,arrowinset=0.6,arrowlength=0.5,arrowsize=0.5pt 2.1,nodesep=0pt,labelsep=2pt}
\pcline{-}(-1.5,-1.5)(1.5,1.5) \nbput{{\scriptsize $$}}}}

% \ncarc[arcangle=45]{->}{a1}{a2}\naput{{\scriptsize $$}}
% \ncarc[arcangle=-45]{->}{a1}{a2}\nbput{{\scriptsize $$}}

\rput(83,15){$=$}

}

%%%%%%%%%%%%%%%%%%%%%%%%%%%%%%%%%%%%%%%%%%%% bottom part

\rput(32,0){

% a1 a2
%    
%
% b1 b2
%    b3

\rput(0,30){\rnode{a1}{$P(k)(x_{k-1},x_k)\cdots(x_0,x_1)$}}  

\rput(0,3){\rnode{b1}{$(x_0,x_k)$}}

\rput(60,0){

\rput(0,30){\rnode{a2}{\makebox[13em][l]{$P(k)(Fx_{k-1},Fx_k)\cdots(Fx_0,Fx_1)$}}}

\rput(0,6){\rnode{b2}{\makebox[3em][l]{$(Fx_0,Fx_k)$}}}  

\rput(3,2.5){\rotatebox{90}{$=$}}
\rput(0,0){\rnode{b3}{\makebox[3em][l]{$(Gx_0,Gx_k)$}}}  % named node, with something placed there

}

\ncarc[arcangle=40,ncurvA=0.6,ncurvB=0.4]{->}{a1}{a2}\naput[npos=0.5]{{\scriptsize $1\times F^k$}}
%\ncarc[arcangle=-45]{->}{a1}{a3}\nbput[npos=0.56]{{\scriptsize $GG$}}

\ncline{->}{a1}{b1} \nbput{{\scriptsize $m$}}
\ncline{->}{a2}{b2} \naput{{\scriptsize $m'$}}

\ncarc[arcangle=40,ncurvA=0.6,ncurvB=0.4]{->}{b1}{b2}\naput[npos=0.5]{{\scriptsize $F$}}
\ncarc[arcangle=-40,ncurvA=0.6,ncurvB=0.4]{->}{b1}{b3}\nbput[npos=0.5]{{\scriptsize $G$}}

{
\rput[c](30,26){\psset{unit=1mm,doubleline=true,arrowinset=0.6,arrowlength=0.5,arrowsize=0.5pt 2.1,nodesep=0pt,labelsep=2pt}
\pcline{-}(-1.5,-1.5)(1.5,1.5) \nbput{{\scriptsize $$}}}}

{
\rput[c](30,0.5){\psset{unit=1mm,doubleline=true,arrowinset=0.6,arrowlength=0.5,arrowsize=0.5pt 2.1,nodesep=0pt,labelsep=2pt}
\pcline{->}(0,3)(0,0) \nbput{{\scriptsize $\alpha$}}}}

% \ncarc[arcangle=45]{->}{a1}{a2}\naput{{\scriptsize $$}}
% \ncarc[arcangle=-45]{->}{a1}{a2}\nbput{{\scriptsize $$}}

}

\eps
\end{small}
\]

Horizontal and vertical composition are built in the obvious way from horizontal and vertical composition of 2-cells in $K$, and it is straightforward to check that these composites satisfy the operadic composition condition.

The 2-category axioms for \kpcat\ are inherited from $K$.

\end{definition}

Note that from this point on, we will be working entirely 2-categorically so we will drop the double underline notation on 2-categories as there should be no ambiguity. We can now iterate all the constructions so we have 2-categories as follows:

\begin{itemize}

\item $\vone = \twocatt{Cat}$ (here the 2-cells are ordinary natural transformations)

% \item $\vtwo = \twocat{\voneponecat}$ with iconic 2-cells

\item $\vtwo = \twocatt{$\cV_1$-Cat}_{P_1}$ with iconic 2-cells

\item $\vthree = \twocatt{\vtwoptwocat}$ with iconic 2-cells

\end{itemize}

\subsection{Extending the \dl\ to the 2-category totalities}

We will now extend the 1-dimensional \dl\ to the 2-category totalities. Parts of this have to some extent been done already: Lack and Paoli \cite{lp1} study the \dl\ at the 2-monad level for bicategories, and Shulman \cite{shul2} studies it in some generality for both strict and weak algebras.  However we need a different combination of strictness and weakness here, as we need our monads to be for weak vertical composition but strict horizontal composition.  Bicategories are often treated as weak algebras for the strict 2-category monad, for example in \cite{shul2}, but we want all our algebras to be strict. That is, we express bicategories as strict algebras for a monad for bicategories, rather than as weak algebras for a monad for 2-categories.

% : Lack and Paoli \cite{lp1} for \twocat{\bicats} being 2-monadic over \twocat{Cat-Gph}, and \cite{shul2} for the case where all operads are 1.  We bring this all together to give a full account in one place.  

% 
% 

As we already know the results are true for the underlying 1-monads on the underlying 1-categories, we just have to check the information at the 2-dimensional level. This is mostly a question of working out what needs to be checked; the calculations are then straightforward.  This could also be seen as a generalisation of \cite{shul2} to include operad actions, but we have decided to include the direct calculations, partly as they are not hard and partly as we found them illuminating.

% % 
% %
% The 1-dimensional version is already established, so for this work we just have to make everything 2-dimensional. That is, we want the results to be true for some 2-monads on 2-categories, but we already know the results are true for the underlying 1-monads on the underlying 1-categories, so then we just have to check the information at the 2-dimensional level.

% This is all routine but worth elucidating.  

% Note that we say a 2-category is infinitely distributive if it is infinitely distributive with respect to 2-products and 2-coproducts, that is, the strictest form of 2-dimensional (co)limits..

\begin{proposition}[2-categorical version of Prop~\ref{prop1}]\label{propfctwocat}

Given any infinitely distributive 2-category $K$ and an operad $P$ in it, the 2-category \kpcato\ is 2-monadic over \kgpho\ via a 2-monad \fckp, whose underlying 1-monad is the one given in Proposition~\ref{prop1}. Furthermore, \kpcato\ is also infinitely distributive.

\end{proposition}

\begin{proof}
Write $T$ for \fckp. We need to extend the 1-monad as follows

\begin{enumerate}

\item make the underlying functor into a 2-functor, i.e., give its action on iconic 2-cells, and

\item make $\eta$ and $\mu$ into 2-transformations, i.e., check the cylinder conditions.

\end{enumerate}

For (1) we know the action of $T\: \kgpho \tra \kgpho$ on 0-cells and 1-cells. On 0-cells, given a $K$-graph $A$, $TA$ has the same objects, and
\[TA(a,a') = \coprod_{k, a=a_0, a_1, \ldots, a_k = a'} P(k) \times A(a_{k-1}, a_k) \times \cdots \times A(a_0, a_1)\]
That is, essentially, composable strings equipped with a reparametrising operation. The action of $T$ on 1-cells is pointwise: given a morphism $A \tmap{F} B$ of $K$-graphs, the morphism $TA \tmap{TF} TB$ has the same action on objects, and on homs

% \[ P(k) \times A(a_{k-1}, a_k) \times \cdots \times A(a_0, a_1) \tmap{1 \times F^k}
% P(k) \times B(Fa_{k-1}, Fa_k) \times \cdots \times B(Fa_0, Fa_1)
% \]

\[\ps(-10,-9)(10,9)

\rput[B](0,8){\Rnode{a}{$P(k) \times A(a_{k-1}, a_k) \times \cdots \times A(a_0, a_1)$}}  
\rput[B](0,-8){\Rnode{b}{$P(k) \times B(Fa_{k-1}, Fa_k) \times \cdots \times B(Fa_0, Fa_1)$}}  
\ncline{->}{a}{b}\naput{\scr $1 \times F^k$}

\eps\]

We can now define the action of $T$ on 2-cells to be pointwise as well.  So consider morphisms of $K$-graphs
\[\psset{labelsep=2pt,nodesep=2pt}
\pspicture(0,-2)(15,2)
\rput[B](0,0){\Rnode{l}{$X$}}
\rput[B](15,0){\Rnode{r}{$Y$}}
\psset{nodesep=3pt,arrows=->}
\ncline[offset=3pt]{l}{r}\naput{{\scriptsize $F$}}
\ncline[offset=-3pt]{l}{r}\nbput{{\scriptsize $G$}}
\endpspicture
\]
agreeing on objects, and a 2-cell $\alpha$ given by $\forall a, b$

\[
\psset{unit=0.09cm,labelsep=2pt,nodesep=2pt}
\pspicture(-14,-14)(14,14)

\rput(-20,0){\rnode{a1}{$X(a,b)$}}  
\rput(20,4){\rnode{a2}{$Y(Fa,Fb)$}}  
\rput(20,0){\rotatebox{90}{$=$}}
\rput(20,-4){\rnode{a3}{$Y(Ga, Gb)$}}  

\ncarc[arcangle=45]{->}{a1}{a2}\naput[npos=0.55]{{\scriptsize $F$}}
\ncarc[arcangle=-45]{->}{a1}{a3}\nbput[npos=0.55]{{\scriptsize $G$}}

{
\rput[c](0,0){\psset{unit=1mm,doubleline=true,arrowinset=0.6,arrowlength=0.5,arrowsize=0.5pt 2.1,nodesep=0pt,labelsep=2pt}
\pcline{->}(0,2.5)(0,-2.5) \nbput{{\scriptsize $\alpha_{ab}$}}}}

\endpspicture
\]

Then $TF$ and $TG$ also agree on objects and we define an icon $T\alpha : TF \Tra TG$ by the following (where in this diagram we omit the subscripts on the components of $\alpha$ so that $\alpha^k$ represents $\alpha_{x_{k-1}x_k}\cdots\alpha_{x_0x_1}$)

\[
\psset{unit=0.09cm,labelsep=2pt,nodesep=2pt}
\pspicture(-14,-18)(14,18)

\rput(-30,0){\rnode{a1}{\makebox[9em][r]{$P(k)(x_{k-1},x_k)(x_0,x_1)$}}}  
\rput(30,4){\rnode{a2}{\makebox[9em][l]{$P(k)(Fx_{k-1},Fx_k)(Fx_0,Fx_1)$}}}  
\rput(37,0){\rotatebox{90}{$=$}}
\rput(30,-4){\rnode{a3}{\makebox[9em][l]{$P(k)(Gx_{k-1},Gx_k)(Gx_0,Gx_1)$}}}  

\ncarc[arcangle=40]{->}{a1}{a2}\naput[npos=0.55]{{\scriptsize $TF=1\times F^k$}}
\ncarc[arcangle=-40]{->}{a1}{a3}\nbput[npos=0.55]{{\scriptsize $TG=1\times G^k$}}

{
\rput[c](-5,0){\psset{unit=1mm,doubleline=true,arrowinset=0.6,arrowlength=0.5,arrowsize=0.5pt 2.1,nodesep=0pt,labelsep=2pt}
\pcline{->}(0,2.5)(0,-2.5) \naput{{\scriptsize $1\times\alpha^k$}}}}

\endpspicture
\]

For (2) there is no extra data for $\eta$ and $\mu$, we just have to check the 2-dimensional cylinder conditions.  We have to check that the following diagram commutes (as a diagram in $K$):
\[\begin{small}
\psset{unit=0.09cm,labelsep=2pt,nodesep=2pt}
\pspicture(0,-10)(120,70)

% a1 b1
%    
%
% a2 b2

\rput(20,40){

\rput(0,30){\rnode{a1}{$(a,b)$}}  
\rput(0,0){\rnode{a2}{$(Fa,Fb) = (Ga,Gb)$}}  
% \rput(38,30){\rotatebox{90}{$=$}}
% \rput(30,27){\rnode{a3}{\makebox[4em][l]{$(Gb,Gc)(Ga,Gb)$}}}  

\rput(45,30){\rnode{b1}{$P(1)(a,b)$}}  
\rput(45,0){\rnode{b2}{$P(1)(Ga,Gb)$}}

\ncarc[arcangle=45]{->}{a1}{a2}\naput[npos=0.5]{{\scriptsize $G$}}
\ncarc[arcangle=-45]{->}{a1}{a2}\nbput[npos=0.5]{{\scriptsize $F$}}

\ncline{->}{a1}{b1} \naput{{\scriptsize $\eta$}}
\ncline{->}{a2}{b2} \nbput{{\scriptsize $\eta$}}

\ncarc[arcangle=45]{->}{b1}{b2}\naput[npos=0.5]{{\scriptsize $1\times G$}}

{
\rput[c](0,15){\psset{unit=1mm,doubleline=true,arrowinset=0.6,arrowlength=0.5,arrowsize=0.5pt 2.1,nodesep=0pt,labelsep=2pt}
\pcline{->}(-1.5,0)(1.5,0) \nbput{{\scriptsize $\alpha$}}}}

{
\rput[c](30,15){\psset{unit=1mm,doubleline=true,arrowinset=0.6,arrowlength=0.5,arrowsize=0.5pt 2.1,nodesep=0pt,labelsep=2pt}
\pcline{-}(-1.5,-1.5)(1.5,1.5) \nbput{{\scriptsize $$}}}}

% \ncarc[arcangle=45]{->}{a1}{a2}\naput{{\scriptsize $$}}
% \ncarc[arcangle=-45]{->}{a1}{a2}\nbput{{\scriptsize $$}}

\rput(73,15){$=$}

}

%%%%%%%%%%%%%%%%%%%%%%%%%%%%%%%%%%%% bottom part

\rput(40,0){

% a1 b1
%    
%
% a2 b2

\rput(0,30){\rnode{a1}{$(a,b)$}}  
\rput(0,0){\rnode{a2}{$(Fa,Fb)$}}  
% \rput(38,30){\rotatebox{90}{$=$}}
% \rput(30,27){\rnode{a3}{\makebox[4em][l]{$(Gb,Gc)(Ga,Gb)$}}}  

\rput(45,30){\rnode{b1}{$P(1)(a,b)$}}  
\rput(45,0){\rnode{b2}{$P(1)(Fa,Fb)=P(1)(Ga,Gb)$}}

% \ncarc[arcangle=45]{->}{a1}{a2}\naput[npos=0.5]{{\scriptsize $G$}}
\ncarc[arcangle=-45]{->}{a1}{a2}\nbput[npos=0.5]{{\scriptsize $F$}}

\ncline{->}{a1}{b1} \naput{{\scriptsize $\eta$}}
\ncline{->}{a2}{b2} \nbput{{\scriptsize $\eta$}}

\ncarc[arcangle=-45]{->}{b1}{b2}\nbput[npos=0.5]{{\scriptsize $1\times F$}}
\ncarc[arcangle=45]{->}{b1}{b2}\naput[npos=0.5]{{\scriptsize $1\times G$}}

{
\rput[c](45,15){\psset{unit=1mm,doubleline=true,arrowinset=0.6,arrowlength=0.5,arrowsize=0.5pt 2.1,nodesep=0pt,labelsep=2pt}
\pcline{->}(-1.5,0)(1.5,0) \nbput{{\scriptsize $1 \times \alpha$}}}}

{
\rput[c](20,15){\psset{unit=1mm,doubleline=true,arrowinset=0.6,arrowlength=0.5,arrowsize=0.5pt 2.1,nodesep=0pt,labelsep=2pt}
\pcline{-}(-1.5,-1.5)(1.5,1.5) \nbput{{\scriptsize $$}}}}

% \ncarc[arcangle=45]{->}{a1}{a2}\naput{{\scriptsize $$}}
% \ncarc[arcangle=-45]{->}{a1}{a2}\nbput{{\scriptsize $$}}

}

\eps
\end{small}
\]

To see that this commutes, note that $\eta$ acts as the identity on objects, and on hom objects the action is via the unit of the operad $P$ as follows:
\[(a,b) \hs{0.3}\cong \hs{0.3} 1 \times A(a,b) \ltmap{\mbox{\sf\scr unit} \times 1} P(1) \times (a,b) \hs{0.2} \subset TA.\]
So the $\eta$ and $\alpha$ parts of the diagram act on different parts of the product, and the above cylinder diagram does indeed commute.

We now turn our attention to $\mu$. We know that $\mu$ acts by concatenation together with composition of the operad $P$. We need to check the following diagram:

\[\begin{small}
\psset{unit=0.09cm,labelsep=2pt,nodesep=2pt}
\pspicture(0,-10)(120,70)

% a1 b1
%    
%
% a2 b2

\rput(20,40){

\rput(0,30){\rnode{a1}{$T^2(a,b)$}}  
\rput(0,0){\rnode{a2}{$T^2(Fa,Fb) = T^2(Ga,Gb)$}}  
% \rput(38,30){\rotatebox{90}{$=$}}
% \rput(30,27){\rnode{a3}{\makebox[4em][l]{$(Gb,Gc)(Ga,Gb)$}}}  

\rput(45,30){\rnode{b1}{$T(a,b)$}}  
\rput(45,0){\rnode{b2}{$T(Ga,Gb)$}}

\ncarc[arcangle=45]{->}{a1}{a2}\naput[npos=0.5]{{\scriptsize $T^2G$}}
\ncarc[arcangle=-45]{->}{a1}{a2}\nbput[npos=0.5]{{\scriptsize $T^2F$}}

\ncline{->}{a1}{b1} \naput{{\scriptsize $\mu$}}
\ncline{->}{a2}{b2} \nbput{{\scriptsize $\mu$}}

\ncarc[arcangle=45]{->}{b1}{b2}\naput[npos=0.5]{{\scriptsize $TG$}}

{
\rput[c](0,15){\psset{unit=1mm,doubleline=true,arrowinset=0.6,arrowlength=0.5,arrowsize=0.5pt 2.1,nodesep=0pt,labelsep=2pt}
\pcline{->}(-1.5,0)(1.5,0) \nbput{{\scriptsize $T^2\alpha$}}}}

{
\rput[c](30,15){\psset{unit=1mm,doubleline=true,arrowinset=0.6,arrowlength=0.5,arrowsize=0.5pt 2.1,nodesep=0pt,labelsep=2pt}
\pcline{-}(-1.5,-1.5)(1.5,1.5) \nbput{{\scriptsize $$}}}}

% \ncarc[arcangle=45]{->}{a1}{a2}\naput{{\scriptsize $$}}
% \ncarc[arcangle=-45]{->}{a1}{a2}\nbput{{\scriptsize $$}}

\rput(73,15){$=$}

}

%%%%%%%%%%%%%%%%%%%%%%%%%%%%%%%%%%%% bottom part

\rput(40,0){

% a1 b1
%    
%
% a2 b2

\rput(0,30){\rnode{a1}{$T^2(a,b)$}}  
\rput(0,0){\rnode{a2}{$T^2(Fa,Fb)$}}  
% \rput(38,30){\rotatebox{90}{$=$}}
% \rput(30,27){\rnode{a3}{\makebox[4em][l]{$(Gb,Gc)(Ga,Gb)$}}}  

\rput(45,30){\rnode{b1}{$T(a,b)$}}  
\rput(45,0){\rnode{b2}{$T(Fa,Fb)=T(Ga,Gb)$}}

% \ncarc[arcangle=45]{->}{a1}{a2}\naput[npos=0.5]{{\scriptsize $G$}}
\ncarc[arcangle=-45]{->}{a1}{a2}\nbput[npos=0.5]{{\scriptsize $T^2 F$}}

\ncline{->}{a1}{b1} \naput{{\scriptsize $\mu$}}
\ncline{->}{a2}{b2} \nbput{{\scriptsize $\mu$}}

\ncarc[arcangle=-45]{->}{b1}{b2}\nbput[npos=0.5]{{\scriptsize $TF$}}
\ncarc[arcangle=45]{->}{b1}{b2}\naput[npos=0.5]{{\scriptsize $TG$}}

{
\rput[c](45,15){\psset{unit=1mm,doubleline=true,arrowinset=0.6,arrowlength=0.5,arrowsize=0.5pt 2.1,nodesep=0pt,labelsep=2pt}
\pcline{->}(-1.5,0)(1.5,0) \nbput{{\scriptsize $T\alpha$}}}}

{
\rput[c](20,15){\psset{unit=1mm,doubleline=true,arrowinset=0.6,arrowlength=0.5,arrowsize=0.5pt 2.1,nodesep=0pt,labelsep=2pt}
\pcline{-}(-1.5,-1.5)(1.5,1.5) \nbput{{\scriptsize $$}}}}

% \ncarc[arcangle=45]{->}{a1}{a2}\naput{{\scriptsize $$}}
% \ncarc[arcangle=-45]{->}{a1}{a2}\nbput{{\scriptsize $$}}

}

\eps
\end{small}
\]
Since $T^2\alpha$ and $T\alpha$ both act pointwise, this commutes.

So we have a 2-monad \fckp\ as claimed.

Next we need to establish an equivalence of 2-categories
\[\twocatt{$T$-Alg} \catequiv \twocatt{$K\cat{-Cat}_P$}.\]
In fact, as in the 1-dimensional case, this is an isomorphism because $(K,P)$-categories are precisely algebras for the monad $T$.  (We could have defined the monad $T$ first, and then defined $(K,P)$-categories to be $T$-algebras.) The 1-cell correspondence works as in the 1-dimensional case, so we just need to examine iconic 2-cells.  A 2-cell $\alpha$ of $T$-algebras is a 2-cell of $K$ satisfying the additional cylinder condition:

\[
\psset{unit=0.06cm,labelsep=2pt,nodesep=2pt}
\pspicture(0,-9)(80,38)

% a1 a2
%    
%
% b1 b2

\rput(0,0){

\rput(0,28){\rnode{a1}{$TA$}}  % named node, with something placed there
\rput(30,28){\rnode{a2}{$TB$}}  

\rput(0,3){\rnode{b1}{$A$}}  
\rput(30,3){\rnode{b2}{$B$}}

\ncarc[arcangle=45]{->}{a1}{a2}\naput[npos=0.5]{{\scriptsize $Tf$}}
\ncarc[arcangle=-45]{->}{a1}{a2}\nbput[npos=0.5]{{\scriptsize $Tg$}}

\ncline{->}{a1}{b1} \nbput{{\scriptsize $a$}}
\ncline{->}{a2}{b2} \naput{{\scriptsize $b$}}

\ncarc[arcangle=-45]{->}{b1}{b2}\nbput[npos=0.5]{{\scriptsize $g$}}

{
\rput[c](16,28){\psset{unit=1mm,doubleline=true,arrowinset=0.6,arrowlength=0.5,arrowsize=0.5pt 2.1,nodesep=0pt,labelsep=2pt}
\pcline{->}(0,1.5)(0,-1.5) \nbput{{\scriptsize $T\alpha$}}}}

{
\rput[c](15,6.5){\psset{unit=1mm,doubleline=true,arrowinset=0.6,arrowlength=0.5,arrowsize=0.5pt 2.1,nodesep=0pt,labelsep=1pt}
\pcline{-}(1,1)(-1,-1) %\nbput{{\scriptsize $\tau_g$}}
}
}

% \ncarc[arcangle=45]{->}{a1}{a2}\naput{{\scriptsize $$}}
% \ncarc[arcangle=-45]{->}{a1}{a2}\nbput{{\scriptsize $$}}

\rput(44,20){$=$}

}

\rput(60,0){

\rput(0,28){\rnode{a1}{$TA$}}  % named node, with something placed there
\rput(30,28){\rnode{a2}{$TB$}}  

\rput(0,3){\rnode{b1}{$A$}}  
\rput(30,3){\rnode{b2}{$B$}}

\ncarc[arcangle=45]{->}{a1}{a2}\naput[npos=0.5]{{\scriptsize $Tf$}}
% \ncarc[arcangle=-45]{->}{a1}{a2}\nbput[npos=0.5]{{\scriptsize $Tg$}}

\ncline{->}{a1}{b1} \nbput{{\scriptsize $a$}}
\ncline{->}{a2}{b2} \naput{{\scriptsize $b$}}

% {
% \rput[c](17,28.5){\psset{unit=1mm,doubleline=true,arrowinset=0.6,arrowlength=0.5,arrowsize=0.5pt 2.1,nodesep=0pt,labelsep=2pt}
% \pcline{->}(0,3)(0,0) \nbput{{\scriptsize $\alpha$}}}}
% 
% 
% {
% \rput[c](17,8.5){\psset{unit=1mm,doubleline=true,arrowinset=0.6,arrowlength=0.5,arrowsize=0.5pt 2.1,nodesep=0pt,labelsep=2pt}
% \pcline{->}(3,3)(0,0) \nbput{{\scriptsize $\tau_f$}}}}

\ncarc[arcangle=45]{->}{b1}{b2}\naput[npos=0.5]{{\scriptsize $f$}}
\ncarc[arcangle=-45]{->}{b1}{b2}\nbput[npos=0.5]{{\scriptsize $g$}}

{
\rput[c](15,23){\psset{unit=1mm,doubleline=true,arrowinset=0.6,arrowlength=0.5,arrowsize=0.5pt 2.1,nodesep=0pt,labelsep=1pt}
\pcline{-}(1,1)(-1,-1) %\nbput{{\scriptsize $\tau_f$}}
}
}

{
\rput[c](15,3){\psset{unit=1mm,doubleline=true,arrowinset=0.6,arrowlength=0.5,arrowsize=0.5pt 2.1,nodesep=0pt,labelsep=2pt}
\pcline{->}(0,1.5)(0,-1.5) \nbput{{\scriptsize $\alpha$}}}}

}
% \ncarc[arcangle=45]{->}{a1}{a2}\naput{{\scriptsize $$}}
% \ncarc[arcangle=-45]{->}{a1}{a2}\nbput{{\scriptsize $$}}

\endpspicture
\]

\noi We see that this condition is exactly the cylinder condition for iconic 2-cells of \twocatt{$K\cat{-Cat}_P$}.

Finally note that infinite distributivity follows as for the 1-dimensional version, that is, since products and coproducts in \twocatt{$K\cat{-Cat}_P$} are constructed from those in $K$.  
\end{proof}

We now use this general free $(K,P)$-category 2-monad to construct the two 2-monads which will be composed via a distributive law.  In the 1-dimensional version, the monad for vertical composition comes from a 2-functor 

\[\pspicture(-20,31)(20,42)

\rput(-12,40){\rnode{a1}{$\twocatt{Cat}$}}  
\rput(12,40){\rnode{a2}{$\twocatt{Cat}$}}  

\rput(-12,33){\rnode{b1}{$\cV$}}  
\rput(12,33){\rnode{b2}{$\vgph$}}  

% \rput(-20,20){\rnode{c1}{$\cV$}}  
% \rput(-20,10){\rnode{d1}{$\cW$}}  
% 
% \rput(20,20){\rnode{e1}{$\vgph$}}  
% \rput(20,10){\rnode{f1}{$\cat{$\cW$-Gph}$}}  
% 
\ncline[nodesep=6pt]{->}{a1}{a2}
\ncline[nodesep=12pt]{|->}{b1}{b2}
% \ncline{->}{c1}{d1}\naput{\scr $F$}
% \ncline{->}{e1}{f1}\naput{\scr $F_*$}
% 
% \ncline[nodesep=16pt,offset=-10pt]{|->}{c1}{e1}

% \rput(-20,){\rnode{}{$$}}  
% \rput(20,){\rnode{}{$$}}  

% \rput(,){\rnode{}{$$}}  

\eps\]

\noi We now need to use an analogous 2-functor %$\twocatt{2-Cat} \tra \twocatt{2-Cat}$. 
\[\pspicture(-20,31)(20,42)

\rput(-12,40){\rnode{a1}{$\twocatt{2-Cat}$}}  
\rput(12,40){\rnode{a2}{$\twocatt{2-Cat}$}}  

\rput(-12,33){\rnode{b1}{$K$}}  
\rput(12,33){\rnode{b2}{\twocatt{$K$-Gph}}}  

% \rput(-20,20){\rnode{c1}{$\cV$}}  
% \rput(-20,10){\rnode{d1}{$\cW$}}  
% 
% \rput(20,20){\rnode{e1}{$\vgph$}}  
% \rput(20,10){\rnode{f1}{$\cat{$\cW$-Gph}$}}  
% 
\ncline[nodesep=6pt]{->}{a1}{a2}
\ncline[nodesep=12pt]{|->}{b1}{b2}
% \ncline{->}{c1}{d1}\naput{\scr $F$}
% \ncline{->}{e1}{f1}\naput{\scr $F_*$}
% 
% \ncline[nodesep=16pt,offset=-10pt]{|->}{c1}{e1}

% \rput(-20,){\rnode{}{$$}}  
% \rput(20,){\rnode{}{$$}}  

% \rput(,){\rnode{}{$$}}  

\eps\]
As in the 1-dimensional version, a 2-functor \[F\:K \tra K'\] 
is sent to the 2-functor
\[F_*\:\twocatt{$K$-Gph} \tra \twocatt{$K'$-Gph}\] 
which leaves objects unchanged, and applies $F$ to the hom 2-categories; the action on icons works similarly. The full definition of this 2-functor is given very succinctly in \cite[Section 3]{shul2} where it is called $\cG$.

\begin{proposition}\label{proptwoalg}
Let $T$ be a 2-monad on $K$ with 2-category of algebras $K^T$.  Then $\cG T$ is a 2-monad on $\cG K$ and $(\cG K)^{(\cG T)} \cong \cG(K^T)$, that is
\[\twocatt{$T_\ast\cat{-Alg}$} \cong \twocatt{$(T\cat{-Alg})\cat{-Gph}$}\]

\end{proposition}

\begin{proof}
A $T_\ast$-algebra is a $K$-graph $A$ together with an algebra action $T_\ast A \tra A$ satisfying the usual algebra axioms. As $T_\ast$ leaves the objects of $A$ unchanged, the algebra axioms ensure that the algebra action must be the identity on objects, together with a $T$-algebra action on homs.  That is, we have a graph enriched in \twocatt{$T$-Alg}.

A map of $T_\ast$-algebras is a 1-cell $A \tmap{f} B$ of the underlying $K$-graphs such that the usual square commutes. As the algebra actions are the identity on objects, this gives us any function $f$ on the underlying objects, and a $T$-algebra map at the level of homs. That is, we have a map of graphs enriched in \twocatt{$T$-Alg}.

A 2-cell
\[
\psset{unit=0.1cm,labelsep=2pt,nodesep=2pt}
\pspicture(20,20)

\rput(0,10){\rnode{a1}{$A$}}  % named node, with something placed there
\rput(20,10){\rnode{a2}{$B$}}  % named node, with something placed there

\ncarc[arcangle=45]{->}{a1}{a2}\naput{{\scriptsize $f$}}
\ncarc[arcangle=-45]{->}{a1}{a2}\nbput{{\scriptsize $g$}}

\pcline[linewidth=0.6pt,doubleline=true,arrowinset=0.6,arrowlength=0.8,arrowsize=0.5pt 2.1]{->}(10,13)(10,7)  \naput{{\scriptsize $\alpha$}}

\endpspicture
\]
between maps of $T_\ast$-algebras is a 2-cell of \twocatt{$K$-Gph} satisfying the cylinder condition. Thus it is an icon, so we start with the condition that $f$ and $g$ must agree on objects; the rest of the data and the axioms amount to a 2-cell between $T$-algebras at the level of homs. That is, we have an iconic 2-cell of \twocatt{$(T\cat{-Alg})\cat{-Gph}$}.\end{proof}

% ***needs to be checked. Mike did this? What is going on?**

The 2-monads for horizontal and vertical composition are then straightforward 2-dimensional extensions of Definitions~\ref{verticalone} and \ref{verticaltwo}.

We start with

\begin{itemize}

\item $\vone = \twocatt{Cat}$ (here the 2-cells are ordinary natural transformations), and $P_1$ is an operad in $\cV_1$

\item $\vtwo = \twocatt{\voneponecat}$ with iconic 2-cells, and $P_2$ is an operad in $\vtwo$

\item $\vthree = \twocatt{\vtwoptwocat} = \twocatt{$\cV_1\cat{-Cat}_{P_1}\cat{-Cat}_{P_2}$}$ with iconic 2-cells.

\end{itemize}

\noi The following diagram of categories and monadic forgetful functors becomes a diagram of 2-categories and 2-monadic forgetful 2-functors.

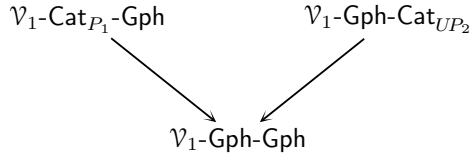
\[\ps(-20,0)(20,20)

%  tl   tr
%     b
\rput(0,0){\rnode{b}{$\cat{$\cV_1$-Gph-Gph}$}}  
\rput(-20,16){\rnode{tl}{\cat{$\cat{$\cV_1$-Cat}_{P_1}$-Gph}}}  
% \rput(20,20){\rnode{tr}{$\cat{$\V_1$-Gph-Cat}_{\UP_1}$}}  
  \rput(20,16){\rnode{tr}{$\cat{$\cV_1$-Gph-Cat}_{\UP_2}$}}

\ncline{->}{tl}{b} \naput{{\scriptsize $$}}
\ncline{->}{tr}{b} \naput{{\scriptsize $$}}

\eps\]

\noi As in the 1-categorical case, we define the following 2-monads on \cat{$\cV_1$-Gph-Gph}

\begin{itemize}

\item For vertical composition: $S = (\fc_{P_1})_\ast$.
\item For horizontal composition: $T = \fc_{\UP_2}$.

\end{itemize}

Finally, we check that the \dl\ from the 1-categorical case
\[ST \Tmap{\lambda} TS\]
extends to a 2-dimensional \dl. We have to check that $\lambda$ becomes a 2-transformation, that is, that the cylinder diagram commutes (as a diagram in \twocatt{$\cV_1$-Gph-Gph}):

\[\begin{small}
\psset{unit=0.09cm,labelsep=2pt,nodesep=2pt}
\pspicture(0,-10)(100,40)

% a1 b1
%    
%
% a2 b2

\rput(0,0){

\rput(0,30){\rnode{a1}{$STX$}}  
\rput(0,0){\rnode{a2}{$STY$}}  
% \rput(38,30){\rotatebox{90}{$=$}}
% \rput(30,27){\rnode{a3}{\makebox[4em][l]{$(Gb,Gc)(Ga,Gb)$}}}  

\rput(30,0){
\rput(0,30){\rnode{b1}{$TSX$}}  
\rput(0,0){\rnode{b2}{$TSY$}}  

{
\rput[c](-7,15){\psset{unit=1mm,doubleline=true,arrowinset=0.6,arrowlength=0.5,arrowsize=0.5pt 2.1,nodesep=0pt,labelsep=2pt}
\pcline{-}(-1.5,-1.5)(1.5,1.5) \nbput{{\scriptsize $$}}}}

}

\ncarc[arcangle=45]{->}{a1}{a2}\naput[npos=0.5]{{\scriptsize $STG$}}
\ncarc[arcangle=-45]{->}{a1}{a2}\nbput[npos=0.5]{{\scriptsize $STF$}}

\ncline{->}{a1}{b1} \naput{{\scriptsize $\lambda_X$}}
\ncline{->}{a2}{b2} \nbput{{\scriptsize $\lambda_Y$}}

\ncarc[arcangle=45]{->}{b1}{b2}\naput[npos=0.5]{{\scriptsize $TSG$}}

{
\rput[c](0,15){\psset{unit=1mm,doubleline=true,arrowinset=0.6,arrowlength=0.5,arrowsize=0.5pt 2.1,nodesep=0pt,labelsep=2pt}
\pcline{->}(-1.5,0)(1.5,0) \nbput{{\scriptsize $ST\alpha$}}}}

% \ncarc[arcangle=45]{->}{a1}{a2}\naput{{\scriptsize $$}}
% \ncarc[arcangle=-45]{->}{a1}{a2}\nbput{{\scriptsize $$}}

\rput(50,15){$=$}

}

%%%%%%%%%%%%%%%%%%%%%%%%%%%%%%%%%%%% bottom part

\rput(70,0){

% a1 b1
%    
%
% a2 b2

\rput(0,30){\rnode{a1}{$STX$}}  
\rput(0,0){\rnode{a2}{$STY$}}  
% \rput(38,30){\rotatebox{90}{$=$}}
% \rput(30,27){\rnode{a3}{\makebox[4em][l]{$(Gb,Gc)(Ga,Gb)$}}}  

\rput(30,0){
\rput(0,30){\rnode{b1}{$TSX$}}  
\rput(0,0){\rnode{b2}{$TSY$}}  

{
\rput[c](0,15){\psset{unit=1mm,doubleline=true,arrowinset=0.6,arrowlength=0.5,arrowsize=0.5pt 2.1,nodesep=0pt,labelsep=2pt}
\pcline{->}(-1.5,0)(1.5,0) \nbput{{\scriptsize $TS\alpha$}}}}

}

% \ncarc[arcangle=45]{->}{a1}{a2}\naput[npos=0.5]{{\scriptsize $G$}}
\ncarc[arcangle=-45]{->}{a1}{a2}\nbput[npos=0.5]{{\scriptsize $STF$}}

\ncline{->}{a1}{b1} \naput{{\scriptsize $\lambda_X$}}
\ncline{->}{a2}{b2} \nbput{{\scriptsize $\lambda_Y$}}

\ncarc[arcangle=-45]{->}{b1}{b2}\nbput[npos=0.5]{{\scriptsize $TSF$}}
\ncarc[arcangle=45]{->}{b1}{b2}\naput[npos=0.5]{{\scriptsize $TSG$}}

{
\rput[c](7,15){\psset{unit=1mm,doubleline=true,arrowinset=0.6,arrowlength=0.5,arrowsize=0.5pt 2.1,nodesep=0pt,labelsep=2pt}
\pcline{-}(-1.5,-1.5)(1.5,1.5) \nbput{{\scriptsize $$}}}}

% \ncarc[arcangle=45]{->}{a1}{a2}\naput{{\scriptsize $$}}
% \ncarc[arcangle=-45]{->}{a1}{a2}\nbput{{\scriptsize $$}}

}

\eps
\end{small}
\]

\noi This is true because all the monads act pointwise, leaving the operad action unchanged.

\subsection{Example: \bicatscats}\label{examplebicatscat}

For this work we use the following special case:

\begin{itemize}

\item We set $\vone = \twocatt{Cat}$ and $P_1 =$ the operad for bicategories (the free contractible operad on one binary and one nullary operation).

\item Then by definition $\vtwo := \twocatt{\voneponecat} = \twocatt{\bicats}$.

\item We set $P_2 = 1$ (the terminal operad).

\item Then by definition $\vthree := \twocatt{$(\bicats, 1)$-Cat} = \twocatt{\bicats-Cat}$.

\end{itemize}

\noi We then have the following 2-monads on \twocatt{\catgphgph}:

\begin{itemize}

\item for (strict) horizontal composition: $H = \fc_{\UP_2}$,

\item for (weak) vertical composition: $V = (\fc_{P_1})_\ast$,

\end{itemize}

% 
% \begin{itemize}
% 
% \item $T$ for (strict) horizontal composition
% 
% \item $S$ for (weak) vertical composition
% 
% 
% \end{itemize}

\noi and  a 2-dimensional \dl\
\[\VH \Tra HV\]
with
\[\twocatt{$HV$-Alg} \cong \twocatt{\bicats-Cat}\]

\noi showing that the 2-category of \bicatscats\ with iconic 2-cells can be constructed as the 2-category of algebras for the composite 2-monad $HV$. This is the construction we will use to define weak maps between \bicatscats\ in the next section.

%%%%%%%%%%%%%%%%%%%%%%%%%%%%%%%%%%%%%%%%%%%%%%%%% 
% Section 2
%%%%%%%%%%%%%%%%%%%%%%%%%%%%%%%%%%%%%%%%%%%%%%%%%

\section{Algebras via distributive laws}
\label{two}

% ***removed from earlier**

% In the previous section we defined a 2-category of \bicatscats\ which only has strict maps.  We will define weak maps via the 2-monads and the \dl\ of the previous section. As we defined the 2-monad for free \bicatscats\ via a 2-monad for vertical composition and a 2-monad for horizontal composition, 

In this section we will give more details about how to use 2-dimensional distributive laws to study weak maps between $TS$-algebras in general.  This will enable us to define weak maps of \bicatscats\ via weak maps of algebras for the individual 2-monads. These are only weak enough in the doubly-degenerate case because, in effect, they force (or assume) strict functoriality on 1-cells. For general tricategories this is too strict, but for doubly-degenerate ones there is only one 0-cell and only one 1-cell so weak functors are trivially strictly functorial on 1-cells.

As our motivating example is not a fully weak 2-dimensional situation we will simplify this situation by keeping it as strict as that example:
\begin{itemize}\setlength{\itemsep}{2pt}

\item strict 2-categories, strict 2-functors, strict transformations
\item strict 2-monads and strict \dls\ between them
\item strict algebras
\item weak maps of algebras
\item strict transformations of algebras

\end{itemize}

For this reason we will not just use the analogous results about pseudo distributive laws from \chp\ giving a biequivalence of bicategories
\[\cat{$TS$-Alg} \catequiv \cat{$T'$-Alg}.\]
We have a much better behaved situation and we can characterise weak maps of $TS$-algebras precisely with a direct approach.  Our goal is to achieve a convenient explicit description of weak maps of doubly-degenerate \bicatscats.  We follow Kelly and Street \cite{ks1} and consider strict 2-monads (which they call doctrines), strict algebras, and weak maps of algebras (although their focus is more on the lax case).  Keeping the 2-monads and the algebras strict makes for a much simpler theory, without us losing any of the expressivity we need.

We begin by giving a closer examination of algebras for a composite monad $TS$.

% recalling the basic definitions and results about monads and distributive laws, which also hold for 2-monads on 2-categories, and strict distributive laws between them ***check?**

%%%%%%%%%%%%%%%%%%%%%%%%%%%%%%%%%%%%%%%%%%%%%% pasted from EH Section 1

\subsection{Algebras via distributive laws}

The basic theorem about a distributive law of $S$ over $T$ (Theorem~\ref{barrwells}) shows that algebras for the composite monad $TS$ are the same as algebras for the lifted monad $T'$.  However, in practice we often express them as a $T$-algebra structure and an $S$-algebra structure satisfying an interaction axiom coming from the distibutive law.  This is one possible answer to the question of why a distributive law is called a ``law'' when it is a piece of extra structure on the monads (a natural transformation): it is a law at the level of algebra structures. 

The following easy corollary makes this precise.  

\begin{cor}\cite[Section 2]{bec1}\label{distpair} Let $S$ and $T$ be monads on a category $\cC$, and $\lambda: ST \Tra TS$ a \dl. Then a $TS$-algebra is equivalently a $T$-algebra 
\pspicture(-3,4)(5,13)
\rput(0,10){\rnode{a}{$TA$}}
\rput(0,0){\rnode{b}{$A$}}
\ncline{->}{a}{b}\naput[npos=0.4]{$t$}
\endpspicture
and $S$-algebra
\pspicture(-3,4)(5,13)
\rput(0,10){\rnode{a}{$SA$}}
\rput(0,0){\rnode{b}{$A$}}
\ncline{->}{a}{b}\naput[npos=0.4]{$s$}
\endpspicture
such that the following diagram commutes:
\[\psset{unit=0.065cm,labelsep=0pt,nodesep=3pt}
\pspicture(20,44)

% a1 a2
% b1 b2
% c1 c2

\rput(0,40){\rnode{a1}{$STA$}} % top left
\rput(26,40){\rnode{a2}{$SA$}} % top mid

\rput(0,20){\rnode{b1}{$TSA$}}   % bottom left
% \rput(20,20){\rnode{b2}{$B_2$}}  % bottom mid

\rput(0,0){\rnode{c1}{$TA$}}   % bottom left
\rput(26,0){\rnode{c2}{$A$}}  % bottom mid

\psset{nodesep=3pt,labelsep=2pt,arrows=->}
\ncline{a1}{a2}\naput{{\scriptsize $St$}} % top
% \ncline{b1}{b2}\nbput{{\scriptsize $T$}} % bottom
\ncline{c1}{c2}\nbput{{\scriptsize $t$}} % bottom

\ncline{a1}{b1}\nbput{{\scriptsize $\lambda_A$}} % left
\ncline{b1}{c1}\nbput{{\scriptsize $Ts$}} % left

\ncline{a2}{c2}\naput{{\scriptsize $s$}} % right
% \ncline{b2}{c2}\naput{{\scriptsize $H_2$}} % right

\endpspicture\]
\end{cor}

We follow Beck and call this a $\lambda$-distributive algebra pair, and will refer to this diagram as ``$s/t$ interaction''.

We have included a proof in the Appendix, expressed in string diagrams, as we found it helpful for what follows.

\begin{remark}
It will later be useful to have an explicit correspondence between $TS$-algebras and $\lambda$-distributive algebra pairs.

\begin{itemize}

\item Given a $TS$-algebra $\left(\pspicture(-3.5,8)(3.5,17)
% a1 
% a2 
\rput(0,14){\rnode{a1}{$TSA$}}  %  top
\rput(0,3){\rnode{a2}{$A$}}  % bottom

\ncline{->}{a1}{a2} \naput{{\scriptsize $\theta$}} % top

\endpspicture\right)$
we get the pair
$\left(\pspicture(-3.5,13.5)(18,28)
% a1 
% a2 
\rput(0,25){\rnode{a0}{$SA$}}
\rput(0,14){\rnode{a1}{$TSA$}}  %  top
\rput(0,3){\rnode{a2}{$A$}}  % bottom

\ncline{->}{a0}{a1} \naput[npos=0.45]{{\scriptsize $\eta^T_{SA}$}} % top
\ncline{->}{a1}{a2} \naput{{\scriptsize $\theta$}} % top
\rput(6,13){,}
\rput(12,0){
\rput(0,25){\rnode{a0}{$TA$}}
\rput(0,14){\rnode{a1}{$TSA$}}  %  top
\rput(0,3){\rnode{a2}{$A$}}  % bottom

\ncline{->}{a0}{a1} \naput[npos=0.45]{{\scriptsize $T\eta^S_A$}} % top
\ncline{->}{a1}{a2} \naput{{\scriptsize $\theta$}} % top
}
\endpspicture\right)$

% \[(SA \ltmap{\eta^T_{TA}} TSA \ltmap{\theta} A), (TA \ltmap{T\eta^S_{A}} TSA \ltmap{\theta} A)\]

\item Given the pair
$\left(\pspicture(-3,8)(14,17)
\rput(0,0){
\rput(0,14){\rnode{a1}{$SA$}}  %  top
\rput(0,3){\rnode{a2}{$A$}}  % bottom
\ncline{->}{a1}{a2} \naput{{\scriptsize $s$}} % top
}
\rput(6,8){,}
\rput(10,0){
\rput(0,14){\rnode{a1}{$TA$}}  %  top
\rput(0,3){\rnode{a2}{$A$}}  % bottom
\ncline{->}{a1}{a2} \naput{{\scriptsize $t$}} % top
}
\endpspicture\right)$
we get the $TS$-algebra
$\left(\pspicture(-3.5,13.5)(3.5,28)
% a1 
% a2 
\rput(0,25){\rnode{a0}{$TSA$}}
\rput(0,14){\rnode{a1}{$TA$}}  %  top
\rput(0,3){\rnode{a2}{$A$}}  % bottom

\ncline{->}{a0}{a1} \naput[npos=0.45]{{\scriptsize $Ts$}} % top
\ncline{->}{a1}{a2} \naput[npos=0.45]{{\scriptsize $t$}} % top
\endpspicture\right)$

% $(SA \tmap{s}A), (TA \tmap{t}A)$ we get the $TS$-algebra
% \[TSA \tmap{Ts} TA \tmap{t}A\]

\end{itemize}

\noi This follows from the above corollary and standard correspondence between $T$-algebras and $T'$-algebras.

\end{remark}

\begin{examples}

\begin{enumerate}

\item In the case of rings, the above correspondence expresses a ring as a set with a monoid structure and an abelian group structure satisfying distributivity of the monoid operation over the group operation, that is, the usual direct definition of a ring.

\item In the case of 2-categories, the above correspondence expresses a 2-category as a 2-globular set equipped with vertical and horizontal composition satisfying interchange, that is, one of the usual direct definitions of a 2-category.

\item We can modify the example of 2-categories to produce doubly degenerate 2-categories. Let $\cC, T, S, \lambda$ be as in Example~\ref{egtwopointfour}.  A doubly degenerate $TS$-algebra is a $\lambda$-distributive pair where the underlying 2-globular set is of the form

%
%ABC
\[
\psset{unit=0.1cm,labelsep=0pt,nodesep=3pt}
\pspicture(0,-3)(40,3)
\rput(0,0){\rnode{A}{$A_2$}}
\rput(20,0){\rnode{B}{$1$}}
\rput(40,0){\rnode{C}{$1$}}
\psset{nodesep=3pt,arrows=->}
\ncline[offset=3.5pt]{A}{B}\naput[npos=0.5,labelsep=2pt]{{\scriptsize $$}}
\ncline[offset=-3.5pt]{A}{B}\nbput[npos=0.5,labelsep=2pt]{{\scriptsize$$}}
\ncline[offset=3.5pt]{B}{C}\naput[npos=0.5,labelsep=2pt]{{\scriptsize$$}}
\ncline[offset=-3.5pt]{B}{C}\nbput[npos=0.5,labelsep=2pt]{{\scriptsize$$}}

\endpspicture
\]

\noi Thus it is a set equipped with two monoid structures, with one distributing over the other. The standard Eckmann--Hilton argument shows that these must be the same and commutative.

\item Finally, the correspondence tells us that a \ddbicatscat\ is a \dd\ \cat{Cat-Gph}-graph (that is, a category) equipped with a weak ``vertical'' tensor product and a strict ``horizontal'' tensor product, satisfying strict interchange.

\end{enumerate}

\end{examples}

\begin{remark}

Doubly degenerate 2-categories cannot be simply expressed by restricting $T$ and $S$ to a category of doubly degenerate 2-globular sets (ie \Set) as $T$ does not restrict --- even if $A$ is a doubly degenerate 2-globular set, $TA$ is not as it creates formal composites of the unique 1-cell. The same is true for doubly-degenerate \bicatscats.
\end{remark}

% ;  this is something we have to deal with when formulating a general monad for free $k$-degenerate $n$-categories, but we will not address that in this work.                                                                                                                                                                                   

% ***p.V7 Trimble 2-cats apparently I'm saving the example for later?**
% 
% 
% \bigskip

% \begin{example}
% \end{example}

\subsection{Maps of algebras via distributive laws}

We now move on to consider maps of algebras via distributive laws. Just as we can express $TS$-algebras in terms of a $T$-algebra and an $S$-algebra (in the presence of a distributive law of $S$ over $T$), we can express maps of $TS$-algebras as a map of $T$-algebras and a map of $S$-algebras.

\begin{theorem}

In the presence of a distributive law of monads $S$ over $T$, a map $f$ of $TS$-algebras is precisely a map $f$ which is a map of both the associated $T$-algebra and the associated $S$-algebra.

\end{theorem}

% ****the following proof is pretty easy to do directly so why not just do that? see frame 4 [diagram 1]**

\begin{proof}
Consider $TS$-algebras $TSA \tmap{\theta} A$ and $TSB \tmap{\phi} B$ and a map $A \tmap{f} B$ between them, so the following diagram commutes. 
\[
\psset{unit=0.1cm,labelsep=2pt,nodesep=3pt}
\pspicture(0,-2)(20,20)

% a1 a2
% a3 a4

%%%%%%%%%% top

\rput(0,16){\rnode{a1}{$TSA$}}  % top left
\rput(20,16){\rnode{a2}{$TSB$}}  % top right
\rput(0,0){\rnode{a3}{$A$}}  % bottom left
\rput(20,0){\rnode{a4}{$B$}}  % bottom right

\ncline{->}{a1}{a2} \naput{{\scriptsize $TSf$}} % top
\ncline{->}{a3}{a4} \nbput{{\scriptsize $f$}} % bottom
\ncline{->}{a1}{a3} \nbput{{\scriptsize $\theta$}} % left
\ncline{->}{a2}{a4} \naput{{\scriptsize $\phi$}} % right

\endpspicture
\]
We need to check that $f$ is a map of the associated $T$-algebras and of the associated $S$-algebras. This is seen from the following diagrams where the top squares are naturality squares:

\[\psset{unit=0.08cm,labelsep=0pt,nodesep=3pt}
\pspicture(0,-5)(60,42)

\rput(0,20){
\pspicture(20,42)

% a1 a2
% b1 b2
% c1 c2

\rput(0,40){\rnode{a1}{$SA$}} % top left
\rput(24,40){\rnode{a2}{$SB$}} % top mid

\rput(0,20){\rnode{b1}{$TSA$}}   % bottom left
\rput(24,20){\rnode{b2}{$TSB$}}  % bottom mid

\rput(0,0){\rnode{c1}{$A$}}   % bottom left
\rput(24,0){\rnode{c2}{$B$}}  % bottom mid

\psset{nodesep=3pt,labelsep=2pt,arrows=->}
\ncline{a1}{a2}\naput{{\scriptsize $Sf$}} % top
\ncline{b1}{b2}\nbput[npos=0.45]{{\scriptsize $TSf$}} % bottom
\ncline{c1}{c2}\nbput{{\scriptsize $f$}} % bottom

\ncline{a1}{b1}\nbput{{\scriptsize $\eta^T_{SA}$}} % left
\ncline{b1}{c1}\nbput{{\scriptsize $\theta_A$}} % left

\ncline{a2}{b2}\naput{{\scriptsize $\eta^T_{SB}$}} % right
\ncline{b2}{c2}\naput{{\scriptsize $\theta_B$}} % right

\psset{labelsep=1.5pt}

% \rput(2,0){
% \pnode(13,33){a3}
% \pnode(7,27){b3}
% \ncline[doubleline=true,arrowinset=0.6,arrowlength=0.8,arrowsize=0.5pt 2.1,nodesep=5pt]{-}{a3}{b3} \naput[npos=0.4]{{\scriptsize $$}}
% }
% 
% \rput(2,0){
% \pnode(13,13){a3}
% \pnode(7,7){b3}
% \ncline[doubleline=true,arrowinset=0.6,arrowlength=0.8,arrowsize=0.5pt 2.1]{a3}{b3} \naput[npos=0.4]{{\scriptsize $\phi$}}
% }

\endpspicture}

\rput(50,20){
\pspicture(20,42)

% a1 a2
% b1 b2
% c1 c2

\rput(-2,40){\rnode{a1}{$TA$}} % top left
\rput(22,40){\rnode{a2}{$TB$}} % top mid

\rput(-2,20){\rnode{b1}{$TSA$}}   % bottom left
\rput(22,20){\rnode{b2}{$TSB$}}  % bottom mid

\rput(-2,0){\rnode{c1}{$A$}}   % bottom left
\rput(22,0){\rnode{c2}{$B$}}  % bottom mid

\psset{nodesep=3pt,labelsep=2pt,arrows=->}
\ncline{a1}{a2}\naput{{\scriptsize $Tf$}} % top
\ncline{b1}{b2}\nbput[npos=0.45]{{\scriptsize $TSf$}} % bottom
\ncline{c1}{c2}\nbput{{\scriptsize $f$}} % bottom

\ncline{a1}{b1}\nbput{{\scriptsize $T\eta^S_{A}$}} % left
\ncline{b1}{c1}\nbput{{\scriptsize $\theta_A$}} % left

\ncline{a2}{b2}\naput{{\scriptsize $T\eta^S_{B}$}} % right
\ncline{b2}{c2}\naput{{\scriptsize $\theta_B$}} % right

\psset{labelsep=1.5pt}

% \pnode(13,33){a3}
% \pnode(7,27){b3}
% \ncline[doubleline=true,arrowinset=0.6,arrowlength=0.8,arrowsize=0.5pt 2.1,nodesep=5pt]{-}{a3}{b3} \naput[npos=0.4]{{\scriptsize $$}}
% 
% \pnode(13,13){a3}
% \pnode(7,7){b3}
% \ncline[doubleline=true,arrowinset=0.6,arrowlength=0.8,arrowsize=0.5pt 2.1]{a3}{b3} \naput[npos=0.4]{{\scriptsize $\phi$}}
% 

\endpspicture}

\endpspicture
\]

Conversely suppose we know that $f$ is a map of the associated $T$-algebra and $S$-algebra structures given by the $\lambda$-distributive pairs

\[\psset{nodesep=2pt}
\left(\pspicture(-3,8)(14,17)
\rput(0,0){
\rput(0,14){\rnode{a1}{$TA$}}  %  top
\rput(0,3){\rnode{a2}{$A$}}  % bottom
\ncline{->}{a1}{a2} \naput{{\scriptsize $a_t$}} % top
}
\rput(6,8){,}
\rput(10,0){
\rput(0,14){\rnode{a1}{$SA$}}  %  top
\rput(0,3){\rnode{a2}{$A$}}  % bottom
\ncline{->}{a1}{a2} \naput{{\scriptsize $a_s$}} % top
}
\endpspicture\right)
\hs{1} \mbox{and} \hs{1}
\psset{nodesep=2pt}
\left(\pspicture(-3,8)(14,17)
\rput(0,0){
\rput(0,14){\rnode{a1}{$TB$}}  %  top
\rput(0,3){\rnode{a2}{$B$}}  % bottom
\ncline{->}{a1}{a2} \naput{{\scriptsize $b_t$}} % top
}
\rput(6,8){,}
\rput(10,0){
\rput(0,14){\rnode{a1}{$SB$}}  %  top
\rput(0,3){\rnode{a2}{$B$}}  % bottom
\ncline{->}{a1}{a2} \naput{{\scriptsize $b_s$}} % top
}
\endpspicture\right)
\]

We check that $f$ satisfies the axioms for a map of $TS$-algebras. This is seen from the following diagram,where by definition the left-hand map is $\theta$ and the right-hand map is $\phi$, and the squares are the axioms for a $T$-algebra map and $S$-algebra map:

 \[\psset{unit=0.08cm,labelsep=0pt,nodesep=3pt}
\pspicture(0,-2)(20,42)

\rput(8,20){
\pspicture(20,42)

% a1 a2
% b1 b2
% c1 c2

\rput(-2,40){\rnode{a1}{$TSA$}} % top left
\rput(22,40){\rnode{a2}{$TSB$}} % top mid

\rput(-2,20){\rnode{b1}{$TA$}}   % bottom left
\rput(22,20){\rnode{b2}{$TB$}}  % bottom mid

\rput(-2,0){\rnode{c1}{$A$}}   % bottom left
\rput(22,0){\rnode{c2}{$B$}}  % bottom mid

\psset{nodesep=3pt,labelsep=2pt,arrows=->}
\ncline{a1}{a2}\naput[npos=0.45]{{\scriptsize $TSf$}} % top
\ncline{b1}{b2}\nbput{{\scriptsize $Tf$}} % bottom
\ncline{c1}{c2}\nbput{{\scriptsize $f$}} % bottom

\ncline{a1}{b1}\nbput{{\scriptsize $Ta_s$}} % left
\ncline{b1}{c1}\nbput{{\scriptsize $a_t$}} % left

\ncline{a2}{b2}\naput{{\scriptsize $Tb_s$}} % right
\ncline{b2}{c2}\naput{{\scriptsize $b_t$}} % right

\psset{labelsep=1.5pt}

% \pnode(13,33){a3}
% \pnode(7,27){b3}
% \ncline[doubleline=true,arrowinset=0.6,arrowlength=0.8,arrowsize=0.5pt 2.1]{->}{a3}{b3} \naput[npos=0.4]{{\scriptsize $T\sigma$}}
% 
% \pnode(13,13){a3}
% \pnode(7,7){b3}
% \ncline[doubleline=true,arrowinset=0.6,arrowlength=0.8,arrowsize=0.5pt 2.1]{a3}{b3} \naput[npos=0.4]{{\scriptsize $\tau$}}

\endpspicture}

\endpspicture
\]

\end{proof}

\begin{examples}

\begin{enumerate}

\item A ring homomorphism consists of a map that is both a group homomorphism and a monoid homomorphism.

\item A functor of 2-categories is a map of the underlying 2-globular sets that respects both the horizontal composition and the vertical composition.  This has a slightly different emphasis from the expression as a $T'$-algebra map which says it is a functor on hom-categories and a \Cat-enriched functor as well.

\end{enumerate}

\end{examples}

%%%%%%%%%%%%%%%%%%%%%%%%%%%%%%%%%%%%%%%%%%%%%%%% end paste from EH section 1

We now move to the case of weak maps of algebras.

%%%%%%%%%%%%%%%%%%%%%%%%%%%%%%%%%%%%%%%%%%%%%%%%%% pasted from EH Section 3

\subsection{Weak maps of algebras via distributive laws}

We now extend the 1-categorical results to a 2-categorical framework. Since we are dealing with strict 2-monads, strict algebras and strict \dls\ throughout this work, the previous theorems characterising $TS$-algebras still hold.  However, need a new result characterising weak maps of $TS$-algebras.  First we recall the definition of a weak map of algebras for a 2-monad $T$.

\begin{definition}
Let $T$ be a 2-monad on a 2-category \cC.  Given strict algebras $TA \tmap{a} A$ and $TB \tmap{b} B$ a weak map between them is given by a 1-cell $A \tmap{f} B$ and a 2-cell isomorphism
\[\psset{unit=0.075cm,labelsep=0pt,nodesep=3pt}
\pspicture(20,22)

% a1 a2
% b1 b2

\rput(0,20){\rnode{a1}{$TA$}} % top left
\rput(20,20){\rnode{a2}{$TB$}} % top right

\rput(0,0){\rnode{b1}{$A$}}   % bottom left
\rput(20,0){\rnode{b2}{$B$}}  % bottom right

\psset{nodesep=3pt,labelsep=2pt,arrows=->}
\ncline{a1}{a2}\naput{{\scriptsize $Tf$}} % top
\ncline{b1}{b2}\nbput{{\scriptsize $f$}} % bottom
\ncline{a1}{b1}\nbput{{\scriptsize $a$}} % left
\ncline{a2}{b2}\naput{{\scriptsize $b$}} % right

\psset{labelsep=1.5pt}
\pnode(13,13){a3}
\pnode(7,7){b3}
\ncline[doubleline=true,arrowinset=0.6,arrowlength=0.8,arrowsize=0.5pt 2.1]{a3}{b3} \nbput[npos=0.4]{{\scriptsize $\tau$}}
\naput[npos=0.4]{{\scriptsize $\sim$}}

\endpspicture\]
satisfying the following axioms. 
\[\psset{unit=0.1cm,labelsep=1pt,nodesep=2pt}
\pspicture(0,5)(50,43)

% d1 d3
%   c2  c4
% b1
%   a2  a4

\rput(26,27){$=$}

\rput(-25,0){

\rput(20,12){\rnode{a2}{$A$}}  
\rput(40,12){\rnode{a4}{$B$}}  
\rput(6,21){\rnode{b1}{$TA$}}  
\rput(20,30){\rnode{c2}{$TA$}}  
\rput(40,30){\rnode{c4}{$TB$}}  
\rput(6,38){\rnode{d1}{$T^2A$}}  
\rput(27,38){\rnode{d3}{$T^2B$}}

\ncline{->}{d1}{d3} \naput{{\scriptsize $T^2f$}} % top face
\ncline{->}{d3}{c4} \naput{{\scriptsize $\mu_B$}}
\ncline{->}{d1}{c2} \nbput{{\scriptsize $\mu_A$}}
\ncline{->}{c2}{c4} \nbput{{\scriptsize $Tf$}} 

\ncline{->}{d1}{b1} \nbput{{\scriptsize $Ta$}}
\ncline{->}{b1}{a2} \nbput{{\scriptsize $a$}}
\ncline{->}{c2}{a2} \naput{{\scriptsize $a$}} 

\ncline{->}{c4}{a4} \naput{{\scriptsize $b$}} % right
\ncline{->}{a2}{a4} \nbput[labelsep=3pt]{{\scriptsize $f$}}

% 2-cells

%top 
{
\rput[c](22,34){\scr $=$}
% {\psset{unit=1mm,doubleline=true,arrowinset=0.6,arrowlength=0.5,arrowsize=0.5pt 2.1,nodesep=0pt,labelsep=1pt}
% \pcline{-}(3,3)(0,0) \naput{{\scriptsize }}}
}

%left
{
\rput[c](13,26){\scr $=$}
% {\psset{unit=1mm,doubleline=true,arrowinset=0.6,arrowlength=0.5,arrowsize=0.5pt 2.1,nodesep=0pt,labelsep=2pt}
% \pcline{-}(3,3)(0,0) \naput{{\scriptsize }}}
}

%front
{
\rput[c](29,20){\psset{unit=1mm,doubleline=true,arrowinset=0.6,arrowlength=0.5,arrowsize=0.5pt 2.1,nodesep=0pt,labelsep=2pt}
\pcline{->}(3,3)(0,0) \naput{{\scriptsize $\tau$}}}}
}

\rput(35,0){

% d1 d3
%       c4
% b1 b3
%   a2  a4

\rput(19,12){\rnode{a2}{$A$}}  
\rput(40,12){\rnode{a4}{$B$}}  
\rput(6,21){\rnode{b1}{$TA$}}  
\rput(27,21){\rnode{b3}{$TB$}}  

% \rput(20,30){\rnode{c2}{$TA$}}  
\rput(40,30){\rnode{c4}{$TB$}}  
\rput(6,38){\rnode{d1}{$T^2A$}}  
\rput(27,38){\rnode{d3}{$T^2B$}}

% 
% \rput(20,12){\rnode{a2}{$A$}}  
% \rput(46,12){\rnode{a4}{$B$}}  
% \rput(3,23){\rnode{b1}{$TA$}}  
% \rput(28,23){\rnode{b3}{$TB$}}  
% 
% 
% % \rput(20,30){\rnode{c2}{$TA$}}  
% \rput(46,30){\rnode{c4}{$TB$}}  
% \rput(3,40){\rnode{d1}{$T^2A$}}  
% \rput(28,40){\rnode{d3}{$T^B$}}  

\ncline{->}{d1}{d3} \naput{{\scriptsize $T^2f$}} % top face
\ncline{->}{d3}{c4} \naput{{\scriptsize $\mu_B$}}
%\ncline{->}{d1}{c2} \naput{{\scriptsize $\mu^T_A$}}
%\ncline{->}{c2}{c4} \naput{{\scriptsize $Tf$}} 

\ncline{->}{d1}{b1} \nbput{{\scriptsize $Ta$}}
\ncline{->}{b1}{a2} \nbput{{\scriptsize $a$}}
%\ncline{->}{c2}{a2} \naput{{\scriptsize $a$}} 

\ncline{->}{c4}{a4} \naput{{\scriptsize $b$}} % right
\ncline{->}{a2}{a4} \nbput[labelsep=3pt]{{\scriptsize $f$}}
\ncline{->}{b3}{a4} \naput{{\scriptsize $b$}} % right

\ncline{->}{d3}{b3} \naput{{\scriptsize $Tb$}}
\ncline{->}{b1}{b3} \naput{{\scriptsize $Tf$}}

% 2-cells

%right
{
\rput[c](34,26){\scr $=$}
% {\psset{unit=1mm,doubleline=true,arrowinset=0.6,arrowlength=0.5,arrowsize=0.5pt 2.1,nodesep=0pt,labelsep=1pt}
% \pcline{-}(3,3)(0,0) \naput{{\scriptsize }}}
}

%bottom
{
\rput[c](23,15){\psset{unit=1mm,doubleline=true,arrowinset=0.6,arrowlength=0.5,arrowsize=0.5pt 2.1,nodesep=0pt,labelsep=2pt}
\pcline{->}(3,3)(0,0) \naput{{\scriptsize $\tau$}}}}

%back
{
\rput[c](15,29){\psset{unit=1mm,doubleline=true,arrowinset=0.6,arrowlength=0.5,arrowsize=0.5pt 2.1,nodesep=0pt,labelsep=2pt}
\pcline{->}(3,3)(0,0) \naput{{\scriptsize T$\tau$}}}}

}

\endpspicture\]

%%%%%%%%%%%%%%%%%%%%%%%%%%%%%%%%%%%%%%%%%%%%%%%%%%%%%%%%%%%

\[\psset{unit=0.085cm,labelsep=1pt,nodesep=2pt}
\pspicture(0,8)(50,36)

% d1 d3
%   c2  c4
% 
%   a2  a4

\rput(-32,0){

\rput(25,12){\rnode{a2}{$A$}}  
\rput(46,12){\rnode{a4}{$B$}}  

\rput(25,30){\rnode{c2}{$TA$}}  
\rput(46,30){\rnode{c4}{$TB$}}  
\rput(5,40){\rnode{d1}{$A$}}  
\rput(27,40){\rnode{d3}{$B$}}

\ncline{->}{d1}{d3} \naput{{\scriptsize $f$}} % top face
\ncline{->}{d3}{c4} \naput{{\scriptsize $\eta_B$}}
\ncline{->}{d1}{c2} \nbput[npos=0.55]{{\scriptsize $\eta_A$}}
\ncline{->}{c2}{c4} \naput{{\scriptsize $Tf$}} 

\ncline{->}{d1}{a2} \nbput{{\scriptsize $1$}}
\ncline{->}{c2}{a2} \naput{{\scriptsize $a$}} 

\ncline{->}{c4}{a4} \naput{{\scriptsize $b$}} % right
\ncline{->}{a2}{a4} \nbput[labelsep=3pt]{{\scriptsize $f$}}

% 2-cells

%top 
{
\rput[c](24,35){\scr $=$}
% {\psset{unit=1mm,doubleline=true,arrowinset=0.6,arrowlength=0.5,arrowsize=0.5pt 2.1,nodesep=0pt,labelsep=1pt}
% \pcline{->}(3,3)(0,0) \naput{{\scriptsize $1_F.\phi$}}}
}

%left
{
\rput[c](19,26){\scr $=$}
% {\psset{unit=1mm,doubleline=true,arrowinset=0.6,arrowlength=0.5,arrowsize=0.5pt 2.1,nodesep=0pt,labelsep=2pt}
% \pcline{->}(3,3)(0,0) \naput{{\scriptsize {\bfseries r}}}}
}

%front
{
\rput[c](34,20){\psset{unit=1mm,doubleline=true,arrowinset=0.6,arrowlength=0.5,arrowsize=0.5pt 2.1,nodesep=0pt,labelsep=2pt}
\pcline{->}(3,3)(0,0) \naput{{\scriptsize $\tau$}}}}

}

\rput(30,23){$=$}

\rput(35,0){

% d1 d3
%       c4
% 
%   a2  a4

\rput(25,12){\rnode{a2}{$A$}}  
\rput(46,12){\rnode{a4}{$B$}}  

% \rput(20,30){\rnode{c2}{$TA$}}  
\rput(46,30){\rnode{c4}{$TB$}}  
\rput(5,40){\rnode{d1}{$A$}}  
\rput(28,40){\rnode{d3}{$B$}}

\ncline{->}{d1}{d3} \naput{{\scriptsize $f$}} % top face
\ncline{->}{d3}{c4} \naput{{\scriptsize $\eta_B$}}
% \ncline{->}{d1}{c2} \naput{{\scriptsize $\eta_A$}}
% \ncline{->}{c2}{c4} \naput{{\scriptsize $Tf$}} 

\ncline{->}{d1}{a2} \nbput{{\scriptsize $1$}}
% \ncline{->}{c2}{a2} \naput{{\scriptsize $a$}} 

\ncline{->}{c4}{a4} \naput{{\scriptsize $b$}} % right
\ncline{->}{a2}{a4} \nbput[labelsep=3pt]{{\scriptsize $f$}}
\ncline{->}{d3}{a4} \nbput{{\scriptsize $1$}}

% 2-cells

%right
{
\rput[c](41,26){\scr $=$}
% {\psset{unit=1mm,doubleline=true,arrowinset=0.6,arrowlength=0.5,arrowsize=0.5pt 2.1,nodesep=0pt,labelsep=1pt}
% \pcline{->}(3,3)(0,0) \naput{{\scriptsize $\mbox{{\bfseries a}}'$}}}
}

\rput[c](25,27){{\scriptsize $=$}}
% \rput[c](23,25){{\scriptsize\textsf{naturality of $a$} }}
}

\endpspicture\]

\end{definition}

\begin{example}
Let $T$ be the 2-monad whose strict algebras are weak monoidal categories. Then the weak maps are weak monoidal functors but expressed slightly differently, with \emph{a priori} a coherence map for every parenthesised word, not just binary and nullary words. That is, not only do we have the usual specified coherence maps
\[Fx \otimes Fy \tra F(x\otimes y)\]
and
\[I \tra FI\]
but also coherence maps such as
\[Fx \otimes (Fy \otimes Fz) \tra F\big(x \otimes (y \otimes z)\big).\]
However, the axioms for a weak map of algebras ensure that the coherence constraints for ``non-standard'' arities must be built from the binary and nullary ones in the usual way, so that these weak maps coincide with the usual biased definition of weak monoidal functor.
\end{example}

We now consider 2-monads $S, T$ with a distributive law, and characterise weak maps of $TS$-algebras in a manner similar to the strict maps. However, as the maps are now weak a new axiom is needed, governing the interaction between the constraint cells.

\begin{theorem}\label{weakTSmap}
Let $S$ and $T$ be 2-monads on a 2-category $\cC$, and let $\lambda\: ST \Tra TS$ be a  2-categorical distributive law. Then a weak map of $TS$-algebras
\[\psset{nodesep=2pt}
\left(\pspicture(-3,8)(14,15)
\rput(0,0){
\rput(0,14){\rnode{a1}{$TA$}}  %  top
\rput(0,3){\rnode{a2}{$A$}}  % bottom
\ncline{->}{a1}{a2} \naput{{\scriptsize $a_t$}} % top
}
\rput(6,8){,}
\rput(10,0){
\rput(0,14){\rnode{a1}{$SA$}}  %  top
\rput(0,3){\rnode{a2}{$A$}}  % bottom
\ncline{->}{a1}{a2} \naput{{\scriptsize $a_s$}} % top
}
\endpspicture\right)
\ltra
\psset{nodesep=2pt}
\left(\pspicture(-3,8)(14,17)
\rput(0,0){
\rput(0,14){\rnode{a1}{$TB$}}  %  top
\rput(0,3){\rnode{a2}{$B$}}  % bottom
\ncline{->}{a1}{a2} \naput{{\scriptsize $b_t$}} % top
}
\rput(6,8){,}
\rput(10,0){
\rput(0,14){\rnode{a1}{$SB$}}  %  top
\rput(0,3){\rnode{a2}{$B$}}  % bottom
\ncline{->}{a1}{a2} \naput{{\scriptsize $b_s$}} % top
}
\endpspicture\right)
\]

\noi is given by a 1-cell $A \tmap{f} B$ and 2-cells as below giving individual weak maps:

\[\psset{unit=0.08cm,labelsep=0pt,nodesep=3pt}
\pspicture(0,-6)(70,22)

% a1 a2
% b1 b2

\rput(0,0){
\rput(0,20){\rnode{a1}{$TA$}} % top left
\rput(20,20){\rnode{a2}{$TB$}} % top right

\rput(0,0){\rnode{b1}{$A$}}   % bottom left
\rput(20,0){\rnode{b2}{$B$}}  % bottom right

\psset{nodesep=3pt,labelsep=2pt,arrows=->}
\ncline{a1}{a2}\naput{{\scriptsize $Tf$}} % top
\ncline{b1}{b2}\nbput{{\scriptsize $f$}} % bottom
\ncline{a1}{b1}\nbput{{\scriptsize $a_t$}} % left
\ncline{a2}{b2}\naput{{\scriptsize $b_t$}} % right

\psset{labelsep=1.5pt}
\pnode(13,13){a3}
\pnode(7,7){b3}
\ncline[doubleline=true,arrowinset=0.6,arrowlength=0.8,arrowsize=0.5pt 2.1]{a3}{b3} \nbput[npos=0.4]{{\scriptsize $\tau$}}
\naput[npos=0.4]{{\scriptsize $\sim$}}
}

\rput(50,0){
\rput(0,20){\rnode{a1}{$SA$}} % top left
\rput(20,20){\rnode{a2}{$SB$}} % top right

\rput(0,0){\rnode{b1}{$A$}}   % bottom left
\rput(20,0){\rnode{b2}{$B$}}  % bottom right

\psset{nodesep=3pt,labelsep=2pt,arrows=->}
\ncline{a1}{a2}\naput{{\scriptsize $Sf$}} % top
\ncline{b1}{b2}\nbput{{\scriptsize $f$}} % bottom
\ncline{a1}{b1}\nbput{{\scriptsize $a_s$}} % left
\ncline{a2}{b2}\naput{{\scriptsize $b_s$}} % right

\psset{labelsep=1.5pt}
\pnode(13,13){a3}
\pnode(7,7){b3}
\ncline[doubleline=true,arrowinset=0.6,arrowlength=0.8,arrowsize=0.5pt 2.1]{a3}{b3} \nbput[npos=0.4]{{\scriptsize $\sigma$}}
\naput[npos=0.4]{{\scriptsize $\sim$}}
}

\endpspicture\]

\noi such that the following ``interaction axiom'' holds:

\[\psset{unit=0.1cm,labelsep=1pt,nodesep=2pt}
\pspicture(0,10)(40,60)

% e0  e2
%   d1  d3
%     c2  c4
% b1
%     a2  a4

\rput(26,27){$=$}

\rput(-25,0){

\rput(20,12){\rnode{a2}{$A$}}  
\rput(42,12){\rnode{a4}{$B$}}  
\rput(-8,30){\rnode{b1}{$SA$}}  
\rput(20,30){\rnode{c2}{$TA$}}  
\rput(42,30){\rnode{c4}{$TB$}}  
\rput(6,38){\rnode{d1}{$TSA$}}  
\rput(28,38){\rnode{d3}{$TSB$}}  

\rput(-8,46){\rnode{e0}{$STA$}}  
\rput(14,46){\rnode{e2}{$STB$}}

\ncline{->}{d1}{d3} \naput{{\scriptsize $TSf$}} % top face
\ncline{->}{d3}{c4} \naput{{\scriptsize $Tb_s$}}
\ncline{->}{d1}{c2} \nbput[labelsep=0pt]{{\scriptsize $Ta_s$}}
\ncline{->}{c2}{c4} \nbput{{\scriptsize $Tf$}} 

\ncline{->}{e0}{b1} \nbput{{\scriptsize $Sa_t$}}
\ncline{->}{b1}{a2} \nbput{{\scriptsize $a_s$}}
\ncline{->}{c2}{a2} \naput{{\scriptsize $a_t$}} 

\ncline{->}{c4}{a4} \naput{{\scriptsize $b_t$}} % right
\ncline{->}{a2}{a4} \nbput[labelsep=3pt]{{\scriptsize $f$}}

\ncline{->}{e0}{e2} \naput{{\scriptsize $STf$}} 
\ncline{->}{e0}{d1} \nbput[labelsep=0pt]{{\scriptsize $\lambda_A$}} 
\ncline{->}{e2}{d3} \naput{{\scriptsize $\lambda_B$}}

% 2-cells

%top 
{
\rput[c](10,42){\scr $=$}
% {\psset{unit=1mm,doubleline=true,arrowinset=0.6,arrowlength=0.5,arrowsize=0.5pt 2.1,nodesep=0pt,labelsep=1pt}
% \pcline{-}(3,3)(0,0) \naput{{\scriptsize }}}
}

{
\rput[c](23,32.5)
{\psset{unit=1mm,doubleline=true,arrowinset=0.6,arrowlength=0.5,arrowsize=0.5pt 2.1,nodesep=0pt,labelsep=1pt}
\pcline{->}(2.5,2.5)(0,0) \naput[npos=0.35]{{\scriptsize $T\sigma$}}}
}

%left
{
\rput[c](6,29){\scr $=$}
% {\psset{unit=1mm,doubleline=true,arrowinset=0.6,arrowlength=0.5,arrowsize=0.5pt 2.1,nodesep=0pt,labelsep=2pt}
% \pcline{-}(3,3)(0,0) \naput{{\scriptsize }}}
}

%front
{
\rput[c](30,20){\psset{unit=1mm,doubleline=true,arrowinset=0.6,arrowlength=0.5,arrowsize=0.5pt 2.1,nodesep=0pt,labelsep=2pt}
\pcline{->}(3,3)(0,0) \naput{{\scriptsize $\tau$}}}}
}

\rput(34,10){

% e0  e2
%       d3
%         c4
% b1  b2
%     a2  a4

\rput(19,12){\rnode{a2}{$A$}}  
\rput(42,12){\rnode{a4}{$B$}}  
\rput(-8,30){\rnode{b1}{$SA$}}  
\rput(14,30){\rnode{b2}{$SB$}}  

\rput(42,30){\rnode{c4}{$TB$}}  
\rput(28,38){\rnode{d3}{$TSB$}}  

\rput(-8,46){\rnode{e0}{$STA$}}  
\rput(14,46){\rnode{e2}{$STB$}}

% \ncline{->}{d1}{d3} \naput{{\scriptsize $T^2f$}} % top face
\ncline{->}{d3}{c4} \naput{{\scriptsize $Tb_s$}}
% \ncline{->}{d1}{c2} \naput{{\scriptsize $\mu_A$}}
% \ncline{->}{c2}{c4} \naput{{\scriptsize $Tf$}} 

\ncline{->}{e0}{b1} \nbput{{\scriptsize $Sa_t$}}
\ncline{->}{b1}{a2} \nbput{{\scriptsize $a_s$}}
\ncline{->}{e2}{b2} \naput{{\scriptsize $Sb_t$}} 

\ncline{->}{c4}{a4} \naput{{\scriptsize $b_t$}} % right
\ncline{->}{a2}{a4} \nbput[labelsep=3pt]{{\scriptsize $f$}}

\ncline{->}{e0}{e2} \naput{{\scriptsize $STf$}} 
% \ncline{->}{e0}{d1} \naput{{\scriptsize $a$}} 
\ncline{->}{e2}{d3} \naput{{\scriptsize $\lambda_B$}} 

\ncline{->}{b1}{b2} \naput{{\scriptsize $Sf$}} 
\ncline{->}{b2}{a4} \naput{{\scriptsize $b_s$}}

% 2-cells

%right
{
\rput[c](28,29){\scr $=$}
% {\psset{unit=1mm,doubleline=true,arrowinset=0.6,arrowlength=0.5,arrowsize=0.5pt 2.1,nodesep=0pt,labelsep=1pt}
% \pcline{-}(3,3)(0,0) \naput{{\scriptsize }}}
}

%bottom
{
\rput[c](16,20){\psset{unit=1mm,doubleline=true,arrowinset=0.6,arrowlength=0.5,arrowsize=0.5pt 2.1,nodesep=0pt,labelsep=2pt}
\pcline{->}(3,3)(0,0) \naput{{\scriptsize $\sigma$}}}}

%back
{
\rput[c](2,36){\psset{unit=1mm,doubleline=true,arrowinset=0.6,arrowlength=0.5,arrowsize=0.5pt 2.1,nodesep=0pt,labelsep=2pt}
\pcline{->}(3,3)(0,0) \naput{{\scriptsize S$\tau$}}}}

}

\endpspicture\]

\end{theorem}

Note that as everything is strict here except the weak maps of algebras, there is a certain amount of overkill in using fully 2-dimensional pasting diagrams.  In particular, the vast majority of faces in the diagrams are merely (strict) naturality squares. Under such circumstances, string diagram notation is particularly efficacious.  All the proofs of commutativity in this section are entirely routine, and only complicated by the difficulty of notating 2-cells.  Thus, we will defer many of the proofs to the appendix where we will use string diagrams.

\begin{proof}
Consider a weak map
\[\psset{unit=0.08cm,labelsep=0pt,nodesep=3pt}
\pspicture(0,-4)(22,24)

% a1 a2
% b1 b2

\rput(0,20){\rnode{a1}{$TSA$}} % top left
\rput(22,20){\rnode{a2}{$TSB$}} % top right

\rput(0,0){\rnode{b1}{$A$}}   % bottom left
\rput(22,0){\rnode{b2}{$B$}}  % bottom right

\psset{nodesep=3pt,labelsep=2pt,arrows=->}
\ncline{a1}{a2}\naput[npos=0.45]{{\scriptsize $TSf$}} % top
\ncline{b1}{b2}\nbput{{\scriptsize $f$}} % bottom
\ncline{a1}{b1}\nbput{{\scriptsize $\theta_A$}} % left
\ncline{a2}{b2}\naput{{\scriptsize $\theta_B$}} % right

\psset{labelsep=1.5pt}
\pnode(13,13){a3}
\pnode(7,7){b3}
\ncline[doubleline=true,arrowinset=0.6,arrowlength=0.8,arrowsize=0.5pt 2.1]{a3}{b3} \nbput[npos=0.4]{{\scriptsize $\phi$}}
\naput[npos=0.4]{{\scriptsize $\sim$}}

\endpspicture\]

\noi First note that $\eta^TS: S \Tra TS$ and $T\eta^S: T \Tra TS$ give monad functors $TS \Tra S$ and $TS \Tra T$.  Thus we know that $\sigma$ and $\tau$ given as follows are weak maps of the constituent $S$-algebras and $T$-algebras respectively:

\[\psset{unit=0.08cm,labelsep=0pt,nodesep=3pt}
\pspicture(60,42)

\rput(0,20){
\pspicture(20,42)

% a1 a2
% b1 b2
% c1 c2

\rput(0,40){\rnode{a1}{$SA$}} % top left
\rput(24,40){\rnode{a2}{$SB$}} % top mid

\rput(0,20){\rnode{b1}{$TSA$}}   % bottom left
\rput(24,20){\rnode{b2}{$TSB$}}  % bottom mid

\rput(0,0){\rnode{c1}{$A$}}   % bottom left
\rput(24,0){\rnode{c2}{$B$}}  % bottom mid

\psset{nodesep=3pt,labelsep=2pt,arrows=->}
\ncline{a1}{a2}\naput{{\scriptsize $Sf$}} % top
\ncline{b1}{b2}\nbput[npos=0.45]{{\scriptsize $TSf$}} % bottom
\ncline{c1}{c2}\nbput{{\scriptsize $f$}} % bottom

\ncline{a1}{b1}\nbput{{\scriptsize $\eta^T_{SA}$}} % left
\ncline{b1}{c1}\nbput{{\scriptsize $\theta_A$}} % left

\ncline{a2}{b2}\naput{{\scriptsize $\eta^T_{SB}$}} % right
\ncline{b2}{c2}\naput{{\scriptsize $\theta_B$}} % right

\psset{labelsep=1.5pt}

\rput(2,0){
\pnode(13,33){a3}
\pnode(7,27){b3}
\ncline[doubleline=true,arrowinset=0.6,arrowlength=0.8,arrowsize=0.5pt 2.1,nodesep=5pt]{-}{a3}{b3} \naput[npos=0.4]{{\scriptsize $$}}
}

\rput(2,0){
\pnode(13,13){a3}
\pnode(7,7){b3}
\ncline[doubleline=true,arrowinset=0.6,arrowlength=0.8,arrowsize=0.5pt 2.1]{a3}{b3} \naput[npos=0.4]{{\scriptsize $\phi$}}
}

\endpspicture}

\rput(50,20){
\pspicture(20,42)

% a1 a2
% b1 b2
% c1 c2

\rput(-2,40){\rnode{a1}{$TA$}} % top left
\rput(22,40){\rnode{a2}{$TB$}} % top mid

\rput(-2,20){\rnode{b1}{$TSA$}}   % bottom left
\rput(22,20){\rnode{b2}{$TSB$}}  % bottom mid

\rput(-2,0){\rnode{c1}{$A$}}   % bottom left
\rput(22,0){\rnode{c2}{$B$}}  % bottom mid

\psset{nodesep=3pt,labelsep=2pt,arrows=->}
\ncline{a1}{a2}\naput{{\scriptsize $Tf$}} % top
\ncline{b1}{b2}\nbput[npos=0.45]{{\scriptsize $TSf$}} % bottom
\ncline{c1}{c2}\nbput{{\scriptsize $f$}} % bottom

\ncline{a1}{b1}\nbput{{\scriptsize $T\eta^S_{A}$}} % left
\ncline{b1}{c1}\nbput{{\scriptsize $\theta_A$}} % left

\ncline{a2}{b2}\naput{{\scriptsize $T\eta^S_{B}$}} % right
\ncline{b2}{c2}\naput{{\scriptsize $\theta_B$}} % right

\psset{labelsep=1.5pt}

\pnode(13,33){a3}
\pnode(7,27){b3}
\ncline[doubleline=true,arrowinset=0.6,arrowlength=0.8,arrowsize=0.5pt 2.1,nodesep=5pt]{-}{a3}{b3} \naput[npos=0.4]{{\scriptsize $$}}

\pnode(13,13){a3}
\pnode(7,7){b3}
\ncline[doubleline=true,arrowinset=0.6,arrowlength=0.8,arrowsize=0.5pt 2.1]{a3}{b3} \naput[npos=0.4]{{\scriptsize $\phi$}}

\endpspicture}

\endpspicture
\]

\vs{1}

\noi We need to check the interaction axiom; see Appendix.

Conversely given individual weak maps $\sigma$ and $\tau$, we construct the following putative weak map of $TS$-algebras:

 \[\psset{unit=0.08cm,labelsep=0pt,nodesep=3pt}
\pspicture(0,-2)(22,42)

\rput(0,20){
\pspicture(20,42)

% a1 a2
% b1 b2
% c1 c2

\rput(-2,40){\rnode{a1}{$TSA$}} % top left
\rput(22,40){\rnode{a2}{$TSB$}} % top mid

\rput(-2,20){\rnode{b1}{$TA$}}   % bottom left
\rput(22,20){\rnode{b2}{$TB$}}  % bottom mid

\rput(-2,0){\rnode{c1}{$A$}}   % bottom left
\rput(22,0){\rnode{c2}{$B$}}  % bottom mid

\psset{nodesep=3pt,labelsep=2pt,arrows=->}
\ncline{a1}{a2}\naput[npos=0.45]{{\scriptsize $TSf$}} % top
\ncline{b1}{b2}\nbput{{\scriptsize $Tf$}} % bottom
\ncline{c1}{c2}\nbput{{\scriptsize $f$}} % bottom

\ncline{a1}{b1}\nbput{{\scriptsize $Ta_s$}} % left
\ncline{b1}{c1}\nbput{{\scriptsize $a_t$}} % left

\ncline{a2}{b2}\naput{{\scriptsize $Tb_s$}} % right
\ncline{b2}{c2}\naput{{\scriptsize $b_t$}} % right

\psset{labelsep=1.5pt}

\pnode(13,33){a3}
\pnode(7,27){b3}
\ncline[doubleline=true,arrowinset=0.6,arrowlength=0.8,arrowsize=0.5pt 2.1]{->}{a3}{b3} \naput[npos=0.4]{{\scriptsize $T\sigma$}}

\pnode(13,13){a3}
\pnode(7,7){b3}
\ncline[doubleline=true,arrowinset=0.6,arrowlength=0.8,arrowsize=0.5pt 2.1]{a3}{b3} \naput[npos=0.4]{{\scriptsize $\tau$}}

\endpspicture}

\endpspicture
\]
We need to check the two axioms for an algebra; see Appendix.

% ***these calculations and this entire calculus probably needs to go in an Appendix.
% 

\end{proof}

\subsection{Transformations of algebras via distributive laws}
\label{transdl}

\begin{definition}
Let $T$ be a 2-monad on a 2-category $\cC$. Given weak maps of $T$-algebras

\[\psset{unit=0.08cm,labelsep=0pt,nodesep=2pt}
\pspicture(0,-3)(60,23)

% a1 a2
% b1 b2

\rput(0,0){

\rput(0,20){\rnode{a1}{$TA$}} % top left
\rput(20,20){\rnode{a2}{$TB$}} % top right

\rput(0,0){\rnode{b1}{$A$}}   % bottom left
\rput(20,0){\rnode{b2}{$B$}}  % bottom right

\psset{nodesep=3pt,labelsep=2pt,arrows=->}
\ncline{a1}{a2}\naput{{\scriptsize $Tf$}} % top
\ncline{b1}{b2}\nbput{{\scriptsize $f$}} % bottom
\ncline{a1}{b1}\nbput{{\scriptsize $a$}} % left
\ncline{a2}{b2}\naput{{\scriptsize $b$}} % right

\psset{labelsep=1.5pt}
\pnode(13,13){a3}
\pnode(7,7){b3}
\ncline[doubleline=true,arrowinset=0.6,arrowlength=0.8,arrowsize=0.5pt 2.1]{a3}{b3} \nbput[npos=0.4]{{\scriptsize $\tau_f$}}
}

\rput(40,0){

\rput(0,20){\rnode{a1}{$TA$}} % top left
\rput(20,20){\rnode{a2}{$TB$}} % top right

\rput(0,0){\rnode{b1}{$A$}}   % bottom left
\rput(20,0){\rnode{b2}{$B$}}  % bottom right

\psset{nodesep=3pt,labelsep=2pt,arrows=->}
\ncline{a1}{a2}\naput{{\scriptsize $Tg$}} % top
\ncline{b1}{b2}\nbput{{\scriptsize $g$}} % bottom
\ncline{a1}{b1}\nbput{{\scriptsize $a$}} % left
\ncline{a2}{b2}\naput{{\scriptsize $b$}} % right

\psset{labelsep=1.5pt}
\pnode(13,13){a3}
\pnode(7,7){b3}
\ncline[doubleline=true,arrowinset=0.6,arrowlength=0.8,arrowsize=0.5pt 2.1]{a3}{b3} \nbput[npos=0.4]{{\scriptsize $\tau_g$}}
}

\endpspicture\]
a transformation consists of a 2-cell
\[
\psset{unit=0.07cm,labelsep=2pt,nodesep=1pt}
\pspicture(0,-3)(20,8)

\rput(0,0){\rnode{a1}{$A$}}  % named node, with something placed there
\rput(20,0){\rnode{a2}{$B$}}  % named node, with something placed there

\ncarc[arcangle=45]{->}{a1}{a2}\naput{{\scriptsize $f$}}
\ncarc[arcangle=-45]{->}{a1}{a2}\nbput{{\scriptsize $g$}}

{
\rput[c](10,-2){\psset{unit=1mm,doubleline=true,arrowinset=0.6,arrowlength=0.5,arrowsize=0.5pt 2.1,nodesep=0pt,labelsep=2pt}
\pcline{->}(0,3)(0,0) \naput{{\scriptsize $\alpha$}}}}

\endpspicture
\]
such that the following cylinder diagram commutes.
\[
\psset{unit=0.06cm,labelsep=2pt,nodesep=2pt}
\pspicture(0,-9)(80,38)

% a1 a2
%    
%
% b1 b2

\rput(0,0){

\rput(0,28){\rnode{a1}{$TA$}}  % named node, with something placed there
\rput(30,28){\rnode{a2}{$TB$}}  

\rput(0,3){\rnode{b1}{$A$}}  
\rput(30,3){\rnode{b2}{$B$}}

\ncarc[arcangle=45]{->}{a1}{a2}\naput[npos=0.5]{{\scriptsize $Tf$}}
\ncarc[arcangle=-45]{->}{a1}{a2}\nbput[npos=0.5]{{\scriptsize $Tg$}}

\ncline{->}{a1}{b1} \nbput{{\scriptsize $a$}}
\ncline{->}{a2}{b2} \naput{{\scriptsize $b$}}

\ncarc[arcangle=-45]{->}{b1}{b2}\nbput[npos=0.5]{{\scriptsize $g$}}

{
\rput[c](16,28){\psset{unit=1mm,doubleline=true,arrowinset=0.6,arrowlength=0.5,arrowsize=0.5pt 2.1,nodesep=0pt,labelsep=2pt}
\pcline{->}(0,1.5)(0,-1.5) \nbput{{\scriptsize $T\alpha$}}}}

{
\rput[c](15,6.5){\psset{unit=1mm,doubleline=true,arrowinset=0.6,arrowlength=0.5,arrowsize=0.5pt 2.1,nodesep=0pt,labelsep=1pt}
\pcline{->}(1.5,1.5)(-1.5,-1.5) \nbput{{\scriptsize $\tau_g$}}}}

% \ncarc[arcangle=45]{->}{a1}{a2}\naput{{\scriptsize $$}}
% \ncarc[arcangle=-45]{->}{a1}{a2}\nbput{{\scriptsize $$}}

\rput(44,20){$=$}

}

\rput(60,0){

\rput(0,28){\rnode{a1}{$TA$}}  % named node, with something placed there
\rput(30,28){\rnode{a2}{$TB$}}  

\rput(0,3){\rnode{b1}{$A$}}  
\rput(30,3){\rnode{b2}{$B$}}

\ncarc[arcangle=45]{->}{a1}{a2}\naput[npos=0.5]{{\scriptsize $Tf$}}
% \ncarc[arcangle=-45]{->}{a1}{a2}\nbput[npos=0.5]{{\scriptsize $Tg$}}

\ncline{->}{a1}{b1} \nbput{{\scriptsize $a$}}
\ncline{->}{a2}{b2} \naput{{\scriptsize $b$}}

% {
% \rput[c](17,28.5){\psset{unit=1mm,doubleline=true,arrowinset=0.6,arrowlength=0.5,arrowsize=0.5pt 2.1,nodesep=0pt,labelsep=2pt}
% \pcline{->}(0,3)(0,0) \nbput{{\scriptsize $\alpha$}}}}
% 
% 
% {
% \rput[c](17,8.5){\psset{unit=1mm,doubleline=true,arrowinset=0.6,arrowlength=0.5,arrowsize=0.5pt 2.1,nodesep=0pt,labelsep=2pt}
% \pcline{->}(3,3)(0,0) \nbput{{\scriptsize $\tau_f$}}}}

\ncarc[arcangle=45]{->}{b1}{b2}\naput[npos=0.5]{{\scriptsize $f$}}
\ncarc[arcangle=-45]{->}{b1}{b2}\nbput[npos=0.5]{{\scriptsize $g$}}

{
\rput[c](15,23){\psset{unit=1mm,doubleline=true,arrowinset=0.6,arrowlength=0.5,arrowsize=0.5pt 2.1,nodesep=0pt,labelsep=1pt}
\pcline{->}(1.5,1.5)(-1.5,-1.5) \nbput{{\scriptsize $\tau_f$}}}}

{
\rput[c](15,3){\psset{unit=1mm,doubleline=true,arrowinset=0.6,arrowlength=0.5,arrowsize=0.5pt 2.1,nodesep=0pt,labelsep=2pt}
\pcline{->}(0,1.5)(0,-1.5) \nbput{{\scriptsize $\alpha$}}}}

}
% \ncarc[arcangle=45]{->}{a1}{a2}\naput{{\scriptsize $$}}
% \ncarc[arcangle=-45]{->}{a1}{a2}\nbput{{\scriptsize $$}}

\endpspicture
\]

\noi We will refer to this as a $T$-algebra transformation for short, or just $T$-transformation.  $T$-algebras, weak $T$-maps and $T$-transformations form a 2-category which we will call \twocatt{$T\cat{-Alg}_w$}.

\end{definition}

\begin{example} 
Let $T$ be the 2-monad for weak monoidal categories. Then a transformation between weak maps of $T$-algebras is a monoidal transformation. As for the weak maps, this is expressed slightly differently from the usual biased way, in that there is now a monoidality axiom for all parenthesised words and not just nullary/binary ones. However, again the concepts are the same because in the biased presentation the axioms for other arities can be derived from the nullary/binary ones. This gives us \twocatt{$T\cat{-Alg}_w$} the 2-category of monoidal categories, weakly monoidal functors, and monoidal transformations.
 \end{example}

\begin{theorem}\label{TStrans}
Let $S$ and $T$ be 2-monads on a 2-category $\cC$, and let $\lambda\: ST \Tra TS$ be a  2-categorical distributive law.  Then a transformation $\alpha$ of $TS$-algebras is precisely a 2-cell $\alpha$ that is both a transformation of $S$-algebras and a transformation of $T$-algebras (with no further axiom).

\end{theorem}

\begin{proof}

Consider a transformation of $TS$-algebras.  As above, the monad functors $TS \Tra S$ and $TS \Tra T$ give us a transformation of $S$-algebras and $T$-algebras respectively.  Conversely suppose we have a 2-cell $\alpha$ that satisfies both the cylinder for $T$ and the cylinder for $S$; it is straightforward to check it then satisfies the cylinder for $TS$.

% ***pic 10**

\end{proof}

%%%%%%%%%%%%%%%%%%%%%%%%%%%%%%%%%%%%%%%%%%%%%%%%%% end paste from EH Section 3

%%%%%%%%%%%%%%%%%%%%%%%%%%%%%%%%%%%%%%%%%%%%%%%%% 
% Section 3
%%%%%%%%%%%%%%%%%%%%%%%%%%%%%%%%%%%%%%%%%%%%%%%%%

\section[Weak maps and transformations of doubly-degenerate \newline \bicats-categories]{Weak maps and transformations of \texorpdfstring{\\}{} doubly-degenerate \bicats-categories}

\label{three}

In this section we unravel the definitions of weak map and transformation in our case of interest, and go on to characterise these with reference only to the vertical monoidal structures. This is the technical content of the comparison with braided monoidal categories.

\subsection{Weak maps by definition}

In this section we unravel the definition of weak map in the case of \ddbicatscats. 

% It is important to note that we are using a definition of weak map that is not fully weak.  A fully weak map of \ddbicatscats\ would include functoriality constraints for composition of 1-cells.  However, \cite{cg3} showed that this is the ``wrong'' notion for doubly-degenerate situations, because even if there is only one 0-cell and one 1-cell any constraint 2-cell would remain as a distinguished invertible element; that is, when we perform the dimension shift, the old 2-cells become 0-cells of a new lower-dimensional structure (in our case a putative braided monoidal category), and the constraint 2-cells would become distinguished invertible 0-cells, adding unwanted extra structure to our braided monoidal category.
% 
% We can eliminate that issue by using stricter functors, specifically, functors that are strictly functorial on any dimension of cell that is going to become degenerate. In this case, that means we want functors that are strictly functorial with respect to 1-cell composition, and only have functoriality constraint with respect to 2-cell composition.  This can be effected technically by using icons, because the existence of an iconic 2-cell constraint encodes an assertion that the source and target agree on objects.   ***fix this and maybe put it in the introduction**
% 
% \bc
% \igw{8cm}{weakmap-note.eps}
% \ec
% 

Recall (Section~\ref{examplebicatscat}) that we construct \bicatscats\ from 2-monads $V$ for vertical composition and $H$ for horizontal composition, equipped with a distributive law of $V$ over $H$. These are monads on the 2-category \cat{Cat-Gph-Gph}, so in the \dd\ case the underlying data is just a category.  Then
\begin{itemize}
\item a $V$-algebra structure is a (weak) monoidal structure, and
\item an $H$-algebra structure is a (strict) monoidal structure.
\end{itemize}

% **comment here about what weak algebra maps are in that case, something to do with binary constraints being enough ie technically a weak map here gives a functoriality constraint for every arity, but coherence tells us that binary and nullary is enough and we can constuct the rest**

\noi As in \wvc\ we will write the $V$-algebra (weak monoidal) structure vertically as $\frac{a}{b}$ and the $H$-algebra (strict monoidal) structure horizontally as $a\|b$. As we have strict interchange we can combine these in grids without ambiguity, for example
\[\begin{array}{c|c}
a & b \\
\myhline c& d
\end{array}\]
Furthermore, we will often ignore issues of associativity in the vertical direction because, as long as we are not simultaneously interacting with associativity in the horizontal direction, coherence means that the diagram of vertical associators will commute.  (We know from \cite{koc3} that care must be taken about the interaction between horizontal and vertical associativity.) Thus we will work with larger grids such as
\[\begin{array}{c|c}
a & b \\
\myhline c& d \\
\myhline e& f
\end{array}\]
but will only increase the height when the width remains at two. By interchange we know that the horizontal and vertical unit are the same, and we write it as 1; we will sometimes write it as an empty space in a grid formation.  We will write the left and right unit constraints for the horizontal tensor product as follows (with $\lambda$ and $\rho$ for ``left'' and ``right''):
\[\begin{array}{rcl}
\htp{1}{a} & \mtmap{\lambda} & a \\
\htp{a}{1} & \mtmap{\rho} & a
\end{array}\]
and those for the vertical tensor product as follows (with $\tau$ and $\beta$ for ``top'' and ``bottom''):
\[\begin{array}{rcl}
\vtp{1}{a} & \mtmap{\tau} & a \\
\vtp{a}{1} & \mtmap{\beta} & a
\end{array}\]

\begin{definition}
We define a weak map of \ddbicatscats\ to be a weak map of their $HV$-algebra structures.
\end{definition}

\noi We can now use the results of the previous section to characterise these weak maps via the horizontal and vertical monoidal structures.

\begin{proposition}
A weak map $X\tmap{F} Y$ of \ddbicatscats\ is a functor on the underlying categories equipped with:

\begin{itemize}

\item A vertical monoidality constraint: for all $a, b \in X$ an isomorphism
\[\frac{Fa}{Fb} \mtmapiso{v_{ab}} F\left(\frac{a}{b}\right) \]
natural in $a$ and $b$; we will usually omit the subscripts.

\item A horizontal monoidality constraint: for all $a, b \in X$ an isomorphism
\[Fa \| Fb \mtmapiso{h_{ab}} F(a\|b) \]
natural in $a$ and $b$; again we will usually omit the subscripts.

\item A unit constraint: an isomorphism $1 \tmapiso{\eta} F1$. 

\end{itemize}
These must satisfy the usual axioms for monoidal functors, plus the interaction axiom:

\[\setlength{\arraycolsep}{3pt}
\psset{unit=0.08cm,labelsep=3pt,nodesep=5pt}
\pspicture(-10,-24)(10,24)

% a1 a2
% a3 a4

%%%%%%%%%% top

\rput(-20,16){\rnode{a1}{$\begin{array}{c|c}
Fa & Fb \\
\myhline Fc & Fd
\end{array}
$}}  % top left

\rput(20,16){\rnode{a2}{$\begin{array}{c}
F(a\|b) \\
\myhline F(c\|d)
\end{array}
$}}  % top right

\rput(-20,-16){\rnode{a3}{$\begin{array}{c|c}
F\left(\frac{a}{c}\right) & F\left(\frac{b}{d}\right) 
\end{array}
$}}  % bottom left

\rput(20,-16){\rnode{a4}{$F\left(\begin{array}{c|c}
a & b \\
\myhline c& d
\end{array}\right)
$}}  % bottom right

\ncline{->}{a1}{a2} \naput{{\scriptsize $\frac{h}{h}$}} % top
\ncline{->}{a3}{a4} \nbput{{\scriptsize $h$}} % bottom
\ncline{->}{a1}{a3} \nbput{{\scriptsize $v\|v$}} % left
\ncline{->}{a2}{a4} \naput{{\scriptsize $v$}} % right

\endpspicture
\]
\end{proposition}

Note that is essentially a ``biased'' presentation of the interaction axiom in the previous section.

\begin{proof}

By Theorem~\ref{weakTSmap} we know that a weak $HV$-map is equivalently a map with the structure of both a weak $H$-map and a weak $V$-map satisfying the interaction axiom.  Thus we get vertical monoidality constraints, horizontal monoidality constraint, and a general interaction axiom (for all arities in both directions). The general interaction axiom tells us in particular that the horizontal and vertical unit constraints $1 \tra F1$ must coincide, and that the 2-by-2 interaction axiom above must hold. Conversely, starting from the interaction axiom for arity 0 and 2-by-2, we can prove the general interaction axiom by double induction. 
\end{proof}

\subsection{Weak maps in terms of vertical structure}

We are now going to use a weak Eckmann--Hilton argument to re-characterise these weak maps further, eliminating the reference to $H$. First recall from \wvc\ that a \ddbicatscat\ has the structure of a \bmc\ with respect to its vertical tensor product, with braiding given by a weak Eckmann--Hilton argument as follows:
\[\psset{unit=1.1mm}
\pspicture(0,-5)(0,85)

\rput(0,80){\rnode{a1}{\vtp{a}{b}}}

\rput(0,60){\rnode{a2}{\fourtp{}{a}{b}{}}}

\rput(0,40){\rnode{a3}{\htp{b}{a}}}

\rput(0,20){\rnode{a4}{\fourtp{b}{}{}{a}}}

\rput(0,0){\rnode{a5}{\vtp{b}{a}}}

\psset{arrows=->, nodesep=6pt,labelsep=4pt}
\ncline[nodesepA=3pt, arrows=-, doubleline=true]{a1}{a2} \naput[labelsep=4pt]{\parbox{10em}{\sf\scr strict horizontal units\\ strict interchange}}
\ncline{a2}{a3} \naput{\sf\scr weak vertical units}\nbput[nrot=:U,labelsep=2pt]{\scr $\sim$}
\ncline{a3}{a4} \naput{\sf\scr weak vertical units}\nbput[nrot=:U,labelsep=2pt]{\scr $\sim$}
\ncline[nodesepB=4pt, arrows=-, doubleline=true]{a4}{a5} \naput[labelsep=4pt]{\parbox{10em}{\sf\scr strict horizontal units\\ strict interchange}}

\endpspicture\]
Note that this involves a choice of orientation; we will keep the above ``clockwise'' orientation throughout. As we will invoke this repeatedly we will express it in terms of the following maps:
\[\alpha\: \htp{a}{b} \ltra \vtp{a}{b}\]
and
\[\ol\alpha\: \htp{a}{b} \ltra \vtp{b}{a}\]
We can think of $\alpha$ as being clockwise and $\ol\alpha$ as anti-clockwise.  Then the braiding in the orientation we have chosen above is given by
\[\gamma := \hs{0.4} \vtp{a}{b} \ltmap{\ol\alpha^{-1}} \htp{b}{a} \ltmap{\alpha} \vtp{b}{a}\]
We will refer to this as the \emph{standard braiding}.

Note that $\alpha$ and $\ol\alpha$ are built from unit constraints so are natural in $a$ and $b$.  Another key result we need encapsulates the usual way in which we extract a braiding from interchange.

\begin{proposition}\label{alphakey}
The following diagram commutes.

\[
\psset{unit=0.1cm,labelsep=2pt,nodesep=3pt,npos=0.4}
\pspicture(0,-30)(30,30)
%    a2
% a1 
%    a3 

\rput[B](0,0){\Rnode{a1}{\fourtp{a}{b}{c}{d}}}  %  left
\rput[B](30,20){\Rnode{a2}{\fourvtp{a}{b}{c}{d}}}  % top right
\rput[B](30,-20){\Rnode{a3}{\fourvtp{a}{c}{b}{d}}}  % bottom right

\ncline[nodesepB=6pt]{->}{a1}{a2} \naput[npos=0.5,labelsep=0pt]{{\scr \vtp{\alpha_{ab}}{\alpha_{cd}}}} % top
\ncline[nodesepB=6pt]{->}{a1}{a3} \nbput[npos=0.5]{{\scr $\alpha_{\scalebox{0.7}{${\left(\vtp{a}{c}\right)\hs{-0.3}\left(\vtp{b}{d}\right)}$}}$}} % left
\ncline[labelsep=2pt]{->}{a2}{a3} \naput{\scr \threevtp{1_a}{\gamma_{bc}}{1_d}}

\endpspicture
\]

\end{proposition}

\begin{proof}
The diagram in question is the outside of the diagram below.
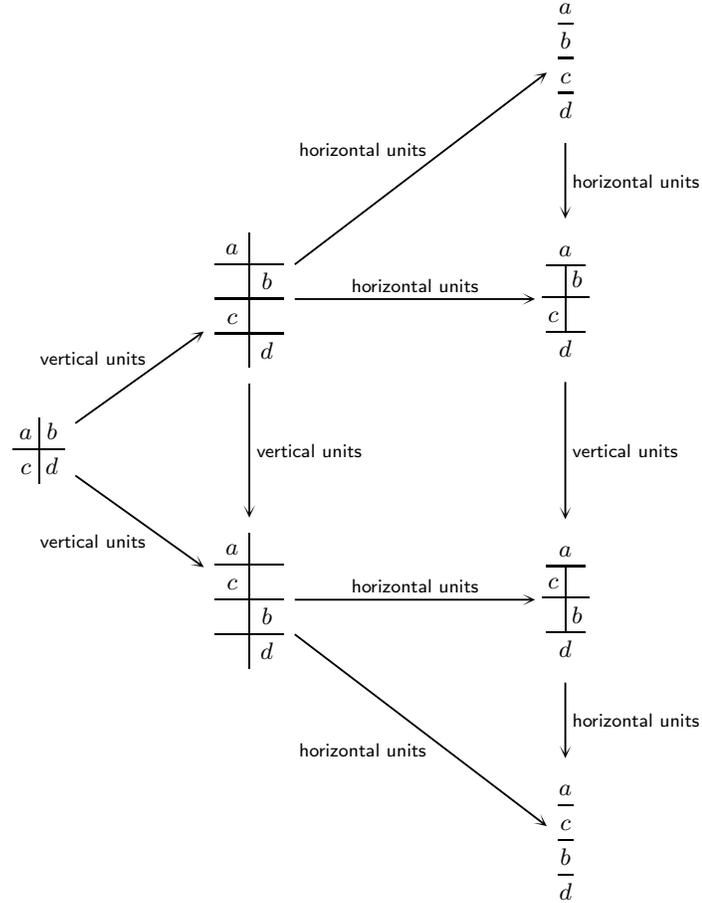
\[\begin{small}
\psset{unit=0.1cm,labelsep=2pt,nodesep=6pt,npos=0.5,arrows=->}
\pspicture(0,-60)(70,60)
%          a2
%    b1    b2
% a1 
%    c1    c2
%          a3 

\rput[B](0,0){\Rnode{a1}{\fourtp{a}{b}{c}{d}}}  %  left
\rput[B](28,20){\Rnode{b1}{\eighttp{a}{}{}{b}{c}{}{}{d}}}  % top right
\rput[B](70,20){\Rnode{b2}{\onefouronetp{a}{}{b}{c}{}{d}}}  % top right

\rput[B](70,52){\Rnode{a2}{\fourvtp{a}{b}{c}{d}}}  % top right

\rput[B](70,-52){\Rnode{a3}{\fourvtp{a}{c}{b}{d}}}  % bottom right

\rput[B](28,-20){\Rnode{c1}{\eighttp{a}{}{c}{}{}{b}{}{d}}}  % bottom right

\rput[B](70,-20){\Rnode{c2}{\onefouronetp{a}{c}{}{}{b}{d}}}  % bottom right

\ncline{a1}{b1} \naput{\scr\sf vertical units}
\ncline{b1}{a2} \naput{\scr\sf horizontal units}
\ncline{b1}{b2}  \naput{\scr\sf horizontal units}
\ncline{a2}{b2} \naput{\scr\sf horizontal units}
\ncline{a1}{c1} \nbput{\scr\sf vertical units}
\ncline{c1}{c2} \naput{\scr\sf horizontal units}
\ncline{b2}{c2}  \naput{\scr\sf vertical units}
\ncline{b1}{c1} \naput{\scr\sf vertical units}
\ncline{c1}{a3} \nbput{\scr\sf horizontal units}
\ncline{c2}{a3} \naput{\scr\sf horizontal units}

% \ncline{->}{a1}{a2} \naput[npos=0.5,labelsep=0pt]{{\scr \vtp{\alpha_{ab}}{\alpha_{cd}}}} % top
% \ncline{->}{a1}{a3} \nbput[npos=0.5]{{\scr $\alpha$}} % left
% \ncline[labelsep=2pt]{->}{a2}{a3} \naput{\scr \threevtp{1}{\gamma_{bc}}{1}}

\endpspicture
\end{small}
\]
We see that each triangle involves only one type of unit constraint and so commutes by coherence. The square commutes by functoriality of vertical tensor product. 

\end{proof}

We are now ready to characterise the weak maps of $HV$-algebras in terms of just the vertical structure, that is, in terms of $V$-structure and the braiding.  In some sense the following proposition is just a corollary of \cite[Theorem 2.14]{cg3}, where it was proved in much greater generality (for fully weak tricategories). However, as that proof involved long coherence calculations we deem it worthwhile to include a direct proof here. First we show that a weak map of $HV$-algebras is a braided monoidal functor with respect to the $V$-structure; afterwards we will show that given a braided monoidal functor, a weak $HV$-map giving rise to it can be reconstructed. 

\begin{proposition}\label{fbraided}
Let $(F, v, h)\: X \tra Y$ be a weak map of \ddbicatscats. Then $F$ is a braided monoidal functor with respect to the vertical tensor product and the standard braiding.  That is, the following diagram commutes:

\[\begin{small}
\psset{unit=0.1cm,labelsep=2pt,nodesep=3pt,arrows=->}
\pspicture(-10,-14)(10,14)

% a1 a2
% a3 a4

%%%%%%%%%% top

\rput(-10,10){\rnode{a1}{$\vtp{Fa}{Fb}$}}  % top left
\rput(10,10){\rnode{a2}{$F\left(\vtp{a}{b}\right)$}}  % top right
\rput(-10,-10){\rnode{a3}{\vtp{Fb}{Fa}}}  % bottom left
\rput(10,-10){\rnode{a4}{$F\left(\vtp{b}{a}\right)$}}  % bottom right

\ncline{->}{a1}{a2} \naput{{\scriptsize $v$}} % top
\ncline{->}{a3}{a4} \nbput{{\scriptsize $v$}} % bottom
\ncline{->}{a1}{a3} \nbput{{\scriptsize $\gamma$}} % left
\ncline{->}{a2}{a4} \naput{{\scriptsize $F\gamma$}} % right

\endpspicture\end{small}
\]

\end{proposition}

% **This is a special case of the fully weak version so in a way we could just cite something or other (Cheng-Gurksi??) but because those proofs are never written out (cos they're too tedious) it seems worth writing this one out.**

\begin{proof}
% ***ok sure but what does this meeeeeeeeeean?**

The diagram can be seen to commute as shown below. Here ``unit'' refers to the unit axiom for a monoidal functor.
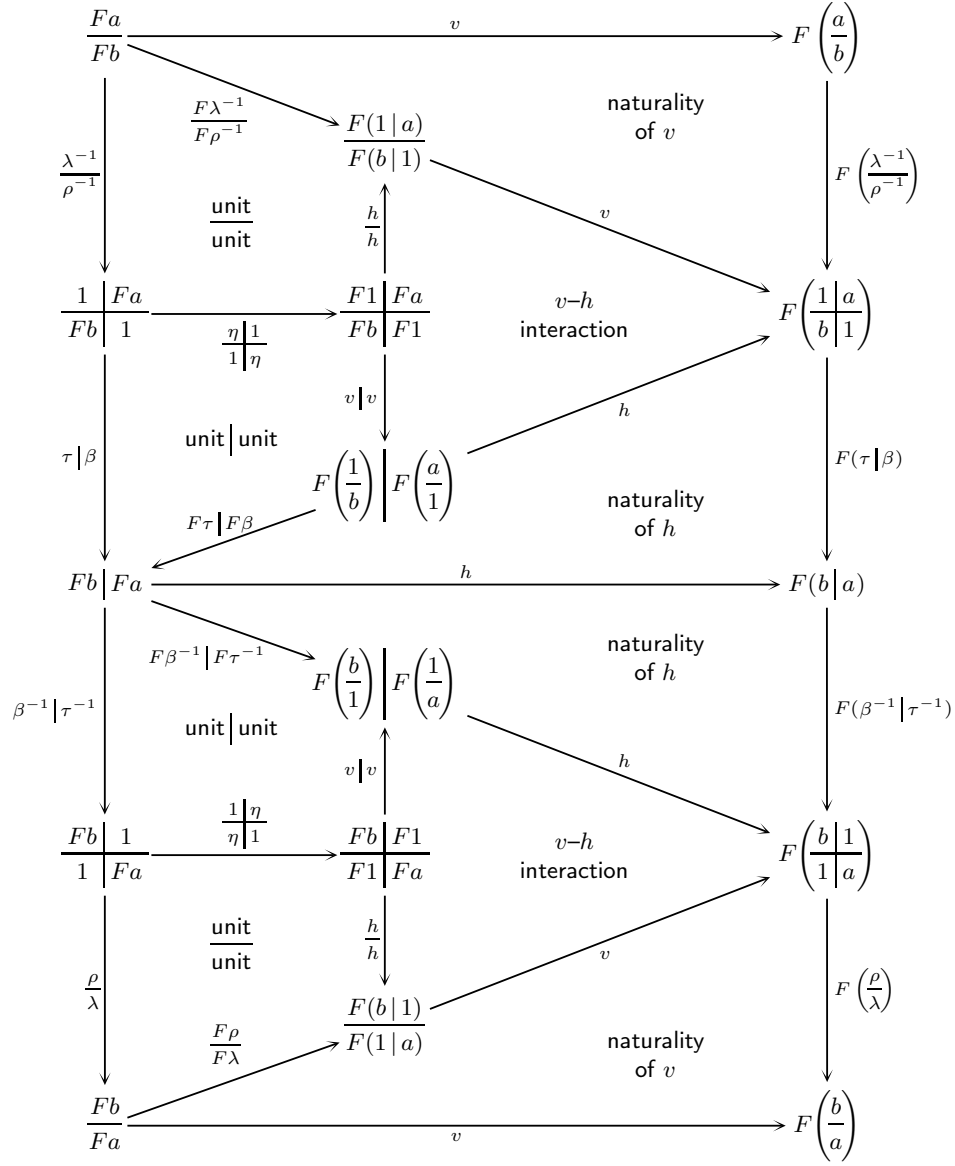
\[\begin{small}
\psset{unit=0.12cm,labelsep=2pt,nodesep=3pt, arrows=->}
\pspicture(-10,-64)(10,64)

% a1       a3
%     a2
% b1   f   b3
%     b2
%e1        e3
%     c2
% c1   g   c3
%     d2
% d1       d3

%%%%%%%%%% top

\rput[B](-40,60){\Rnode{a1}{$\vtp{Fa}{Fb}$}}
\rput(-9,49){\rnode{a2}{$\vtp{F(1 \| a)}{F(b \| 1 )}$}}
\rput[B](40,60){\Rnode{a3}{$F\left(\vtp{a}{b}\right)$}}

\rput(-40,30){\rnode{b1}{$\fourtp{1}{Fa}{Fb}{1}$}}
\rput(-9,11){\rnode{b2}{$\htp{F\hs{-0.2}\left(\vtp{1}{b}\right)}{F\hs{-0.2}\left(\vtp{a}{1}\right)}$}}
\rput(40,30){\rnode{b3}{$F\hs{-0.2}\left(\fourtp{1}{a}{b}{1}\right)$}}

\rput(-40,0){\rnode{e1}{$\htp{Fb}{Fa}$}}
\rput(40,0){\rnode{e3}{$F(\htp{b}{a})$}}

\rput(-9,30){\rnode{f}{$\fourtp{F1}{Fa}{Fb}{F1}$}}
\rput(-9,-30){\rnode{g}{$\fourtp{Fb}{F1}{F1}{Fa}$}}

\rput(-40,-30){\rnode{c1}{$\fourtp{Fb}{1}{1}{Fa}$}}
\rput(-9,-11){\rnode{c2}{$\htp{F\hs{-0.2}\left(\vtp{b}{1}\right)}{F\hs{-0.2}\left(\vtp{1}{a}\right)}$}}
\rput(40,-30){\rnode{c3}{$F\hs{-0.2}\left(\fourtp{b}{1}{1}{a}\right)$}}

\rput(-40,-60){\rnode{d1}{$\vtp{Fb}{Fa}$}}
\rput(-9,-49){\rnode{d2}{$\vtp{F(b \| 1 )}{F(1 \| a)}$}}
\rput(40,-60){\rnode{d3}{$F\hs{-0.2}\left(\vtp{b}{a}\right)$}}

\ncline{a1}{a3} \naput{\scr $v$}
\ncline{a1}{a2} \nbput{\scr $\vtp{F\lambda^{-1}}{F\rho^{-1}}$}
\ncline{a1}{b1} \nbput{\scr $\vtp{\lambda^{-1}}{\rho^{-1}}$}

\ncline{a3}{b3} \naput{\scr $F\left(\vtp{\lambda^{-1}}{\rho^{-1}}\right)$}
\ncline{a2}{b3} \naput{\scr $v$}
\ncline{f}{a2} \naput{\scr $\vtp{h}{h}$}
\ncline{b1}{f} \nbput{\scr $\fourtp{\eta}{1}{1}{\eta}$}
\ncline{f}{b2} \nbput{\scr $\htp{v}{v}$}
\ncline{b2}{b3} \nbput{\scr $h$}
\ncline{b1}{e1} \nbput{\scr $\htp{\tau}{\beta}$}
\ncline[nodesepA=-2pt]{b2}{e1} \nbput[labelsep=-2pt]{\scr $\htp{F\tau}{F\beta}$}
\ncline{e1}{c1} \nbput{\scr $\htp{\beta^{-1}}{\tau^{-1}}$}
\ncline[nodesepB=-2pt]{e1}{c2} \nbput[labelsep=1pt]{\scr $\htp{F\beta^{-1}}{F\tau^{-1}}$}
\ncline{c2}{c3} \naput{\scr $h$}
\ncline{g}{c2} \naput{\scr $\htp{v}{v}$} 
\ncline{c1}{g} \naput{\scr $\fourtp{1}{\eta}{\eta}{1}$}
\ncline{g}{d2} \nbput{\scr $\vtp{h}{h}$}
\ncline{c1}{d1} \nbput{\scr $\vtp{\rho}{\lambda}$}
\ncline{d1}{d2} \naput{\scr $\vtp{F\rho}{F\lambda}$}
\ncline{d1}{d3} \nbput{\scr $v$}
\ncline{d2}{c3}\nbput{\scr $v$}

\ncline{e1}{e3} \naput{\scr $h$}
\ncline{b3}{e3} \naput{\scr $F(\htp{\tau}{\beta})$}
\ncline{e3}{c3} \naput{\scr $F(\htp{\beta^{-1}}{\tau^{-1}})$}
\ncline{c3}{d3} \naput{\scr $F\left(\vtp{\rho}{\lambda}\right)$}

\rput(21, 52){\parbox{10em}{\small\sf\bc naturality\\of $v$\ec}}
\rput(21, -52){\parbox{10em}{\small\sf\bc naturality\\of $v$\ec}}

\rput(21, 8){\parbox{10em}{\small\sf\bc naturality\\of $h$\ec}}
\rput(21, -8){\parbox{10em}{\small\sf\bc naturality\\of $h$\ec}}

\rput(12, 30){\parbox{10em}{\small\sf\bc $v$--$h$\\ interaction\ec}}
\rput(12, -30){\parbox{10em}{\small\sf\bc $v$--$h$\\ interaction\ec}}

\rput(-26, 40){\vtp{\mbox{\sf unit}}{\mbox{\sf unit}}}
\rput(-26, -40){\vtp{\mbox{\sf unit}}{\mbox{\sf unit}}}

\rput(-26, 16){\htp{\mbox{\sf unit}}{\mbox{\sf unit}}}
\rput(-26, -16){\htp{\mbox{\sf unit}}{\mbox{\sf unit}}}

% \ncline{->}{a1}{a2} \naput{{\scriptsize $v$}} % top
% \ncline{->}{a3}{a4} \nbput{{\scriptsize $v$}} % bottom
% \ncline{->}{a1}{a3} \nbput{{\scriptsize $\gamma$}} % left
% \ncline{->}{a2}{a4} \naput{{\scriptsize $F\gamma$}} % right

\endpspicture\end{small}
\]

\end{proof}

Note that the diagram in the above proof splits into two halves (top and bottom) and the halves are themselves key for what follows, showing that, given a weak $HV$-map, the $h$ constraint can always be derived from $v$. 

% Note that if we only had $v$ we might not be able to construct $h$; the last part of the puzzle will be to show that if we  $v$ and a braiding for it then we can in fact always construct $h$.

\begin{proposition}\label{hfromv}
Let $(F, v, h)$ be a weak map of \ddbicatscats\ $X \tra Y$. It follows from the above proof that the following diagram commutes, showing that $h$ can be derived from $v$ (as $\alpha$ and thus $F\alpha$ are invertible):

\[\begin{small}
\psset{unit=0.1cm,labelsep=2pt,nodesep=3pt,arrows=->}
\pspicture(-10,-14)(10,14)

% a1 a2
% a3 a4

%%%%%%%%%% top

\rput(-10,10){\rnode{a1}{$\htp{Fa}{Fb}$}}  % top left
\rput(10,10){\rnode{a2}{$F\left(\htp{a}{b}\right)$}}  % top right
\rput(-10,-10){\rnode{a3}{\vtp{Fa}{Fb}}}  % bottom left
\rput(10,-10){\rnode{a4}{$F\left(\vtp{a}{b}\right)$}}  % bottom right

\ncline{->}{a1}{a2} \naput{{\scriptsize $h$}} % top
\ncline{->}{a3}{a4} \nbput{{\scriptsize $v$}} % bottom
\ncline{->}{a1}{a3} \nbput{{\scriptsize $\alpha$}} % left
\ncline{->}{a2}{a4} \naput{{\scriptsize $F\alpha$}} % right

\endpspicture\end{small}
\]

\end{proposition}

Note that this means that if we know we have a weak $HV$-map, then $h$ can be reconstructed from $v$, but if we started with only $v$ we might not be able to make it into a weak $HV$-map.  It remains to show that any $v$ satisfying the braid axiom will yield a weak map by reconstructing $h$ according to this diagram.  

The idea is that if we start with a constraint $v$ there are two options for extending the weak map structure, depending on how we're thinking about the overall structure. Either we're thinking about braided monoidal categories in which case we want $v$ to satisfy the braid axiom. Or we're thinking about $HV$-algebras in which case we want to reconstruct an $h$ and check the interaction axiom. The point is that these turn out to be equivalent.

\begin{proposition}\label{fweakmap}
Let $X$ and $Y$ be \ddbicatscats\ and $(F,v)$ a monoidal functor $X \tra Y$ with respect to the vertical tensor products, satisfying the braid axiom. Then defining $h$ according to Proposition~\ref{hfromv} makes $F$ into a weak map of \ddbicatscats. 
\end{proposition}

\begin{proof}
We need to show that with $h$ defined in this way the interaction axiom holds.  This is seen by the following diagram; the outside  is the interaction axiom, and we see that it follows from the braid axiom.

\[\begin{small}
\psset{unit=0.13cm,labelsep=2pt,nodesep=5pt, arrows=->}
\pspicture(-10,-68)(10,68)

% a1 a2  a4 a5
%      b3
%    c2  c4
%      d3
% e1 e2  e4 e5

\rput(-45,44){\rnode{a1}{$\fourtp{Fa}{Fb}{Fc}{Fd}$}}
\rput(-15,44){\rnode{a2}{$\fourvtp{Fa}{Fb}{Fc}{Fd}$}}
\rput(15,44){\rnode{a4}{$\vtp{\raisebox{1em}{$F\left(\vtp{a}{b}\right)$}}{\raisebox{-0.8em}{$F\left(\vtp{c}{d}\right)$}}$}}
\rput(45,44){\rnode{a5}{$\vtp{F(a\|b)}{F(c\|d)}$}}

\rput(0,22){\rnode{b3}{$\threevtpp{Fa}{F\left(\vtp{b}{c}\right)}{Fd}$}}

\rput(-15,0){\rnode{c2}{$\fourvtp{Fa}{Fc}{Fb}{Fd}$}}
\rput(15,0){\rnode{c4}{$F\left(\hs{0.1}\fourvtp{a}{b}{c}{d}\hs{0.1}\right)$}}

\rput(0,-22){\rnode{d3}{$\threevtpp{Fa}{F\left(\vtp{c}{b}\right)}{Fd}$}}

\rput(-45,-44){\rnode{e1}{$\htp{F\left(\vtp{a}{c}\right)}{F\left(\vtp{b}{d}\right)}$}}
\rput(-15,-44){\rnode{e2}{$\vtp{\raisebox{1em}{$F\left(\vtp{a}{c}\right)$}}{\raisebox{-0.8em}{$F\left(\vtp{b}{d}\right)$}}$}}
\rput(15,-44){\rnode{e4}{$F\left(\hs{0.1}\fourvtp{a}{c}{b}{d}\hs{0.1}\right)$}}
\rput(45,-44){\rnode{e5}{$F\left(\fourtp{a}{b}{c}{d}\right)$}}

\ncline{a1}{a2} \naput{\scr $\vtp{\alpha}{\alpha}$}
\ncline{a2}{a4} \naput{\scr $\vtp{v}{v}$}
\ncline{a4}{a5} \naput{\scr $\vtp{F\alpha\inv}{F\alpha\inv}$}

\ncline[nodesepB=-1pt]{a2}{b3} \naput[labelsep=0pt,npos=0.45]{\scr $\threevtp{1}{v}{1}$}
\ncline{a4}{c4} \naput{\scr $v$}

\ncline{a1}{e1} \nbput{\scr $\htp{v}{v}$}
\ncline[nodesepB=7pt]{a1}{c2} \naput{\scr $\alpha$}
\ncline{a2}{c2} \nbput{\scr $\threevtp{1}{\gamma}{1}$}
\ncline[nodesepA=7pt,nodesepB=-1pt]{c2}{d3} \naput[labelsep=0pt,npos=0.45]{\scr $\threevtp{1}{v}{1}$}
\ncline[nodesepA=0pt,nodesepB=-4pt]{b3}{c4} \naput{\scr $v$}
\ncline{c4}{e5} \naput[npos=0.45,labelsep=-1pt]{\scr $F\left(\vtp{\alpha\inv}{\alpha\inv}\right)$}
\ncline{a5}{e5} \naput{\scr $v$}
\ncline{b3}{d3} \naput[npos=0.55]{\scr $\threevtp{1}{F\gamma}{1}$}
\ncline[nodesepA=0pt,nodesepB=-4pt]{d3}{e4} \naput{\scr $v$}
\ncline{e1}{e2} \nbput{\scr $\alpha$}
\ncline{e2}{e4} \nbput{\scr $v$}
\ncline{e4}{e5} \nbput{\scr $F\alpha\inv$}

\ncline{e2}{c2} \nbput{\scr $\vtp{v}{v}$}
\ncline{c4}{e4} \naput[npos=0.6]{\scr $F\left(\threevtp{1}{\gamma}{1}\right)$}

\rput(-28, 35){\parbox{10em}{\small\sf\bc Prop \ref{alphakey} \ec}}
\rput(28, -35){\parbox{10em}{\small\sf\bc Prop \ref{alphakey} \ec}}

\rput(-30, -15){\parbox{10em}{\small\sf\bc naturality\\ of $\alpha$ \ec}}
\rput(30, 15){\parbox{10em}{\small\sf\bc naturality\\ of $v$ \ec}}

\rput(-3, -37){\parbox{10em}{\small\sf\bc coherence\\ of $v$ \ec}}

\rput(3, 37){\parbox{10em}{\small\sf\bc coherence\\ of $v$ \ec}}

\rput(7, -11){\parbox{10em}{\small\sf\bc naturality\\ of $v$ \ec}}

\rput(-7, 11){\parbox{10em}{\small\sf\bc braid\\ axiom\\ for $v$ \ec}}

\eps\end{small}\]

\end{proof}

\begin{remark}
Note that we could equally define $h$ using $\ol\alpha$, but in the presence of the braid axiom for $v$, this produces the same constraint, as seen from the diagram below, where the top and bottom edges are the two different ways of producing an $h$, the square is the braid axiom, and the triangles come from the definition of $\gamma$. 

\[\begin{small}
\psset{unit=0.1cm,labelsep=2pt,nodesep=3pt,arrows=->}
\pspicture(-10,-14)(10,14)

%   a1 a2
% a0      a5
%   a3 a4

%%%%%%%%%% top

\rput(-30,0){\rnode{a0}{$\htp{Fa}{Fb}$}}  % top left
\rput(30,0){\rnode{a5}{$F\big(\htp{a}{b}\big)$}}  % top left

\rput(-10,-10){\rnode{a1}{$\vtp{Fa}{Fb}$}}  % top left
\rput(10,-10){\rnode{a2}{$F\left(\vtp{a}{b}\right)$}}  % top right
\rput(-10,10){\rnode{a3}{\vtp{Fb}{Fa}}}  % bottom left
\rput(10,10){\rnode{a4}{$F\left(\vtp{b}{a}\right)$}}  % bottom right

\ncline[nodesepA=4pt]{->}{a1}{a2} \nbput{{\scriptsize $v$}} % top
\ncline[nodesepA=4pt]{->}{a3}{a4} \naput{{\scriptsize $v$}} % bottom
\ncline{->}{a3}{a1} \naput{{\scriptsize $\gamma$}} % left
\ncline{->}{a4}{a2} \nbput{{\scriptsize $F\gamma$}} % right

\ncline[nodesepA=2pt]{->}{a0}{a1} \nbput{{\scriptsize ${\alpha}$}} % top
\ncline[nodesepA=2pt]{->}{a0}{a3} \naput{{\scriptsize $\ol{\alpha}$}} % bottom
\ncline[nodesepB=2pt]{->}{a2}{a5} \nbput{{\scriptsize $F{\alpha}^{-1}$}} % top
\ncline[nodesepB=2pt]{->}{a4}{a5} \naput{{\scriptsize $F\ol{\alpha}^{-1}$}} % top

\endpspicture\end{small}
\]

\noi In fact, as the triangles in this diagram are simply the definition of $\gamma$, we see that the braid axiom is the assertion that these two ways of producing a horizontal constraint are the same.

\end{remark}

% ***we need a note about where the choice came in, or if it didn't**

We have now shown that a weak map of $HV$-algebras is equivalently a weak map of $V$-algebras satisfying the braid axiom.  We now move on to transformations.

\subsection{Transformations in terms of vertical structure}
\label{four}

We will now characterise transformations of \ddbicatscats.  A priori we know that these have a vertical and a horizontal monoidal structure (with an interaction axiom) and we know that weak maps between such are monoidal with respect to each of those structures.  We will refer to transformations being ``horizontally monoidal'' and ``vertically monoidal'' if they are monoidal with respect to the horizontal or vertical monoidal structures respectively.

By the results of Section~\ref{transdl} we know that a transformation of weak maps of $HV$-algebras is a 2-cell that is both an $H$-transformation and a $V$-transformation.  We will now show that in the case of \ddbicatscats, being a $V$-transformation suffices as we can use the weak Eckmann--Hilton argument to derive the $H$-structure (horizontal monoidal) axiom.  This tells us that a transformation of \ddbicatscats\ is precisely a monoidal transformation between the associated braided monoidal categories.

\begin{proposition}\label{ddtrans}
Consider
$\psset{unit=0.08cm,labelsep=2pt,nodesep=2pt}
\pspicture(-13,-1)(13,7)

\rput(-10,0){\rnode{a1}{$X$}}  % named node, with something placed there
\rput(10,0){\rnode{a2}{$Y$}}  % named node, with something placed there

\ncarc[arcangle=45]{->}{a1}{a2}\naput{{\scriptsize $F$}}
\ncarc[arcangle=-45]{->}{a1}{a2}\nbput{{\scriptsize $G$}}

{
\rput[c](0,0){\psset{unit=1mm,doubleline=true,arrowinset=0.6,arrowlength=0.5,arrowsize=0.5pt 2.1,nodesep=0pt,labelsep=2pt}
\pcline{->}(0,1.8)(0,-1.8) \naput{{\scriptsize $\theta$}}}}

\endpspicture$
where

\vs{0.9}

\begin{itemize}
\item $X$ and $Y$ are \ddbicatscats, 
\item $F$ and $G$ are weak maps, and 
\item $\theta$ is a $V$-transformation, that is, vertically monoidal. 
\end{itemize}
Then $\theta$ is an $H$-transformation, that is, horizontally monoidal.
\end{proposition}

% \begin{proposition}\label{ddtrans}
% Let $F, G\: X \tra Y$ be weak maps of \ddbicatscats.  Let $\theta\: F \Tra G$ be a $T$-transformation, that is, vertically monoidal. Then it is an $S$-transformation, that is, horizontally monoidal.
% \end{proposition}

\begin{proof}
As usual we write $h$ and $v$ for the horizontal and vertical monoidal functor constraints respectively (for both $F$ and $G$).  We know that the diagram for a monoidal transformation with respect to $v$ commutes; we need to check that the diagram for a monoidal transformation with respect to $h$ follows. By Proposition~\ref{hfromv} we know that $h$ can be expressed in terms of $v$ (for $F$ and $G$ respectively) giving the top and bottom edges of the diagram below. Thus the monoidal transformation diagram we need to check becomes the outside of the diagram below, which is seen to commute as shown:

\[
\psset{unit=0.12cm,labelsep=4pt,nodesep=3pt}
\pspicture(-20,-9)(20,14)

% a1 a2 a3 a4
% b1 b2 b3 b4

\rput[B](-36,10){\Rnode{a1}{$\htp{Fa}{Fb}$}}  % top left
\rput[B](-12,10){\Rnode{a2}{$\vtp{Fa}{Fb}$}}  % top right
\rput[B](12,10){\Rnode{a3}{$F\left(\vtp{a}{b}\right)$}}  % top right
\rput[B](36,10){\Rnode{a4}{$F(\htp{a}{b})$}}  % top right

\rput[B](-36,-10){\Rnode{b1}{$\htp{Ga}{Gb}$}}  % bottom left
\rput[B](-12,-10){\Rnode{b2}{$\vtp{Ga}{Gb}$}}  % bottom right
\rput[B](12,-10){\Rnode{b3}{$G\left(\vtp{a}{b}\right)$}}  % bottom right
\rput[B](36,-10){\Rnode{b4}{$G(\htp{a}{b})$}}  % bottom right

\ncline{->}{a1}{a2} \naput{{\scriptsize $\alpha$}} % top
\ncline{->}{b1}{b2} \nbput{{\scriptsize $\alpha$}} % bottom

\ncline{->}{a2}{a3} \naput{{\scriptsize $v$}} % top
\ncline{->}{b2}{b3} \nbput{{\scriptsize $v$}} % bottom

\ncline{->}{a3}{a4} \naput{{\scriptsize $F\alpha\inv$}} % top
\ncline{->}{b3}{b4} \nbput[labelsep=2pt]{{\scriptsize $G\alpha\inv$}} % bottom

\ncline{->}{a1}{b1} \nbput{{\scriptsize $\htp{\theta_a}{\theta_b}$}} % left
\ncline{->}{a2}{b2} \naput{{\scriptsize $\vtp{\theta_a}{\theta_b}$}} % right
\ncline{->}{a3}{b3} \naput{{\scriptsize $\theta_{\frac{a}{b}}$}} % right
\ncline{->}{a4}{b4} \naput{{\scriptsize $\theta_{\htp{a}{b}}$}} % right

\rput(-24, 1){\parbox{10em}{\small\sf\bc naturality\\ of $\alpha$ \ec}}
\rput(1, 1){\parbox{10em}{\small\sf\bc vertical\\monoidal\\axiom \ec}}
\rput(25, 1){\parbox{10em}{\small\sf\bc naturality\\ of $\theta$ \ec}}

\endpspicture
\]
\end{proof}

Note that \ddbicatscats, weak maps, and transformations between them form a 2-category which we will write as \twocatt{dd\bicats-Cat}; it is a full sub-2-category of \twocatt{$HV\cat{-Alg}_w$}. 

%%%%%%%%%%%%%%%%%%%%%%%%%%%%%%%%%%%%%%%%%%%%%%%%% 
% Section 5
%%%%%%%%%%%%%%%%%%%%%%%%%%%%%%%%%%%%%%%%%%%%%%%%%

\section{Biadjoint biequivalence}
\label{five}

We are now ready to state and prove our main theorem. We exhibit a comparison 2-functor between \twocatt{dd\bicats-Cat} and \twocatt{BrMonCat} and prove that it is part of a biadjoint biequivalence. Here we write \twocatt{BrMonCat} for the 2-category of braided (weakly) monoidal categories, weak monoidal functors between them, and monoidal transformations. 

In fact, all the technical components of the equivalence have been proved in \wvc\ and the previous section, so this is just a case of bringing all those results together.

\begin{theorem}[Main Theorem]\label{maintheorem}
There is a 2-functor
\[U \: \twocatt{dd\bicats-Cat} \mtra \twocatt{BrMonCat}\]
extending the construction on 0-cells given in \wvc, and it is part of a biadjoint biequivalence of 2-categories.
\end{theorem}

\begin{proof}
First we construct the 2-functor $U$.
\begin{itemize}

\item On 0-cells: given a \ddbicatscat\ $X$, $\UX$ is the braided monoidal category given by the vertical tensor product and the standard braiding $\gamma$.

\item On 1-cells: given a weak map of \ddbicatscats\ 
\[(F,v,h)\: X \tra Y\]
$UF$ is the associated braided monoidal functor $(F,v)$, which we know is braided by Proposition~\ref{fbraided}.

\item On 2-cells: given a transformation between weak maps of \ddbicatscats\
\[
\psset{unit=0.1cm,labelsep=2pt,nodesep=3pt}
\pspicture(-10,-5)(10,5)

\rput(-10,0){\rnode{a1}{$X$}}  % named node, with something placed there
\rput(10,0){\rnode{a2}{$Y$}}  % named node, with something placed there

\ncarc[arcangle=45]{->}{a1}{a2}\naput{{\scriptsize $F$}}
\ncarc[arcangle=-45]{->}{a1}{a2}\nbput{{\scriptsize $G$}}

{
\rput[c](0,0){\psset{unit=1mm,doubleline=true,arrowinset=0.6,arrowlength=0.5,arrowsize=0.5pt 2.1,nodesep=0pt,labelsep=2pt}
\pcline{->}(0,1.5)(0,-1.5) \naput{{\scriptsize $\theta$}}}}

\endpspicture
\]
$U\theta$ is the underlying transformation with respect to just the vertical monoidal structure.

\end{itemize}

\noi This is a strict 2-functor.  The main theorem of \wvc\ proved that $U$ is biessentially surjective on 0-cells. By Proposition~\ref{fweakmap} we know $U$ is locally essentially surjective on 1-cells (in fact locally surjective). By Proposition~\ref{ddtrans} we know $U$ is locally full and faithful on 2-cells. Then by \cite[Lemma~3.1]{gur3} it follows that $U$ is part of a biadjoint biequivalence of 2-categories.
\end{proof}

Note that constructing a pseudo-inverse for $U$ is non-trivial.  A candidate construction on 0-cells was made in \wvc\ to prove the biessential surjectivity, but extending it to a 2-functor requires more work and we defer it to a sequel.

In future work we will perform an analogous analysis for doubly-degenerate tricategories according to the theory of Trimble \cite{tri1}. This theory uses iterated enrichment with more general operad actions, with the result that although the ideas are analogous the technicalities are a little more intricate.

%%%%%%%%%%%%%%%%%%%%%%%%%%%%%%%%%%%%%%%%%%%%%%%%%%%%%%%%%%%%%%%%%%%%%%%
%%%%%%%%%%%%%%%%%%%%%%%%%%%%%%%%%%%%%%%%%%%%%%%%%%%%%%%%%%%%%%%%%%%%%%%
%%
%% Appendix
%%
%%%%%%%%%%%%%%%%%%%%%%%%%%%%%%%%%%%%%%%%%%%%%%%%%%%%%%%%%%%%%%%%%%%%%%%
%%%%%%%%%%%%%%%%%%%%%%%%%%%%%%%%%%%%%%%%%%%%%%%%%%%%%%%%%%%%%%%%%%%%%%%

\appendix

\section{String diagram calculations}

In this section we will give the deferred 2-categorical proofs that are more efficaciously performed using string diagrams.  The advantage of string diagrams in this case is that almost all our 2-categorical concepts are strict, and so we can ``ignore'' naturality squares.  The conventions we use are as follows.  We read the diagrams from top to bottom. We are working with 2-categories, 2-functors, and transformations (all strict), so in particular the relative heights do not matter (because of naturality).  To simplify the diagrams we will omit labels wherever there is no ambiguity.

% ***justify the fact that we represent objects by strings as well**

% 
% monads
% algebras
% 
% monad TS
% 
% distributive pairs
% 
% weak maps
% 
% Theorem~\ref{eh-map}
% 
% abstract braiding
% 
% Theorem~\ref{eh-trans}
% 
% transformations of weak maps
% 
% ***

\subsection{Monads and distributive laws}

We begin by laying out our basic notation for the classical results of 2-monads and distributive laws. For a 2-monad $T$ on a 2-category $\cC$, we write its multiplication, unit, and axioms as follows.

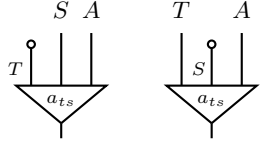
\[\psset{unit=0.1cm,labelsep=2pt,nodesep=3pt,linewidth=0.8pt}
\pspicture(0,15)(90,25)
\begin{small}

%%%%%%%%%%%%%%%%%%%%%%%%%%%%%%%%%%%%%%%% monad

%%%%%%%%%%% mu

\rput(0,20){

\rput(-2,5){$T$}
\rput(2,5){$T$}

\rput(0,-4){$T$}

\psline(-2,2)(0,0)(2,2)
\psline(0,0)(0,-2)

}

%%%%%%%%%%%%%%%%%%%%%%%%%%%%%%%%%%%% eta

\rput(12,20){

\rput(0,5){$1$}

\rput(0,-4){$T$}

\psline(0,0)(0,-2)
\pscircle(0,0.5){0.5}

}

%%%%%%%%%%%%%%%%%%%%%%%%%%%%%%%%%%% multiplication

\rput(35,20){

\rput(0,0){
\psline(-2,2)(0,0)(2,2)
\psline(0,0)(0,-2)
}

\rput(2,-4){
\psline(-2,2)(0,0)(2,2)
\psline(0,0)(0,-2)

\psline(2,2)(2,6)

}

\rput(10,-2){$=$}

\rput(18,0){

\rput(2,0){
\psline(-2,2)(0,0)(2,2)
\psline(0,0)(0,-2)
}

\rput(0,-4){
\psline(-2,2)(0,0)(2,2)
\psline(0,0)(0,-2)

\psline(-2,2)(-2,6)

}
}

}

%%%%%%%%%%%%%%%%%%%%%%%%%%%%%%%%%%% unit triangles

\rput(70,20){

\psline(-2,0)(-2,-2)
\pscircle(-2,0.5){0.5}
\psline(2,2)(2,-2)
\rput(0,-4){
\psline(-2,2)(0,0)(2,2)
\psline(0,0)(0,-2)
}

\rput(7,-2){$=$}

\rput(10,0){
\psline(2,2)(2,-6)
}

\rput(17,-2){$=$}

\rput(24,0){
\psline(2,0)(2,-2)
\pscircle(2,0.5){0.5}
\psline(-2,2)(-2,-2)

\rput(0,-4){
\psline(-2,2)(0,0)(2,2)
\psline(0,0)(0,-2)
}
}

}

\end{small}
\endpspicture\]

%%%%%%%%%%%%%%%%%%%%%%%%%%%%%% DL

\noi We write a distributive law $ST \Tra TS$ as
\[\psset{unit=0.1cm,labelsep=2pt,nodesep=3pt,linewidth=0.8pt}
\pspicture(0,-3)(10,7)
\begin{small}

\rput(0,0){

\rput(-2,5){$S$}
\rput(2,5){$T$}

\rput(-2,-4){$T$}
\rput(2,-4){$S$}

\rput(0,0.5){
\psline(-2,2)(2,-2)
\psline[border=2pt](2,2)(-2,-2)
}
}

\end{small}
\endpspicture\]

\noi and the axioms for a distributive law as follows:

\[\psset{unit=0.1cm,labelsep=2pt,nodesep=3pt,linewidth=0.8pt}
\pspicture(20,-10)(120,4)
\begin{small}

% \rput(0,0){
% 
% \rput(-2,5){$S$}
% \rput(2,5){$T$}
% 
% \rput(-2,-4){$T$}
% \rput(2,-4){$S$}
% 
% 
% \psline(-2,2)(2,-2)
% \psline[border=2pt](2,2)(-2,-2)
% 
% }
% 

%%%%%%%%%%%%%%%%%%%%%%%%%%%%%%%%%%%%%%%%%%%%%%%%%%%% first DL axiom

\rput(20,0){

 \rput(-6,5){$S$}
\rput(-2,5){$S$}
\rput(2,5){$T$}
\psline(-2,2)(2,-2)
\psline[border=2pt](2,2)(-2,-2)
\psline(-6,2)(-6,-2)

\rput(-4,-4){
\psline(-2,2)(2,-2)
\psline[border=2pt](2,2)(-2,-2)
\psline(6,2)(6,-2)
}

\rput(0,-8){
\psline(-2,2)(0,0)(2,2)
\psline(0,0)(0,-2)
\psline(-6,2)(-6,-2)
}

\rput(7,-4){$=$}

\rput(16,0){

% \rput(-2,5){$S$}
% \rput(2,5){$T$}
%  \rput(-6,5){$S$}

\rput(-4,0){
\psline(-2,2)(0,0)(2,2)
\psline(0,0)(0,-2)
\psline(6,2)(6,-2)

\psline(0,-2)(6,-10)
\psline[border=2pt](6,-2)(0,-10)

}

}

}

%%%%%%%%%%%%%%%%%%%%%%%%%%%%%%%%%%%%%%%%%%%%%%%%%%%% 2nd DL axiom

\rput(54,0){

 \rput(-6,5){$S$}
\rput(-2,5){$T$}
\rput(2,5){$T$}
\rput(-4,0){
\psline(-2,2)(2,-2)
\psline[border=2pt](2,2)(-2,-2)
}
\psline(2,2)(2,-2)

\rput(0,-4){
\psline(-2,2)(2,-2)
\psline[border=2pt](2,2)(-2,-2)
\psline(-6,2)(-6,-2)
}

\rput(-4,-8){
\psline(-2,2)(0,0)(2,2)
\psline(0,0)(0,-2)
\psline(6,2)(6,-2)
}

\rput(7,-4){$=$}

\rput(17,0){

%  \rput(-6,5){$S$}
% \rput(-2,5){$T$}
% \rput(2,5){$T$}

\rput(0,0){
\psline(-2,2)(0,0)(2,2)
\psline(0,0)(0,-2)
\psline(-6,2)(-6,-2)

\rput(-6,0){
\psline(0,-2)(6,-10)
\psline[border=2pt](6,-2)(0,-10)
}

}

}

}

%%%%%%%%%%%%%%%%%%%%%%%%%%%%%%%%%%%%%%%% DL unit axioms

%%%%%%%%%%%%%%%%%%%%%%%%%%%%%%%%%%%%%%%% unit for S

\rput(85,-2){

\rput(2,5){$T$}

\rput(-2,0){
\psline(0,0)(0,-2)
\pscircle(0,0.5){0.5}
}
\psline(2,2)(2,-2)

\rput(0,-4){
\psline(-2,2)(2,-2)
\psline[border=2pt](2,2)(-2,-2)
}

\rput(-2,-8){$T$}
\rput(2,-8){$S$}

\rput(6.5,-4){$=$}

%%%%%%%%%%%%%%%%%%%%%%%%%%%%%%%%%%%%%%%%%%%%%%%%%%%%%%%%%%%%%%% RHS

\rput(13,0){

\rput(-2,5){$T$}

\rput(2,-2){
\psline(0,0)(0,-4)
\pscircle(0,0.5){0.5}
}

\psline(-2,2)(-2,-6)

\rput(-2,-8){$T$}
\rput(2,-8){$S$}

}

%%%%%%%%%%%%%%%%%%%%%%%%%%%%%%%%%%%%%%%%%%%%%%%% end RHS

}

%%%%%%%%%%%%%%%%%%%%%%%%%%%%%%%%%%%%%%%% unit for T

\rput(110,-2){

\rput(-2,5){$S$}

\rput(2,0){
\psline(0,0)(0,-2)
\pscircle(0,0.5){0.5}
}
\psline(-2,2)(-2,-2)

\rput(0,-4){
\psline(-2,2)(2,-2)
\psline[border=2pt](2,2)(-2,-2)
}

\rput(-2,-8){$T$}
\rput(2,-8){$S$}

\rput(6.5,-4){$=$}

%%%%%%%%%%%%%%%%%%%%%%%%%%%%%%%%%%%%%%%%%%%%%%%%%%%%%%%%%%%%%%% RHS

\rput(13,0){

\rput(2,5){$S$}

\rput(-2,-2){
\psline(0,0)(0,-4)
\pscircle(0,0.5){0.5}
}

\psline(2,2)(2,-6)

\rput(-2,-8){$T$}
\rput(2,-8){$S$}

}

%%%%%%%%%%%%%%%%%%%%%%%%%%%%%%%%%%%%%%%%%%%%%%%% end RHS

}

\end{small}
\endpspicture\]

\noi Given such a distributive law, $TS$ becomes a monad with the following multiplication and unit:
\[\psset{unit=0.1cm,labelsep=2pt,nodesep=3pt,linewidth=0.8pt}
\pspicture(0,-8)(45,8)
\begin{small}

\rput(0,0){
 \rput(-6,5){$T$}
\rput(-2,5){$S$}
\rput(2,5){$T$}
\rput(6,5){$S$}

\psline(-6,2)(-6,-2)
\psline(-2,2)(2,-2)
\psline[border=2pt](2,2)(-2,-2)
\psline(6,2)(6,-2)

\rput(-4,-4){
\rput(0,-4){$T$}
\psline(-2,2)(0,0)(2,2)
\psline(0,0)(0,-2)
}

\rput(4,-4){
\rput(0,-4){$S$}
\psline(-2,2)(0,0)(2,2)
\psline(0,0)(0,-2)
}
}

\rput(20,0){

\rput(-2,0){
\rput(0,-4){$T$}
\psline(0,0)(0,-2)
\pscircle(0,0.5){0.5}
}

\rput(2,0){
\rput(0,-4){$S$}
\psline(0,0)(0,-2)
\pscircle(0,0.5){0.5}
}

}

\end{small}
\endpspicture\]

\noi We write an algebra for a monad $T$ as follows, together with its axioms. (Note that this can be made consistent with the string diagram notation by regarding the object $A$ as a functor $1 \tra \cC$.)
\[\psset{unit=0.1cm,labelsep=2pt,nodesep=3pt,linewidth=0.8pt}
\pspicture(0,22)(100,40)
\begin{small}

%%%%%%%%%%%%%%%%%%%%%%%%%%%%%%%%%%%%%%%%%%%% algebra

\rput(0,30){

\rput(-2,10){$T$}
\rput(2,10){$A$}
\rput(0,-5){$A$}

\psline(0,0)(-4,5)(4,5)(0,0)
\psline(-2,5)(-2,7)
\psline(2,5)(2,7)
\psline(0,0)(0,-2)
\rput(0,3){\scr $a_t$}
}

% \rput(20,60){
% 
% \psline(0,0)(-4,5)(4,5)(0,0)
% \rput(-2,10){$S$}
% \rput(2,10){$A$}
% \rput(0,-5){$A$}
% \psline(-2,5)(-2,7)
% \psline(2,5)(2,7)
% \psline(0,0)(0,-2)
% \rput(0,3){\scr $a_s$}
% }

%%%%%%%%%%%%%%%%%%%%%%%%%%%%%%%%% algebra axioms

\rput(25,30){

\rput(1,0){
\rput(-6,10){$T$}
\rput(-2,10){$T$}
\rput(2,10){$A$}
}

\rput(1,0){

\psline(0,0)(-4,5)(4,5)(0,0)
\psline(-2,5)(-2,7)
\psline(2,5)(2,7)
\psline(0,0)(0,-2)
\rput(0,3){\scr $a_t$}

}

\rput(-2,-8){

\psline(0,0)(-5,5)(5,5)(0,0)
\psline(-3,5)(-3,15)
\psline(3,5)(3,7)
\psline(0,0)(0,-2)
\rput(0,3){\scr $a_t$}

}

\rput(9,0){$=$}

\rput(20,-1){

\rput(0,0){
% \rput(-6,10){$T$}
% \rput(-2,10){$T$}
% \rput(2,10){$A$}

\rput(-4,5){
\psline(-2,2)(0,0)(2,2)
\psline(0,0)(0,-2)
\psline(6,2)(6,-2)
}

\rput(-1,-8){
\psline(0,0)(-5,5)(5,5)(0,0)
\psline(-3,5)(-3,12)
\psline(3,5)(3,12)
\psline(0,0)(0,-2)
\rput(0,3){\scr $a_t$}
}

}

}

}

\rput(68,25){

% \rput(-2,10){$T$}
% \rput(2,10){$A$}
% \rput(0,-5){$A$}

\psline(0,0)(-4,5)(4,5)(0,0)
\psline(-2,5)(-2,7.5)
\pscircle(-2,8){0.5}

\psline(2,5)(2,10)
\psline(0,0)(0,-2)
\rput(0,3){\scr $a_t$}

\rput(10,4){$=$}

\rput(16,0){
\psline(0,-2)(0,10)

}

}

\end{small}
\endpspicture\]

\noi We now turn to algebras for a composite monad $TS$ arising from a distributive law.  Corollory~\ref{distpair} says that a $TS$-algebra is equivalently a $T$ algebra and an $S$-algebra satisfying the following interaction axiom.
\[\psset{unit=0.1cm,labelsep=2pt,nodesep=3pt,linewidth=0.8pt}
\pspicture(0,0)(35,20)
\begin{small}

% \rput(0,30){interaction axiom}

\rput(0,10){

\rput(-6,10){$S$}
\rput(-2,10){$T$}
\rput(2,10){$A$}

\rput(-4,5){
\psline(-2,2)(2,-2)
\psline[border=2pt](2,2)(-2,-2)
\psline(6,2)(6,-2)
}

\rput(0,-4){
\psline(0,0)(-4,5)(4,5)(0,0)
\psline(-2,5)(-2,7)
\psline(2,5)(2,7)
\psline(0,0)(0,-2)
\rput(0,3){\scr $a_s$}
\psline(-6,7)(-6,-2)
}

\rput(-3,-11){
\psline(0,0)(-4.5,5)(4.5,5)(0,0)
\psline(-3,5)(-3,7)
\psline(3,5)(3,7)
\psline(0,0)(0,-2)
\rput(0,3){\scr $a_t$}
}

}

\rput(12,8){$=$}

\rput(25,10){

% \rput(-6,10){$S$}
% \rput(-2,10){$T$}
% \rput(2,10){$A$}

\rput(0,-4){
\psline(0,0)(-4,5)(4,5)(0,0)
\psline(-2,5)(-2,11)
\psline(2,5)(2,11)
\psline(0,0)(0,-2)
\rput(0,3){\scr $a_t$}
\psline(-6,11)(-6,-2)
}

\rput(-3,-11){
\psline(0,0)(-4.5,5)(4.5,5)(0,0)
\psline(-3,5)(-3,7)
\psline(3,5)(3,7)
\psline(0,0)(0,-2)
\rput(0,3){\scr $a_s$}
}

}

\end{small}
\endpspicture\]

\noi Given a $TS$-algebra $a_{ts}$ we produce an $S$-algebra and a $T$-algebra as follows; it is then quite straightforward to check they satisfy the interaction axiom using the diagrams.

\[\psset{unit=0.1cm,labelsep=2pt,nodesep=3pt,linewidth=0.8pt}
\pspicture(0,-5)(35,9)
\begin{small}

% \rput(0,-10){distributive pair from $TS$-algebra}

\rput(0,0){

\rput(-6.1,2.3){\scr $T$}
\rput(0,10){$S$}
\rput(4,10){$A$}

\rput(0,-5){
\psline(0,0)(-6,5)(6,5)(0,0)
\pscircle(-4,10.5){0.5}
\psline(-4,5)(-4,10)
\psline(0,5)(0,12)
\psline(4,5)(4,12)
\psline(0,0)(0,-2)
\rput(0,3){\scr $a_{ts}$}
}

}

\rput(20,0){

\rput(-4,10){$T$}
\rput(-1.7,2.3){\scr $S$}
\rput(4,10){$A$}

\rput(0,-5){
\psline(0,0)(-6,5)(6,5)(0,0)
\pscircle(0,10.5){0.5}
\psline(0,5)(0,10)
\psline(-4,5)(-4,12)
\psline(4,5)(4,12)
\psline(0,0)(0,-2)
\rput(0,3){\scr $a_{ts}$}
}

}

\end{small}
\endpspicture\]

\noi Conversely, given an $S$-algebra and a $T$ algebra we construct a putative $TS$-algebra as follows, and can then use the string diagrams to check that the algebra axioms follow from the interaction axiom. (That is, the multiplication axiom follows from the interaction axiom; the unit axiom follows from the individual unit axioms.)

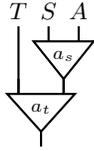
\[\psset{unit=0.1cm,labelsep=2pt,nodesep=3pt,linewidth=0.8pt}
\pspicture(0,-5)(45,10)
\begin{small}

\rput(0,8){

\rput(0,1){
\rput(-6,4){$T$}
\rput(-2,4){$S$}
\rput(2,4){$A$}
}

\rput(0,-4){
\psline(0,0)(-4,5)(4,5)(0,0)
\psline(-2,5)(-2,7)
\psline(2,5)(2,7)
\psline(0,0)(0,-2)
\rput(0,3){\scr $a_s$}
\psline(-6,7)(-6,-2)
}

\rput(-3,-11){
\psline(0,0)(-4.5,5)(4.5,5)(0,0)
\psline(-3,5)(-3,7)
\psline(3,5)(3,7)
\psline(0,0)(0,-2)
\rput(0,3){\scr $a_t$}
}

}

\end{small}
\endpspicture\]

\subsection{Weak maps of algebras}

We now address weak maps of algebras.  Note that in all that follows, we will label 2-cells between string diagrams just with the name of the non-trivial part of the 2-cell.

A weak map of algebras 
\[\psset{nodesep=2pt}
\left(\pspicture(-3,8)(4,17)
\rput(0,0){
\rput(0,14){\rnode{a1}{$TA$}}  %  top
\rput(0,3){\rnode{a2}{$A$}}  % bottom
\ncline{->}{a1}{a2} \naput{{\scriptsize $a_t$}} % top
}
% \rput(6,8){,}
% \rput(10,0){
% \rput(0,14){\rnode{a1}{$SA$}}  %  top
% \rput(0,3){\rnode{a2}{$A$}}  % bottom
% \ncline{->}{a1}{a2} \naput{{\scriptsize $a_s$}} % top
% }
\endpspicture\right)
\ltra
\psset{nodesep=2pt}
\left(\pspicture(-3,8)(3,17)
\rput(0,0){
\rput(0,14){\rnode{a1}{$TB$}}  %  top
\rput(0,3){\rnode{a2}{$B$}}  % bottom
\ncline{->}{a1}{a2} \naput{{\scriptsize $b_t$}} % top
}
% \rput(6,8){,}
% \rput(10,0){
% \rput(0,14){\rnode{a1}{$SB$}}  %  top
% \rput(0,3){\rnode{a2}{$B$}}  % bottom
% \ncline{->}{a1}{a2} \naput{{\scriptsize $b_s$}} % top
% }
\endpspicture\right)
\]
consists of a 1-cell $A \tmap{f} B$ and a 2-cell isomorphism as shown below:

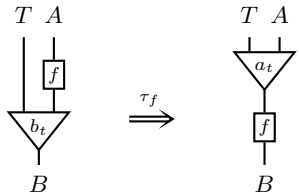
\[\psset{unit=0.1cm,labelsep=2pt,nodesep=3pt,linewidth=0.8pt}
\pspicture(0,53)(50,80)
\begin{small}

%%%%%%%%%%%%%%%%%%%%%%%%%%%%%%%%%%%%%%%%%%% weak map

\rput(0,60){

\rput(0,4){
\rput(-2,14){$T$}
\rput(2,14){$A$}
}

\psline(0,0)(-4,5)(4,5)(0,0)
\psline(-2,5)(-2,15)
\psline(2,5)(2,15)

\rput(2,10){
\psframe*[linecolor=white](-1.5,-2)(1.5,2)
\psframe(-1.5,-2)(1.5,2)
\rput(0,0){\scr $f$}
}

\psline(0,0)(0,-2)
\rput(0,3){\scr $b_t$}
\rput(0,-4.5){$B$}

\rput(15,5){\Tmap{\tau_f}}

\rput(30,0){

\rput(0,8){
\rput(-2,10){$T$}
\rput(2,10){$A$}
% \rput(0,-5){$A$}

\psline(0,0)(-4,5)(4,5)(0,0)
\psline(-2,5)(-2,7)
\psline(2,5)(2,7)
\psline(0,0)(0,-10)
\rput(0,3){\scr $a_t$}
}

% \rput(0,-4.5){$B$}

\rput(0,3){
\psframe*[linecolor=white](-1.5,-2)(1.5,2)
\psframe(-1.5,-2)(1.5,2)
\rput(0,0){\scr $f$}
}

\rput(0,-4.5){$B$}

}

}

%%%%%%%%%%%%%%%%%%%%%%%%%%%%%%%%%%%%%%%%%%% end weak map weak map

%%%%%%%%%%%%%%%%%%%%%%%%%%%%%%%%%%%%%%%%%%%%%%%%%%%%%%%%%%%%%%%

\end{small}
\endpspicture\]

\noi satisfying the following axioms
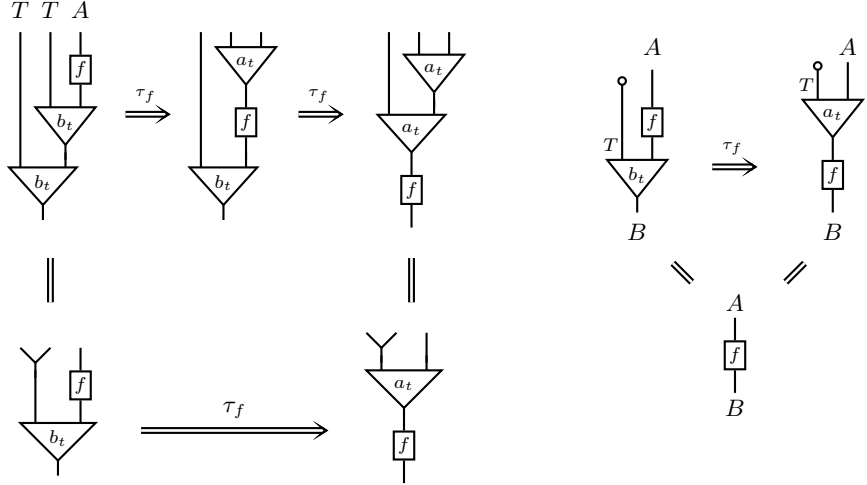
\[\psset{unit=0.1cm,labelsep=2pt,nodesep=3pt,linewidth=0.8pt}
\pspicture(0,-30)(110,40)
\begin{small}

%%%%%%%%%%%%%%%%%%%%%%%%%%%%%%%%%%%%%%%%%%%%%%%%%% box axiom

\rput(0,35){

\rput(2,-27){
\pcline[doubleline=true,arrowinset=0.7,arrowlength=0.8, arrowsize=3.5pt 1.5]{-}(0,0)(0,-8)
}

\rput(50,-27){
\pcline[doubleline=true,arrowinset=0.7,arrowlength=0.8, arrowsize=3.5pt 1.5]{-}(0,0)(0,-8)
}

%%%%%%%%%%%%%%%%%%%%%%%%%%%%%%%%%%%%%%%%%%%%%%% top left

\rput(0,0){

\rput(-2,5){$T$}
\rput(2,5){$T$}
\rput(6,5){$A$}

% \psline(-2,2)(0,0)(2,2)
% \psline(0,0)(0,-2)

\rput(4,-13)
{

\psline(0,0)(-4,5)(4,5)(0,0)
\psline(-2,5)(-2,15)
\psline(2,5)(2,15)

\rput(2,10){
\psframe*[linecolor=white](-1.5,-2)(1.5,2)
\psframe(-1.5,-2)(1.5,2)
\rput(0,0){\scr $f$}
}

\psline(0,0)(0,-2)
\rput(0,3){\scr $b_t$}
% \rput(0,-4.5){$B$}

\rput(-3,-8){
\psline(0,0)(-4.5,5)(4.5,5)(0,0)
\psline(-3,5)(-3,23)
\psline(3,5)(3,7)

\psline(0,0)(0,-2)
\rput(0,3){\scr $b_t$}
}

\rput(11,5){\Tmap{\tau_f}}

}

%%%%%%%%%%%%%%%%%%%%%%%%%%%%%%%%%%%%%%%%%%%%%%%%%%%%%%%%%%%%%%% top middle

\rput(17,0){

\rput(7,-13){

\rput(4,8){
\rput(-4,0){
% \rput(-2,10){$T$}
% \rput(2,10){$T$}
% \rput(6,10){$A$}
}

\psline(0,0)(-4,5)(4,5)(0,0)
\psline(-2,5)(-2,7)
\psline(2,5)(2,7)
\psline(0,0)(0,-10)
\rput(0,3){\scr $a_t$}
}

% \rput(0,-4.5){$B$}

\rput(4,3){
\psframe*[linecolor=white](-1.5,-2)(1.5,2)
\psframe(-1.5,-2)(1.5,2)
\rput(0,0){\scr $f$}
}

% \rput(0,-4.5){$B$}

\rput(1,-8){
\psline(0,0)(-4.5,5)(4.5,5)(0,0)
\psline(-3,5)(-3,23)
\psline(3,5)(3,7)

\psline(0,0)(0,-2)
\rput(0,3){\scr $b_t$}
}

}

}

\rput(38,-8){\Tmap{\tau_f}}

}
}

%%%%%%%%%%%%%%%%%%%%%%%%%%%%%%%%%%%%%%%%%%%%%%%%%%%%%%%%%%%%%%% top right

\rput(49,42){

\rput(0,-7){
% \rput(-2,5){$T$}
% \rput(2,5){$T$}
% \rput(6,5){$A$}
}

% \psline(-2,2)(0,0)(2,2)
% \psline(0,0)(0,-2)

\rput(4,-13)
{

\psline(0,0)(-4,5)(4,5)(0,0)
\psline(-2,5)(-2,8)
\psline(2,5)(2,8)

\psline(0,0)(0,-2)
\rput(0,3){\scr $a_t$}
% \rput(0,-4.5){$B$}

\rput(-3,-8){
\psline(0,0)(-4.5,5)(4.5,5)(0,0)
\psline(-3,5)(-3,16)
\psline(3,5)(3,7)

\psline(0,0)(0,-10)
\rput(0,3){\scr $a_t$}

\rput(0,-5){
\psframe*[linecolor=white](-1.5,-2)(1.5,2)
\psframe(-1.5,-2)(1.5,2)
\rput(0,0){\scr $f$}
}

}

}
}

%%%%%%%%%%%%%%%%%%%%%%%%%%%%%%%%%%%%%%%%%%%%%%%%%%%%%%%%%%%%%%% bottom left

\rput(4,-12){

\rput(0,0){
% \rput(-6,10){$T$}
% \rput(-2,10){$T$}
% \rput(2,10){$A$}

\rput(-4,5){
\psline(-2,2)(0,0)(2,2)
\psline(0,0)(0,-2)
\psline(6,2)(6,-2)
}

\rput(-1,-8){
\psline(0,0)(-5,5)(5,5)(0,0)
\psline(-3,5)(-3,12)
\psline(3,5)(3,12)
\psline(0,0)(0,-2)
\rput(0,3){\scr $b_t$}

\rput(3,10){
\psframe*[linecolor=white](-1.5,-2)(1.5,2)
\psframe(-1.5,-2)(1.5,2)
\rput(0,0){\scr $f$}
}

}

}

}

%%%%%%%%%%%%%%%%%%%%%%%%%%%%%%%%%%%%%%%%%%%%%%%%%%%%%%%%%%%%%%% bottom right

\rput(50,-10){

\rput(0,0){
% \rput(-6,10){$T$}
% \rput(-2,10){$T$}
% \rput(2,10){$A$}

\rput(-4,5){
\psline(-2,2)(0,0)(2,2)
\psline(0,0)(0,-2)
\psline(6,2)(6,-2)
}

\rput(-1,-3){
\psline(0,0)(-5,5)(5,5)(0,0)
\psline(-3,5)(-3,7)
\psline(3,5)(3,7)
\psline(0,0)(0,-10)
\rput(0,3){\scr $a_t$}

\rput(0,-5){
\psframe*[linecolor=white](-1.5,-2)(1.5,2)
\psframe(-1.5,-2)(1.5,2)
\rput(0,0){\scr $f$}
}

}

}

}

% \rput(72,2){\rotatebox{90}{$=$}}
% \rput(2,2){\rotatebox{90}{$=$}}

\rput(13,-16){
\pcline[doubleline=true,arrowinset=0.7,arrowlength=0.8, arrowsize=3.5pt 1.5]{->}(0,0)(27,0)\naput{$\tau_f$}
}

%%%%%%%%%%%%%%%%%%%%%%%%%%%%%%%%%%%%%%%%%%%%%%%%%%%%%%%%%% end box axiom

%%%%%%%%%%%%%%%%%%%%%%%%%%%%%%%%%%%%%%%%%%%%%%%%%%%%%%%%%%%% unit axiom

\rput(80,15){

\rput(0,4){
\rput(-3.5,3){\scr $T$}
\rput(2,16){$A$}
}

\psline(0,0)(-4,5)(4,5)(0,0)
\psline(-2,5)(-2,15)
\pscircle(-2,15.5){0.5}
\psline(2,5)(2,17)

\rput(2,10){
\psframe*[linecolor=white](-1.5,-2)(1.5,2)
\psframe(-1.5,-2)(1.5,2)
\rput(0,0){\scr $f$}
}

\psline(0,0)(0,-2)
\rput(0,3){\scr $b_t$}
\rput(0,-4.5){$B$}

\rput(13,5){\Tmap{\tau_f}}

\rput(26,0){

\rput(0,8){
\rput(-3.5,7){\scr $T$}
\rput(2,12){$A$}
% \rput(0,-5){$A$}

\psline(0,0)(-4,5)(4,5)(0,0)
\psline(-2,5)(-2,9)
\pscircle(-2,9.5){0.5}
\psline(2,5)(2,10)
\psline(0,0)(0,-10)
\rput(0,3){\scr $a_t$}
}

% \rput(0,-4.5){$B$}

\rput(0,3){
\psframe*[linecolor=white](-1.5,-2)(1.5,2)
\psframe(-1.5,-2)(1.5,2)
\rput(0,0){\scr $f$}
}

\rput(0,-4.5){$B$}

}

\rput(13,-24){
\rput(0,3){

\psline(0,-5)(0,5)
\rput(0,7){$A$}
\rput(0,-7){$B$}

\psframe*[linecolor=white](-1.5,-2)(1.5,2)
\psframe(-1.5,-2)(1.5,2)
\rput(0,0){\scr $f$}
}
}

\rput(4,-8){
\pcline[doubleline=true,arrowinset=0.7,arrowlength=0.8, arrowsize=3.5pt 1.5]{-}(0,0)(4,-4)
}

\rput(23,-8){
\pcline[doubleline=true,arrowinset=0.7,arrowlength=0.8, arrowsize=3.5pt 1.5]{-}(-4,-4)(0,0)
}

}

%%%%%%%%%%%%%%%%%%%%%%%%%%%%%%%%%%%%%%%%%%%%%%%%%%%%%%%%%%%%%%% end unit axiom

\end{small}
\endpspicture\]

\noi Note that as above, we can make this consistent with the string diagrams: if we are regarding a 0-cell $A$ as a 2-functor $1 \tra A$, then a 1-cell $f:A \tra B$ is a strict transformation, and a 2-cell $f \Tra g$ is a modification.

We now turn to weak maps for the composite monad $TS$ arising from a distributive law.  We know that a $TS$-algebra structure on $A$ can be expressed as a pair $(a_s, a_t)$ where $a_s$ is an $S$-algebra structure on $A$ and $a_t$ is a $T$-algebra structure on $A$ and they satisfy the interaction axiom.  Theorem~\ref{weakTSmap} says that a weak map of $TS$-algebras
\[\psset{nodesep=2pt}
\left(\pspicture(-3,8)(14,17)
\rput(0,0){
\rput(0,14){\rnode{a1}{$TA$}}  %  top
\rput(0,3){\rnode{a2}{$A$}}  % bottom
\ncline{->}{a1}{a2} \naput{{\scriptsize $a_t$}} % top
}
\rput(6,8){,}
\rput(10,0){
\rput(0,14){\rnode{a1}{$SA$}}  %  top
\rput(0,3){\rnode{a2}{$A$}}  % bottom
\ncline{->}{a1}{a2} \naput{{\scriptsize $a_s$}} % top
}
\endpspicture\right)
\ltra
\psset{nodesep=2pt}
\left(\pspicture(-3,8)(14,17)
\rput(0,0){
\rput(0,14){\rnode{a1}{$TB$}}  %  top
\rput(0,3){\rnode{a2}{$B$}}  % bottom
\ncline{->}{a1}{a2} \naput{{\scriptsize $b_t$}} % top
}
\rput(6,8){,}
\rput(10,0){
\rput(0,14){\rnode{a1}{$SB$}}  %  top
\rput(0,3){\rnode{a2}{$B$}}  % bottom
\ncline{->}{a1}{a2} \naput{{\scriptsize $b_s$}} % top
}
\endpspicture\right)
\]
is given by
\begin{itemize}
\item a 1-cell $A \tmap{f} B$
\item 2-cells $\sigma_f$ and $\tau_f$ as shown below giving the structure of a weak map of underlying $S$-algebras and a weak map of underlying $T$-algebras,

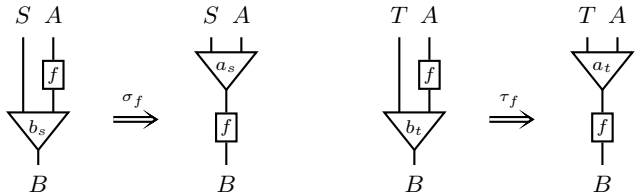
\[\psset{unit=0.1cm,labelsep=2pt,nodesep=3pt,linewidth=0.8pt}
\pspicture(0,55)(80,77)
\begin{small}

%%%%%%%%%%%%%%%%%%%%%%%%%%%%%%%%%%%%%%%%%%%% sigma

\rput(0,60){

\rput(0,4){
\rput(-2,14){$S$}
\rput(2,14){$A$}
}

\psline(0,0)(-4,5)(4,5)(0,0)
\psline(-2,5)(-2,15)
\psline(2,5)(2,15)

\rput(2,10){
\psframe*[linecolor=white](-1.5,-2)(1.5,2)
\psframe(-1.5,-2)(1.5,2)
\rput(0,0){\scr $f$}
}

\psline(0,0)(0,-2)
\rput(0,3){\scr $b_s$}
\rput(0,-4.5){$B$}

\rput(13,5){\Tmap{\sigma_f}}

\rput(25,0){

\rput(0,8){
\rput(-2,10){$S$}
\rput(2,10){$A$}
% \rput(0,-5){$A$}

\psline(0,0)(-4,5)(4,5)(0,0)
\psline(-2,5)(-2,7)
\psline(2,5)(2,7)
\psline(0,0)(0,-10)
\rput(0,3){\scr $a_s$}
}

% \rput(0,-4.5){$B$}

\rput(0,3){
\psframe*[linecolor=white](-1.5,-2)(1.5,2)
\psframe(-1.5,-2)(1.5,2)
\rput(0,0){\scr $f$}
}

\rput(0,-4.5){$B$}

}

}

%%%%%%%%%%%%%%%%%%%%%%%%%%%%%%%%%%%%%%%%%%%%%%%%%%%% tau

\rput(50,60){

\rput(0,4){
\rput(-2,14){$T$}
\rput(2,14){$A$}
}

\psline(0,0)(-4,5)(4,5)(0,0)
\psline(-2,5)(-2,15)
\psline(2,5)(2,15)

\rput(2,10){
\psframe*[linecolor=white](-1.5,-2)(1.5,2)
\psframe(-1.5,-2)(1.5,2)
\rput(0,0){\scr $f$}
}

\psline(0,0)(0,-2)
\rput(0,3){\scr $b_t$}
\rput(0,-4.5){$B$}

\rput(13,5){\Tmap{\tau_f}}

\rput(25,0){

\rput(0,8){
\rput(-2,10){$T$}
\rput(2,10){$A$}
% \rput(0,-5){$A$}

\psline(0,0)(-4,5)(4,5)(0,0)
\psline(-2,5)(-2,7)
\psline(2,5)(2,7)
\psline(0,0)(0,-10)
\rput(0,3){\scr $a_t$}
}

% \rput(0,-4.5){$B$}

\rput(0,3){
\psframe*[linecolor=white](-1.5,-2)(1.5,2)
\psframe(-1.5,-2)(1.5,2)
\rput(0,0){\scr $f$}
}

\rput(0,-4.5){$B$}

}

}

%%%%%%%%%%%%%%%%%%%%%%%%%%%%%%%%%%%%%%%%%%%%%%%%%%%%%%%%%%%%%%%%%% end tau

\end{small}
\endpspicture\]

\item satisfying the interaction axiom shown below.
\end{itemize}

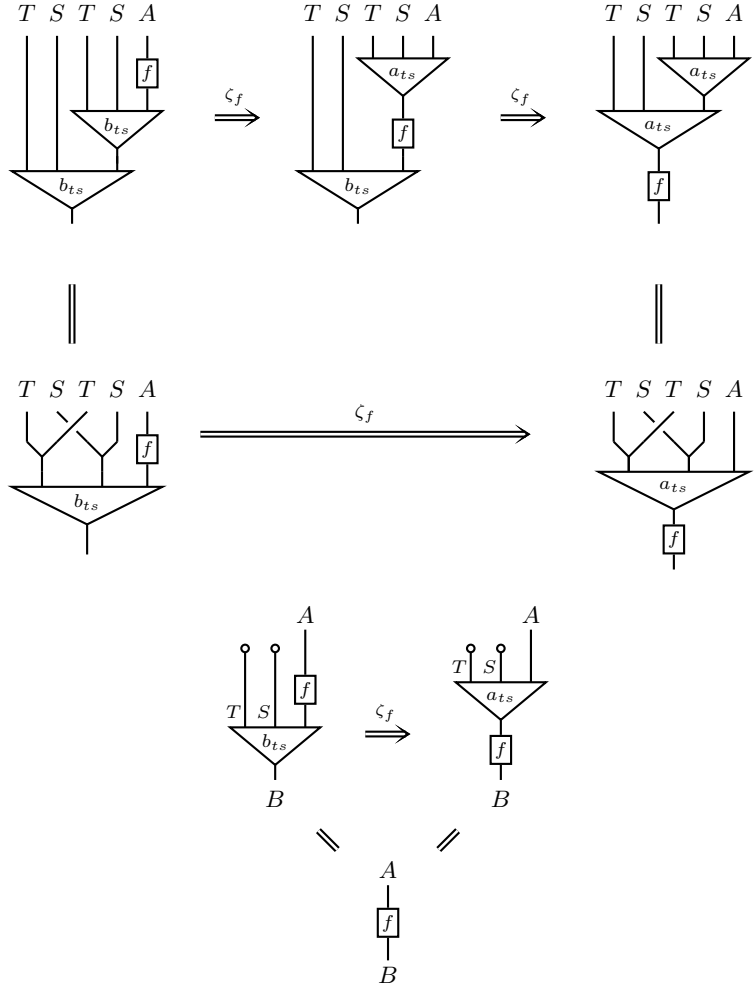
\[\psset{unit=0.1cm,labelsep=2pt,nodesep=3pt,linewidth=0.8pt}
\pspicture(0,-35)(72,43)
\begin{small}

%%%%%%%%%%%%%%%%%%%%%%%%%%%%%%%%%%%%%%%%%%%%%%%%%%%%%% interaction axiom

\rput(0,35){

%%%%%%%%%%%%%%%%%%%%%%%% top left

\rput(0,0){

\rput(0,4){
\rput(-2,5){$S$}
\rput(2,5){$T$}
\rput(6,5){$A$}
}

\rput(0,3){
\psline(-2,2)(2,-2)
\psline[border=2pt](2,2)(-2,-2)
}

% \psline(-2,2)(0,0)(2,2)
% \psline(0,0)(0,-2)

\rput(4,-13)
{

\psline(0,0)(-4,5)(4,5)(0,0)
\psline(-2,5)(-2,14)
\psline(2,5)(2,18)

\rput(2,10){
\psframe*[linecolor=white](-1.5,-2)(1.5,2)
\psframe(-1.5,-2)(1.5,2)
\rput(0,0){\scr $f$}
}

\psline(0,0)(0,-2)
\rput(0,3){\scr $b_s$}
% \rput(0,-4.5){$B$}

\rput(-3,-8){
\psline(0,0)(-4.5,5)(4.5,5)(0,0)
\psline(-3,5)(-3,22)
\psline(3,5)(3,7)

\psline(0,0)(0,-2)
\rput(0,3){\scr $b_t$}
}

\rput(15,5){\Tmap{\sigma_f}}

}
}

%%%%%%%%%%%%%%%%%%%%%%%%%%%%%%%%%%%%%%%%%%%%%%%%%%%%%%%%%%%%% top middle

\rput(30,0){

\rput(4,-13){

\rput(4,8){
\rput(-4,4){
\rput(-2,10){$S$}
\rput(2,10){$T$}
\rput(6,10){$A$}
}

\rput(0,-1){

\psline(-6,11)(-2,7)
\psline[border=2pt](-2,11)(-6,7)

\psline(0,0)(-4,5)(4,5)(0,0)
\psline(-2,5)(-2,7)
\psline(2,5)(2,11)
\psline(0,0)(0,-9)
\rput(0,3){\scr $a_s$}
}
}

% \rput(0,-4.5){$B$}

\rput(4,2){
\psframe*[linecolor=white](-1.5,-2)(1.5,2)
\psframe(-1.5,-2)(1.5,2)
\rput(0,0){\scr $f$}
}

% \rput(0,-4.5){$B$}

\rput(1,-8){
\psline(0,0)(-4.5,5)(4.5,5)(0,0)
\psline(-3,5)(-3,22)
\psline(3,5)(3,7)

\psline(0,0)(0,-2)
\rput(0,3){\scr $b_t$}
}

}

}

\rput(55,-8){\Tmap{\tau_f}}

}

%%%%%%%%%%%%%%%%%%%%%%%%%%%%%%%%%%%%%%%%%%%%%%%%%%%%%%%%%%%%%%% top right

\rput(70,42){

\rput(0,-3){
\rput(-2,5){$S$}
\rput(2,5){$T$}
\rput(6,5){$A$}
}

\rput(0,-4){
\psline(-2,2)(2,-2)
\psline[border=2pt](2,2)(-2,-2)
}

% \psline(-2,2)(0,0)(2,2)
% \psline(0,0)(0,-2)

\rput(4,-13)
{

\psline(0,0)(-4,5)(4,5)(0,0)
\psline(-2,5)(-2,7)
\psline(2,5)(2,11)

\psline(0,0)(0,-2)
\rput(0,3){\scr $a_s$}
% \rput(0,-4.5){$B$}

\rput(-3,-8){
\psline(0,0)(-4.5,5)(4.5,5)(0,0)
\psline(-3,5)(-3,15)
\psline(3,5)(3,7)

\psline(0,0)(0,-10)
\rput(0,3){\scr $a_t$}

\rput(0,-5){
\psframe*[linecolor=white](-1.5,-2)(1.5,2)
\psframe(-1.5,-2)(1.5,2)
\rput(0,0){\scr $f$}
}

}

}
}

%%%%%%%%%%%%%%%%%%%%%%%%%%%%%%%%%%%%%%%%%%%%%%%%%%%%%%%%%%%%%%% bottom left

\rput(2,9){
\pcline[doubleline=true,arrowinset=0.7,arrowlength=0.8, arrowsize=3.5pt 1.5]{-}(0,0)(0,-6)
}

\rput(0,-8){

\rput(-2,5){$S$}
\rput(2,5){$T$}
\rput(6,5){$A$}

% \psline(-2,2)(0,0)(2,2)
% \psline(0,0)(0,-2)

\rput(4,-13)
{

\psline(0,0)(-4,5)(4,5)(0,0)
\psline(-2,5)(-2,15)
\psline(2,5)(2,15)

\rput(2,10){
\psframe*[linecolor=white](-1.5,-2)(1.5,2)
\psframe(-1.5,-2)(1.5,2)
\rput(0,0){\scr $f$}
}

\psline(0,0)(0,-2)
\rput(0,3){\scr $b_t$}
% \rput(0,-4.5){$B$}

\rput(-3,-8){
\psline(0,0)(-4.5,5)(4.5,5)(0,0)
\psline(-3,5)(-3,23)
\psline(3,5)(3,7)

\psline(0,0)(0,-2)
\rput(0,3){\scr $b_s$}
}

\rput(15,5){\Tmap{\tau_f}}

%%%%%%%%%%%%%%%%%%%%%%%%%%%%%%%%%%%%%%%%%%%%%%%%%%%%%%%%%%%%%%% bottom middle

\rput(30,0){

\rput(4,8){
\rput(-4,0){
\rput(-2,10){$S$}
\rput(2,10){$T$}
\rput(6,10){$A$}
}

\psline(0,0)(-4,5)(4,5)(0,0)
\psline(-2,5)(-2,7)
\psline(2,5)(2,7)
\psline(0,0)(0,-10)
\rput(0,3){\scr $a_t$}
}

% \rput(0,-4.5){$B$}

\rput(4,3){
\psframe*[linecolor=white](-1.5,-2)(1.5,2)
\psframe(-1.5,-2)(1.5,2)
\rput(0,0){\scr $f$}
}

% \rput(0,-4.5){$B$}

\rput(1,-8){
\psline(0,0)(-4.5,5)(4.5,5)(0,0)
\psline(-3,5)(-3,23)
\psline(3,5)(3,7)

\psline(0,0)(0,-2)
\rput(0,3){\scr $b_s$}
}

}

}

\rput(55,-8){\Tmap{\sigma_f}}

}

%%%%%%%%%%%%%%%%%%%%%%%%%%%%%%%%%%%%%%%%%%%%%%%%%%%%%%%%%%%%%%%  bottom right

\rput(72,9){
\pcline[doubleline=true,arrowinset=0.7,arrowlength=0.8, arrowsize=3.5pt 1.5]{-}(0,0)(0,-6)
}

\rput(70,-1){

\rput(0,-7){
\rput(-2,5){$S$}
\rput(2,5){$T$}
\rput(6,5){$A$}
}

% \psline(-2,2)(0,0)(2,2)
% \psline(0,0)(0,-2)

\rput(4,-13)
{

\psline(0,0)(-4,5)(4,5)(0,0)
\psline(-2,5)(-2,8)
\psline(2,5)(2,8)

\psline(0,0)(0,-2)
\rput(0,3){\scr $a_t$}
% \rput(0,-4.5){$B$}

\rput(-3,-8){
\psline(0,0)(-4.5,5)(4.5,5)(0,0)
\psline(-3,5)(-3,16)
\psline(3,5)(3,7)

\psline(0,0)(0,-10)
\rput(0,3){\scr $a_s$}

\rput(0,-5){
\psframe*[linecolor=white](-1.5,-2)(1.5,2)
\psframe(-1.5,-2)(1.5,2)
\rput(0,0){\scr $f$}
}

}

}
}

%%%%%%%%%%%%%%%%%%%%%%%%%%%%%%%%%%%%%%%%%%%%%%%%%%%%%% end interaction axiom

\end{small}
\endpspicture\]

\begin{proof}
Now, \emph{a priori} a weak map of $TS$-algebras is a 1-cell $A \tmap{f} B$ and a 2-cell

\[\psset{unit=0.1cm,labelsep=2pt,nodesep=3pt,linewidth=0.8pt}
\pspicture(0,40)(30,63)
\begin{small}

\rput(0,50){

\rput(0,1){
\rput(-4,10){$T$}
\rput(0,10){$S$}
\rput(4,10){$A$}
}

\rput(0,-5){
\psline(0,0)(-6,5)(6,5)(0,0)
\psline(-4,5)(-4,13)
\psline(0,5)(0,13)
\psline(4,5)(4,13)
\psline(0,0)(0,-2)

\rput(4,9){
\psframe*[linecolor=white](-1.5,-2)(1.5,2)
\psframe(-1.5,-2)(1.5,2)
\rput(0,0){\scr $f$}
}

\rput(0,3){\scr $b_{ts}$}
}

\rput(15,5){\Tmap{\zeta_f}}

}

\rput(30,50){

\rput(0,1){
\rput(-4,10){$T$}
\rput(0,10){$S$}
\rput(4,10){$A$}
}

\rput(0,1){
\psline(0,0)(-6,5)(6,5)(0,0)
\psline(-4,5)(-4,7)
\psline(0,5)(0,7)
\psline(4,5)(4,7)
\psline(0,0)(0,-8)

\rput(0,-4){
\psframe*[linecolor=white](-1.5,-2)(1.5,2)
\psframe(-1.5,-2)(1.5,2)
\rput(0,0){\scr $f$}
}

\rput(0,3){\scr $a_{ts}$}
}

}

\end{small}
\endpspicture\]

\noi such that the following diagrams commute

\[\psset{unit=0.1cm,labelsep=2pt,nodesep=3pt,linewidth=0.8pt}
\pspicture(0,-90)(80,50)
\begin{small}

%%%%%%%%%%%%%%%%%%%%%%%%%%%%%%%%%%%%%%%%%%%%%%%%%% box axiom

% \rput(0,35){

%%%%%%%%%%%%%%%%%%%%%%%%%%%%%%%%%%%%%%%%% top left

\rput(0,35){

\rput(-8,5){$T$}
\rput(-4,5){$S$}
\rput(0,5){$T$}
\rput(4,5){$S$}
\rput(8,5){$A$}

\rput(4,-13)
{

\rput(0,0){
\psline(0,0)(-6,5)(6,5)(0,0)
\psline(-4,5)(-4,15)
\psline(0,5)(0,15)
\psline(4,5)(4,15)
\psline(0,0)(0,-2)
\rput(0,3){\scr $b_{ts}$}
}

\rput(4,10){
\psframe*[linecolor=white](-1.5,-2)(1.5,2)
\psframe(-1.5,-2)(1.5,2)
\rput(0,0){\scr $f$}
}

\rput(-6,-8){
\psline(0,0)(-8,5)(8,5)(0,0)
\psline(-6,5)(-6,23)
\psline(-2,5)(-2,23)
\psline(6,5)(6,7)
\psline(0,0)(0,-2)
\rput(0,3){\scr $b_{ts}$}
}

\rput(16,5){\Tmap{\zeta_f}}

}

}

%%%%%%%%%%%%%%%%%%%%%%%%%%%%%%%%%%%%%%%%%%%%%%%%%%%%%%%%%%%%%%% top middle

\rput(38,35){

\rput(-8,5){$T$}
\rput(-4,5){$S$}
\rput(0,5){$T$}
\rput(4,5){$S$}
\rput(8,5){$A$}

\rput(4,-13)
{

\rput(0,7){
\psline(0,0)(-6,5)(6,5)(0,0)
\psline(-4,5)(-4,8)
\psline(0,5)(0,8)
\psline(4,5)(4,8)
\psline(0,0)(0,-8)
\rput(0,3){\scr $a_{ts}$}
}

\rput(0,2){
\psframe*[linecolor=white](-1.5,-2)(1.5,2)
\psframe(-1.5,-2)(1.5,2)
\rput(0,0){\scr $f$}
}

\rput(-6,-8){
\psline(0,0)(-8,5)(8,5)(0,0)
\psline(-6,5)(-6,23)
\psline(-2,5)(-2,23)
\psline(6,5)(6,7)

\psline(0,0)(0,-2)
\rput(0,3){\scr $b_{ts}$}
}

\rput(16,5){\Tmap{\zeta_f}}

}

}

%%%%%%%%%%%%%%%%%%%%%%%%%%%%%%%%%%%%%%%%%%%%%%%%%%%%%%%%%%%%%%% top right

\rput(78,35){

\rput(-8,5){$T$}
\rput(-4,5){$S$}
\rput(0,5){$T$}
\rput(4,5){$S$}
\rput(8,5){$A$}

\rput(4,-13)
{

\rput(0,7){
\psline(0,0)(-6,5)(6,5)(0,0)
\psline(-4,5)(-4,8)
\psline(0,5)(0,8)
\psline(4,5)(4,8)
\psline(0,0)(0,-2)
\rput(0,3){\scr $a_{ts}$}
}

\rput(-6,0){
\psline(0,0)(-8,5)(8,5)(0,0)
\psline(-6,5)(-6,15)
\psline(-2,5)(-2,15)
\psline(6,5)(6,7)
\psline(0,0)(0,-10)
\rput(0,3){\scr $a_{ts}$}
}

\rput(-6,-5){
\psframe*[linecolor=white](-1.5,-2)(1.5,2)
\psframe(-1.5,-2)(1.5,2)
\rput(0,0){\scr $f$}
}

}

}

\rput(-2,5){
\pcline[doubleline=true,arrowinset=0.7,arrowlength=0.8, arrowsize=3.5pt 1.5]{-}(0,0)(0,-10)
}

\rput(76,5){
\pcline[doubleline=true,arrowinset=0.7,arrowlength=0.8, arrowsize=3.5pt 1.5]{-}(0,0)(0,-10)
}

\rput(14,-16){
\pcline[doubleline=true,arrowinset=0.7,arrowlength=0.8, arrowsize=3.5pt 1.5]{->}(0,0)(46,0)\naput{\scr $\zeta_f$}
}

%%%%%%%%%%%%%%%%%%%%%%%%%%%%%%%%%%%%%%%%%%%%%%%%%%%%%%%%%%%%%%% bottom left

\rput(0,-20){

\rput(0,0){

\rput(0,5){
\rput(-8,5){$T$}
\rput(-4,5){$S$}
\rput(0,5){$T$}
\rput(4,5){$S$}
\rput(8,5){$A$}
}

\rput(-2,5){
\psline(-6,2)(-6,-2)
\psline(-2,2)(2,-2)
\psline[border=2pt](2,2)(-2,-2)
\psline(6,2)(6,-2)
}

\rput(-6,1){
\psline(-2,2)(0,0)(2,2)
\psline(0,0)(0,-2)
}

\rput(2,1){
\psline(-2,2)(0,0)(2,2)
\psline(0,0)(0,-2)
}

\rput(-1,-8){

\rput(1,0){
\psline(0,0)(-10,5)(10,5)(0,0)
\psline(-6,5)(-6,7)
\psline(2,5)(2,7)
\psline(8,5)(8,15)
\psline(0,0)(0,-4)
\rput(0,3){\scr $b_{ts}$}
}

\rput(9,10){
\psframe*[linecolor=white](-1.5,-2)(1.5,2)
\psframe(-1.5,-2)(1.5,2)
\rput(0,0){\scr $f$}
}

}

}

}

%%%%%%%%%%%%%%%%%%%%%%%%%%%%%%%%%%%%%%%%%%%%%%%%%%%%%%%%%%%%%%% bottom right

\rput(78,-20){

\rput(0,0){

\rput(0,5){
\rput(-8,5){$T$}
\rput(-4,5){$S$}
\rput(0,5){$T$}
\rput(4,5){$S$}
\rput(8,5){$A$}
}

\rput(-2,5){
\psline(-6,2)(-6,-2)
\psline(-2,2)(2,-2)
\psline[border=2pt](2,2)(-2,-2)
\psline(6,2)(6,-2)
}

\rput(-6,1){
\psline(-2,2)(0,0)(2,2)
\psline(0,0)(0,-2)
}

\rput(2,1){
\psline(-2,2)(0,0)(2,2)
\psline(0,0)(0,-2)
}

\rput(-1,-8){

\rput(1,2){
\psline(0,0)(-10,5)(10,5)(0,0)
\psline(-6,5)(-6,7)
\psline(2,5)(2,7)
\psline(8,5)(8,13)
\psline(0,0)(0,-8)
\rput(0,3){\scr $a_{ts}$}
}

\rput(1,-2){
\psframe*[linecolor=white](-1.5,-2)(1.5,2)
\psframe(-1.5,-2)(1.5,2)
\rput(0,0){\scr $f$}
}

}

}

}

%%%%%%%%%%%%%%%%%%%%%%%%%%%%%%%%%%%%%%%%%%%%%%%%%%%%%%%%%% end box axiom

%%%%%%%%%%%%%%%%%%%%%%%%%%%%%%%%%%%%%%%%%%%%%%%%%%%%%%%%%%%% unit axiom

\rput(-85,-75){

%%%%%%%%%%%%%%%%%%%%%%%%%%%%%%%%%%%%%%%%%%%%%%%%%%%% top left

\rput(110,15){

\rput(0,4){
\rput(-5.5,3){\scr $T$}
\rput(-1.5,3){\scr $S$}
\rput(4,16){$A$}
}

\psline(0,0)(-6,5)(6,5)(0,0)

\rput(-4,15){
\psline(0,0)(0,-10)
\pscircle(0,0.5){0.5}
}

\rput(0,15){
\psline(0,0)(0,-10)
\pscircle(0,0.5){0.5}
}

\psline(4,5)(4,18)

\rput(4,10){
\psframe*[linecolor=white](-1.5,-2)(1.5,2)
\psframe(-1.5,-2)(1.5,2)
\rput(0,0){\scr $f$}
}

\psline(0,0)(0,-2)
\rput(0,3){\scr $b_{ts}$}
\rput(0,-4.5){$B$}

%%%%%%%%%%%%%%%%%%%%%%%%%%%%%%%%%%%%%%%%%%%%%%%%%%%%%%%%%%%%%%%%% top right

\rput(15,5){\Tmap{\zeta_f}}

\rput(30,0){

\rput(0,6){
\rput(-5.5,7){\scr $T$}
\rput(-1.5,7){\scr $S$}
\rput(4,14){$A$}

\psline(0,0)(-6,5)(6,5)(0,0)
% \psline(-2,5)(-2,9)
% \pscircle(-2,9.5){0.5}
% \psline(2,5)(2,10)
% \psline(0,0)(0,-10)

\rput(-4,9){
\psline(0,0)(0,-4)
\pscircle(0,0.5){0.5}
}

\rput(0,9){
\psline(0,0)(0,-4)
\pscircle(0,0.5){0.5}
}

\psline(4,5)(4,12)

\rput(0,3){\scr $a_{ts}$}
}

% \rput(0,-4.5){$B$}

\psline(0,6)(0,-2)

\rput(0,2){
\psframe*[linecolor=white](-1.5,-2)(1.5,2)
\psframe(-1.5,-2)(1.5,2)
\rput(0,0){\scr $f$}
}

\rput(0,-4.5){$B$}

}

\rput(15,-24){
\rput(0,3){

\psline(0,-5)(0,5)
\rput(0,7){$A$}
\rput(0,-7){$B$}

\psframe*[linecolor=white](-1.5,-2)(1.5,2)
\psframe(-1.5,-2)(1.5,2)
\rput(0,0){\scr $f$}
}
}

\rput(5,-8){
\pcline[doubleline=true,arrowinset=0.7,arrowlength=0.8, arrowsize=3.5pt 1.5]{-}(0,0)(4,-4)
}

\rput(25,-8){
\pcline[doubleline=true,arrowinset=0.7,arrowlength=0.8, arrowsize=3.5pt 1.5]{-}(-4,-4)(0,0)
}

}

%%%%%%%%%%%%%%%%%%%%%%%%%%%%%%%%%%%%%%%%%%%%%%%%%%%%%%%%%%%%%%% end unit axiom

}

\end{small}
\endpspicture\]

% \[\psset{unit=0.1cm,labelsep=2pt,nodesep=3pt,linewidth=0.8pt}
% \pspicture(0,30)(30,90)
% \begin{small}
% 
% 
% 
% \rput(0,50){
% 
% \rput(0,1){
% \rput(-4,10){$T$}
% \rput(0,10){$S$}
% \rput(4,10){$A$}
% }
% 
% 
% \rput(0,-5){
% \psline(0,0)(-6,5)(6,5)(0,0)
% \psline(-4,5)(-4,13)
% \psline(0,5)(0,13)
% \psline(4,5)(4,13)
% \psline(0,0)(0,-2)
% 
% \rput(4,9){
% \psframe*[linecolor=white](-1.5,-2)(1.5,2)
% \psframe(-1.5,-2)(1.5,2)
% \rput(0,0){\scr $f$}
% }
% 
% 
% \rput(0,3){\scr $b_{ts}$}
% }
% 
% 
% \rput(15,5){\Tmap{\phi_f}}
% 
% 
% }
% 
% 
% 
% 
% \rput(30,50){
% 
% \rput(0,1){
% \rput(-4,10){$T$}
% \rput(0,10){$S$}
% \rput(4,10){$A$}
% }
% 
% 
% \rput(0,1){
% \psline(0,0)(-6,5)(6,5)(0,0)
% \psline(-4,5)(-4,7)
% \psline(0,5)(0,7)
% \psline(4,5)(4,7)
% \psline(0,0)(0,-8)
% 
% \rput(0,-4){
% \psframe*[linecolor=white](-1.5,-2)(1.5,2)
% \psframe(-1.5,-2)(1.5,2)
% \rput(0,0){\scr $f$}
% }
% 
% 
% \rput(0,3){\scr $a_{ts}$}
% }
% 
% 
% 
% }
% 
% 
% 
% 
% 
% \end{small}
% \endpspicture\]

\noi First we will show that this gives a pair $(\sigma_f, \tau_f)$ satisfying interaction.  We begin by expressing the $TS$-algebras $a_{ts}$ and $b_{ts}$  as $\lambda$-distributive pairs:

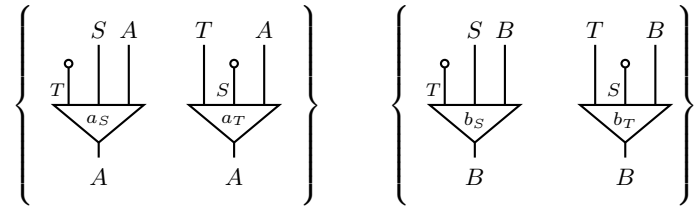
\[\psset{unit=0.1cm,labelsep=2pt,nodesep=3pt,linewidth=0.8pt}
\pspicture(0,-10)(70,20)
\begin{small}

%%%%%%%%%%%%%%%%%%%%%%%%%%%%%%%%%%%%%%%%%%%%%%%%%%%%%%%%%%%%%%%%%% a_s

\rput(0,0){

\rput(0,4){
\rput(-5.5,3){\scr $T$}
\rput(0,11){$S$}
\rput(4,11){$A$}
}

\psline(0,0)(-6,5)(6,5)(0,0)

\rput(-4,10){
\psline(0,0)(0,-5)
\pscircle(0,0.5){0.5}
}

\psline(4,5)(4,13)
\psline(0,5)(0,13)

\psline(0,0)(0,-2)
\rput(0,3){\scr $a_{S}$}
\rput(0,-4.5){$A$}

}

%%%%%%%%%%%%%%%%%%%%%%%%%%%%%%%%%%%%%%%%%%%%%%%%%%%%%%%%%%%%%%%%%% a_t

\rput(18,0){

\rput(0,4){
\rput(-4,11){$T$}
\rput(-1.5,3){\scr $S$}
\rput(4,11){$A$}
}

\psline(0,0)(-6,5)(6,5)(0,0)

\rput(0,10){
\psline(0,0)(0,-5)
\pscircle(0,0.5){0.5}
}

\psline(4,5)(4,13)
\psline(-4,5)(-4,13)

\psline(0,0)(0,-2)
\rput(0,3){\scr $a_{T}$}
\rput(0,-4.5){$A$}

}

\rput(-9,5){
$\left\{\begin{array}{c}\vs{7}\end{array}\right.$}

\rput(26,5){
$\left.\begin{array}{c}\vs{7}\end{array}\right\}$}

\rput(50,0){
\rput(-9,5){
$\left\{\begin{array}{c}\vs{7}\end{array}\right.$}

\rput(26,5){
$\left.\begin{array}{c}\vs{7}\end{array}\right\}$}
}

%%%%%%%%%%%%%%%%%%%%%%%%%%%%%%%%%%%%%%%%%%%%%%%%%%%%%%%%%%%%%%%%%% b_s

\rput(50,0){

\rput(0,4){
\rput(-5.5,3){\scr $T$}
\rput(0,11){$S$}
\rput(4,11){$B$}
}

\psline(0,0)(-6,5)(6,5)(0,0)

\rput(-4,10){
\psline(0,0)(0,-5)
\pscircle(0,0.5){0.5}
}

\psline(4,5)(4,13)
\psline(0,5)(0,13)

\psline(0,0)(0,-2)
\rput(0,3){\scr $b_{S}$}
\rput(0,-4.5){$B$}

}

%%%%%%%%%%%%%%%%%%%%%%%%%%%%%%%%%%%%%%%%%%%%%%%%%%%%%%%%%%%%%%%%%% b_t

\rput(70,0){

\rput(0,4){
\rput(-4,11){$T$}
\rput(-1.5,3){\scr $S$}
\rput(4,11){$B$}
}

\psline(0,0)(-6,5)(6,5)(0,0)

\rput(0,10){
\psline(0,0)(0,-5)
\pscircle(0,0.5){0.5}
}

\psline(4,5)(4,13)
\psline(-4,5)(-4,13)

\psline(0,0)(0,-2)
\rput(0,3){\scr $b_{T}$}
\rput(0,-4.5){$B$}

}

\end{small}
\endpspicture\]

Next we make $\sigma$ for the $S$ components and $\tau$ for the $T$ components.

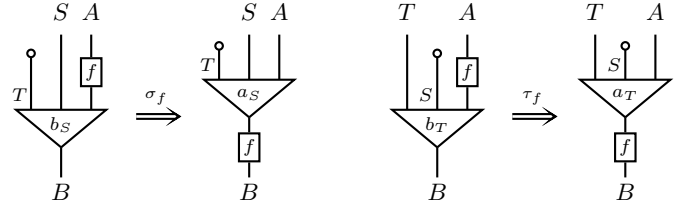
\[\psset{unit=0.1cm,labelsep=2pt,nodesep=3pt,linewidth=0.8pt}
\pspicture(0,-10)(80,20)
\begin{small}

%%%%%%%%%%%%%%%%%%%%%%%%%%%%%%%%%%%%%%%%%%%%%%%%%%%%%%%%%%%%%%%%%% b_s

\rput(0,0){

\rput(0,4){
\rput(-5.5,3){\scr $T$}
\rput(0,14){$S$}
\rput(4,14){$A$}
}

\psline(0,0)(-6,5)(6,5)(0,0)

\rput(-4,12){
\psline(0,0)(0,-7)
\pscircle(0,0.5){0.5}
}

\psline(4,5)(4,15)
\psline(0,5)(0,15)

\rput(4,10){
\psframe*[linecolor=white](-1.5,-2)(1.5,2)
\psframe(-1.5,-2)(1.5,2)
\rput(0,0){\scr $f$}
}

\psline(0,0)(0,-4)
\rput(0,3){\scr $b_{S}$}
\rput(0,-6){$B$}

}

%%%%%%%%%%%%%%%%%%%%%%%%%%%%%%%%%%%%%%%%%%%%%%%%%%%%%%%%%%%%%%%%%% a_s

\rput(25,0){

\rput(0,8){
\rput(-5.5,3){\scr $T$}
\rput(0,10){$S$}
\rput(4,10){$A$}
}

\rput(0,4){
\psline(0,0)(-6,5)(6,5)(0,0)
\rput(0,3){\scr $a_{S}$}

}

\rput(-4,13){
\psline(0,0)(0,-4)
\pscircle(0,0.5){0.5}
}

\rput(0,4){
\psline(4,5)(4,11)
\psline(0,5)(0,11)
}

\psline(0,4)(0,-4)
\rput(0,-6){$B$}

\rput(0,0){
\psframe*[linecolor=white](-1.5,-2)(1.5,2)
\psframe(-1.5,-2)(1.5,2)
\rput(0,0){\scr $f$}
}

}

\rput(13,5){\Tmap{\sigma_f}}

%%%%%%%%%%%%%%%%%%%%%%%%%%%%%%%%%%%%%%%%%%%%%%%%%%%%%%%%%%%%%%%%%% b_t

\rput(50,0){

\rput(0,4){
\rput(-4,14){$T$}
\rput(-1.5,3){\scr $S$}
\rput(4,14){$A$}
}

\psline(0,0)(-6,5)(6,5)(0,0)

\rput(0,12){
\psline(0,0)(0,-7)
\pscircle(0,0.5){0.5}
}

\psline(4,5)(4,15)
\psline(-4,5)(-4,15)

\rput(4,10){
\psframe*[linecolor=white](-1.5,-2)(1.5,2)
\psframe(-1.5,-2)(1.5,2)
\rput(0,0){\scr $f$}
}

\psline(0,0)(0,-4)
\rput(0,3){\scr $b_{T}$}
\rput(0,-6){$B$}

\rput(13,5){\Tmap{\tau_f}}

}

%%%%%%%%%%%%%%%%%%%%%%%%%%%%%%%%%%%%%%%%%%%%%%%%%%%%%%%%%%%%%%%%%% a_t

\rput(75,0){

\rput(0,8){
\rput(-4,10){$T$}
\rput(-1.5,3){\scr $S$}
\rput(4,10){$A$}
}

\rput(0,4){
\psline(0,0)(-6,5)(6,5)(0,0)
\rput(0,3){\scr $a_{T}$}

}

\rput(0,13){
\psline(0,0)(0,-4)
\pscircle(0,0.5){0.5}
}

\rput(0,4){
\psline(4,5)(4,11)
\psline(-4,5)(-4,11)
}

\psline(0,4)(0,-4)
\rput(0,-6){$B$}

\rput(0,0){
\psframe*[linecolor=white](-1.5,-2)(1.5,2)
\psframe(-1.5,-2)(1.5,2)
\rput(0,0){\scr $f$}
}

}

\end{small}
\endpspicture\]

% \newpage

\noi We check the interaction diagram:

% hello

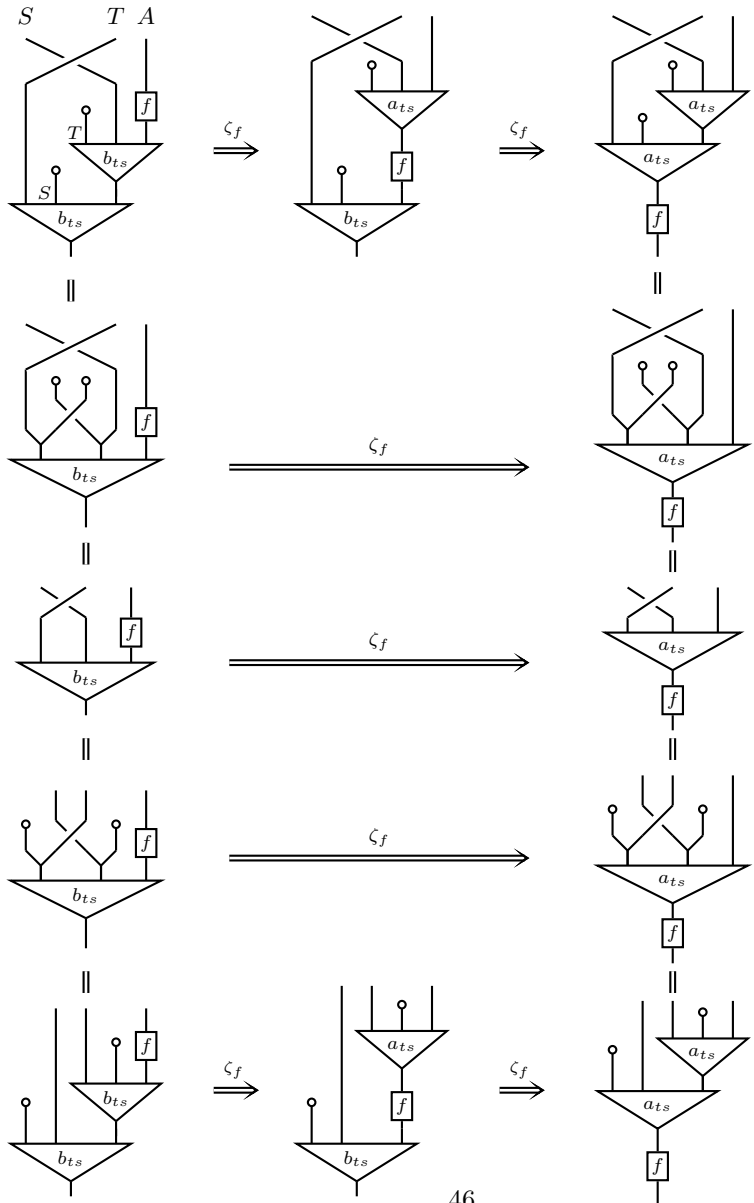
\[\psset{unit=0.1cm,labelsep=2pt,nodesep=3pt,linewidth=0.8pt}
\pspicture(0,-100)(100,41)
\begin{small}

%%%%%%%%%%%%%%%%%%%%%%%%%%%%%%%%%%%%%%%%% top left

\rput(0,35){

\rput(0,5){

\rput(0,-1){
\rput(-8,5){$S$}
\rput(4,5){$T$}
\rput(8,5){$A$}
}

\rput(-5.5,-19.5){\scr $S$}
\rput(-1.5,-11.5){\scr $T$}

}

\rput(4,-13)
{

\rput(0,0){
\psline(0,0)(-6,5)(6,5)(0,0)
\psline(-4,5)(-4,9)
\pscircle(-4,9.5){0.5}
\psline(0,5)(0,13)
\psline(4,5)(4,19)
\psline(0,0)(0,-2)
\rput(0,3){\scr $b_{ts}$}

\psline(0,13)(-12,19)
\psline[border=2pt](0,19)(-12,13)

}

\rput(4,10){
\psframe*[linecolor=white](-1.5,-2)(1.5,2)
\psframe(-1.5,-2)(1.5,2)
\rput(0,0){\scr $f$}
}

\rput(-6,-8){
\psline(0,0)(-8,5)(8,5)(0,0)
\psline(-6,5)(-6,21)
\psline(-2,5)(-2,9)
\pscircle(-2,9.5){0.5}
\psline(6,5)(6,7)
\psline(0,0)(0,-2)
\rput(0,3){\scr $b_{ts}$}
}

\rput(16,5){\Tmap{\zeta_f}}

}

}

%%%%%%%%%%%%%%%%%%%%%%%%%%%%%%%%%%%%%%%%%%%%%%%%%%%%%%%%%%%%%%% top middle

\rput(38,35){

% \rput(-8,5){$T$}
% \rput(-4,5){$S$}
% \rput(0,5){$T$}
% \rput(4,5){$S$}
% \rput(8,5){$A$}

\rput(4,-13)
{

\rput(0,7){
\psline(0,0)(-6,5)(6,5)(0,0)
\psline(-4,5)(-4,8)
\pscircle(-4,8.5){0.5}
\psline(0,5)(0,9)
\psline(4,5)(4,15)
\psline(0,0)(0,-8)
\rput(0,3){\scr $a_{ts}$}

\rput(0,-4){
\psline(0,13)(-12,19)
\psline[border=2pt](0,19)(-12,13)
}

}

\rput(0,2){
\psframe*[linecolor=white](-1.5,-2)(1.5,2)
\psframe(-1.5,-2)(1.5,2)
\rput(0,0){\scr $f$}
}

\rput(-6,-8){
\psline(0,0)(-8,5)(8,5)(0,0)
\psline(-6,5)(-6,24)
\psline(-2,5)(-2,9)
\pscircle(-2,9.5){0.5}
\psline(6,5)(6,7)

\psline(0,0)(0,-2)
\rput(0,3){\scr $b_{ts}$}
}

\rput(16,5){\Tmap{\zeta_f}}

}

}

%%%%%%%%%%%%%%%%%%%%%%%%%%%%%%%%%%%%%%%%%%%%%%%%%%%%%%%%%%%%%%% top right

\rput(78,35){

% \rput(-8,5){$T$}
% \rput(-4,5){$S$}
% \rput(0,5){$T$}
% \rput(4,5){$S$}
% \rput(8,5){$A$}

\rput(4,-13)
{

\rput(0,7){
\psline(0,0)(-6,5)(6,5)(0,0)
\psline(-4,5)(-4,8)
\pscircle(-4,8.5){0.5}
\psline(0,5)(0,9)
\psline(4,5)(4,15)
\psline(0,0)(0,-2)
\rput(0,3){\scr $a_{ts}$}

\rput(0,-4){
\psline(0,13)(-12,19)
\psline[border=2pt](0,19)(-12,13)
}

}

\rput(-6,0){
\psline(0,0)(-8,5)(8,5)(0,0)
\psline(-6,5)(-6,16)
\psline(-2,5)(-2,8)
\pscircle(-2,8.5){0.5}
\psline(6,5)(6,7)
\psline(0,0)(0,-10)
\rput(0,3){\scr $a_{ts}$}
}

\rput(-6,-5){
\psframe*[linecolor=white](-1.5,-2)(1.5,2)
\psframe(-1.5,-2)(1.5,2)
\rput(0,0){\scr $f$}
}

}

}

\rput(-2,10){
\pcline[doubleline=true,arrowinset=0.7,arrowlength=0.8, arrowsize=3.5pt 1.5]{-}(0,0)(0,-5)
}

\rput(76,11){
\pcline[doubleline=true,arrowinset=0.7,arrowlength=0.8, arrowsize=3.5pt 1.5]{-}(0,0)(0,-5)
}

\rput(18,-16){
\pcline[doubleline=true,arrowinset=0.7,arrowlength=0.8, arrowsize=3.5pt 1.5]{->}(0,0)(42,0)\naput{\scr $\zeta_f$}
}

%%%%%%%%%%%%%%%%%%%%%%%%%%%%%%%%%%%%%%%%%%%%%%%%%%%%%%%%%%%%%%% left upper

\rput(0,-12){

\rput(0,0){

\rput(0,5){
% \rput(-8,5){$T$}
% \rput(-4,5){$S$}
% \rput(0,5){$T$}
% \rput(4,5){$S$}
% \rput(8,5){$A$}
}

\rput(-2,3){
\psline(-6,6)(-6,-2)
\psline(-2,2)(2,-2)
\psline[border=2pt](2,2)(-2,-2)
\psline(6,6)(6,-2)

\psline(2,2)(2,4)\pscircle(2,4.5){0.5}
\psline(-2,2)(-2,4)\pscircle(-2,4.5){0.5}

\rput(6,-7){
\psline(0,13)(-12,19)
\psline[border=2pt](0,19)(-12,13)
}

}

\rput(-6,-1){
\psline(-2,2)(0,0)(2,2)
\psline(0,0)(0,-2)
}

\rput(2,-1){
\psline(-2,2)(0,0)(2,2)
\psline(0,0)(0,-2)
}

\rput(-1,-8){

\rput(1,0){
\psline(0,0)(-10,5)(10,5)(0,0)
\psline(-6,5)(-6,7)
\psline(2,5)(2,7)
\psline(8,5)(8,23)
\psline(0,0)(0,-4)
\rput(0,3){\scr $b_{ts}$}
}

\rput(9,10){
\psframe*[linecolor=white](-1.5,-2)(1.5,2)
\psframe(-1.5,-2)(1.5,2)
\rput(0,0){\scr $f$}
}

}

}

}

%%%%%%%%%%%%%%%%%%%%%%%%%%%%%%%%%%%%%%%%%%%%%%%%%%%%%%%%%%%%%%% right upper

\rput(78,-12){

\rput(0,0){

\rput(0,5){
% \rput(-8,5){$T$}
% \rput(-4,5){$S$}
% \rput(0,5){$T$}
% \rput(4,5){$S$}
% \rput(8,5){$A$}
}

% \rput(-2,5){
% \psline(-6,2)(-6,-2)
% \psline(-2,2)(2,-2)
% \psline[border=2pt](2,2)(-2,-2)
% \psline(6,2)(6,-2)
% }

\rput(-2,5){
\psline(-6,6)(-6,-2)
\psline(-2,2)(2,-2)
\psline[border=2pt](2,2)(-2,-2)
\psline(6,6)(6,-2)

\psline(2,2)(2,4)\pscircle(2,4.5){0.5}
\psline(-2,2)(-2,4)\pscircle(-2,4.5){0.5}

\rput(6,-7){
\psline(0,13)(-12,19)
\psline[border=2pt](0,19)(-12,13)
}

}

\rput(-6,1){
\psline(-2,2)(0,0)(2,2)
\psline(0,0)(0,-2)
}

\rput(2,1){
\psline(-2,2)(0,0)(2,2)
\psline(0,0)(0,-2)
}

\rput(-1,-8){

\rput(1,2){
\psline(0,0)(-10,5)(10,5)(0,0)
\psline(-6,5)(-6,7)
\psline(2,5)(2,7)
\psline(8,5)(8,23)
\psline(0,0)(0,-8)
\rput(0,3){\scr $a_{ts}$}
}

\rput(1,-2){
\psframe*[linecolor=white](-1.5,-2)(1.5,2)
\psframe(-1.5,-2)(1.5,2)
\rput(0,0){\scr $f$}
}

}

}

}

%%%%%%%%%%%%%%%%%%%%%%%%%%%%%%%%%%%%%%%%%%%%%%%%%%%%%%%% left middle

\rput(0,-42){

\rput(0,1){
% \rput(-4,10){$T$}
% \rput(0,10){$S$}
% \rput(4,10){$A$}
}

\rput(0,-5){
\psline(0,0)(-9,5)(9,5)(0,0)
\psline(-6,5)(-6,11)
\psline(0,5)(0,11)
\psline(6,5)(6,15)
\psline(0,0)(0,-2)

\rput(0,11){
\psline(0,0)(-6,4)
\psline[border=2pt](-6,0)(0,4)
}

\rput(6,9){
\psframe*[linecolor=white](-1.5,-2)(1.5,2)
\psframe(-1.5,-2)(1.5,2)
\rput(0,0){\scr $f$}
}

\rput(0,3){\scr $b_{ts}$}
}

\rput(0,17){
\pcline[doubleline=true,arrowinset=0.7,arrowlength=0.8, arrowsize=3.5pt 1.5]{-}(0,0)(0,-5)
}

\rput(78,16){
\pcline[doubleline=true,arrowinset=0.7,arrowlength=0.8, arrowsize=3.5pt 1.5]{-}(0,0)(0,-5)
}

\rput(18,0){
\pcline[doubleline=true,arrowinset=0.7,arrowlength=0.8, arrowsize=3.5pt 1.5]{->}(0,0)(42,0)\naput{\scr $\zeta_f$}
}

}

%%%%%%%%%%%%%%%%%%%%%%%%%%%%%%%%%%%%%%%%%%%%%%%%%%%%%%% right middle

\rput(78,-44){

\rput(0,1){
% \rput(-4,10){$T$}
% \rput(0,10){$S$}
% \rput(4,10){$A$}
}

\rput(0,1){
\psline(0,0)(-9,5)(9,5)(0,0)
\psline(-6,5)(-6,7)
\psline(0,5)(0,7)
\psline(6,5)(6,11)
\psline(0,0)(0,-8)

\rput(0,7){
\psline(0,0)(-6,4)
\psline[border=2pt](-6,0)(0,4)
}

\rput(0,-4){
\psframe*[linecolor=white](-1.5,-2)(1.5,2)
\psframe(-1.5,-2)(1.5,2)
\rput(0,0){\scr $f$}
}

\rput(0,3){\scr $a_{ts}$}
}

}

%%%%%%%%%%%%%%%%%%%%%%%%%%%%%%%%%%%%%%%%%%%%%%%%%%%% left lower

\rput(0,-68){

\rput(0,0){

\rput(0,5){
% \rput(-8,5){$T$}
% \rput(-4,5){$S$}
% \rput(0,5){$T$}
% \rput(4,5){$S$}
% \rput(8,5){$A$}
}

\rput(-2,3){
\psline(-6,1)(-6,-2)\pscircle(-6,1.5){0.5}
\psline(-2,2)(2,-2)
\psline[border=2pt](2,2)(-2,-2)
\psline(6,1)(6,-2)\pscircle(6,1.5){0.5}

\psline(-2,2)(-2,6)
\psline(2,2)(2,6)

% \rput(6,-7){
% \psline(-4,13)(-8,17)
% \psline[border=2pt](-4,17)(-8,13)
% }

}

\rput(-6,-1){
\psline(-2,2)(0,0)(2,2)
\psline(0,0)(0,-2)
}

\rput(2,-1){
\psline(-2,2)(0,0)(2,2)
\psline(0,0)(0,-2)
}

\rput(-1,-8){

\rput(1,0){
\psline(0,0)(-10,5)(10,5)(0,0)
\psline(-6,5)(-6,7)
\psline(2,5)(2,7)
\psline(8,5)(8,17)
\psline(0,0)(0,-4)
\rput(0,3){\scr $b_{ts}$}
}

\rput(9,10){
\psframe*[linecolor=white](-1.5,-2)(1.5,2)
\psframe(-1.5,-2)(1.5,2)
\rput(0,0){\scr $f$}
}

}

}

\rput(0,17){
\pcline[doubleline=true,arrowinset=0.7,arrowlength=0.8, arrowsize=3.5pt 1.5]{-}(0,0)(0,-5)
}

\rput(78,17){
\pcline[doubleline=true,arrowinset=0.7,arrowlength=0.8, arrowsize=3.5pt 1.5]{-}(0,0)(0,-5)
}

\rput(18,0){
\pcline[doubleline=true,arrowinset=0.7,arrowlength=0.8, arrowsize=3.5pt 1.5]{->}(0,0)(42,0)\naput{\scr $\zeta_f$}
}

}

%%%%%%%%%%%%%%%%%%%%%%%%%%%%%%%%%%%%%%%%%%%%%%%%%%%%%%%%%%%%%%% right lower

\rput(78,-68){

\rput(0,0){

\rput(0,5){
% \rput(-8,5){$T$}
% \rput(-4,5){$S$}
% \rput(0,5){$T$}
% \rput(4,5){$S$}
% \rput(8,5){$A$}
}

% \rput(-2,5){
% \psline(-6,2)(-6,-2)
% \psline(-2,2)(2,-2)
% \psline[border=2pt](2,2)(-2,-2)
% \psline(6,2)(6,-2)
% }

% 
% \rput(-2,5){
% \psline(-6,6)(-6,-2)
% \psline(-2,2)(2,-2)
% \psline[border=2pt](2,2)(-2,-2)
% \psline(6,6)(6,-2)
% 
% \psline(2,2)(2,4)\pscircle(2,4.5){0.5}
% \psline(-2,2)(-2,4)\pscircle(-2,4.5){0.5}
% 
% 
% \rput(6,-7){
% \psline(0,13)(-12,19)
% \psline[border=2pt](0,19)(-12,13)
% }
% 
% }

\rput(-2,5){
\psline(-6,1)(-6,-2)\pscircle(-6,1.5){0.5}
\psline(-2,2)(2,-2)
\psline[border=2pt](2,2)(-2,-2)
\psline(6,1)(6,-2)\pscircle(6,1.5){0.5}

\psline(-2,2)(-2,6)
\psline(2,2)(2,6)

% \rput(6,-7){
% \psline(-4,13)(-8,17)
% \psline[border=2pt](-4,17)(-8,13)
% }

}

\rput(-6,1){
\psline(-2,2)(0,0)(2,2)
\psline(0,0)(0,-2)
}

\rput(2,1){
\psline(-2,2)(0,0)(2,2)
\psline(0,0)(0,-2)
}

\rput(-1,-8){

\rput(1,2){
\psline(0,0)(-10,5)(10,5)(0,0)
\psline(-6,5)(-6,7)
\psline(2,5)(2,7)
\psline(8,5)(8,17)
\psline(0,0)(0,-8)
\rput(0,3){\scr $a_{ts}$}
}

\rput(1,-2){
\psframe*[linecolor=white](-1.5,-2)(1.5,2)
\psframe(-1.5,-2)(1.5,2)
\rput(0,0){\scr $f$}
}

}

}

}

%%%%%%%%%%%%%%%%%%%%%%%%%%%%%%%%%%%%%%%%% bottom left

\rput(0,-90){

% \rput(-8,5){$T$}
% \rput(-4,5){$S$}
% \rput(0,5){$T$}
% \rput(4,5){$S$}
% \rput(8,5){$A$}

\rput(4,-13)
{

\rput(0,0){
\psline(0,0)(-6,5)(6,5)(0,0)
\psline(-4,5)(-4,15)
\psline(0,5)(0,10)\pscircle(0,10.5){0.5}
\psline(4,5)(4,15)
\psline(0,0)(0,-2)
\rput(0,3){\scr $b_{ts}$}
}

\rput(4,10){
\psframe*[linecolor=white](-1.5,-2)(1.5,2)
\psframe(-1.5,-2)(1.5,2)
\rput(0,0){\scr $f$}
}

\rput(-6,-8){
\psline(0,0)(-8,5)(8,5)(0,0)
\psline(-6,5)(-6,10)\pscircle(-6,10.5){0.5}
\psline(-2,5)(-2,23)
\psline(6,5)(6,7)
\psline(0,0)(0,-2)
\rput(0,3){\scr $b_{ts}$}
}

\rput(16,5){\Tmap{\zeta_f}}

}

\rput(0,8){
\pcline[doubleline=true,arrowinset=0.7,arrowlength=0.8, arrowsize=3.5pt 1.5]{-}(0,0)(0,-5)
}

\rput(78,8){
\pcline[doubleline=true,arrowinset=0.7,arrowlength=0.8, arrowsize=3.5pt 1.5]{-}(0,0)(0,-5)
}

}

%%%%%%%%%%%%%%%%%%%%%%%%%%%%%%%%%%%%%%%%%%%%%%%%%%%%%%%%%%%%%%% bottom middle

\rput(38,-90){

% \rput(-8,5){$T$}
% \rput(-4,5){$S$}
% \rput(0,5){$T$}
% \rput(4,5){$S$}
% \rput(8,5){$A$}

\rput(4,-13)
{

\rput(0,7){
\psline(0,0)(-6,5)(6,5)(0,0)
\psline(-4,5)(-4,11)
\psline(0,5)(0,8)\pscircle(0,8.5){0.5}
\psline(4,5)(4,11)
\psline(0,0)(0,-8)
\rput(0,3){\scr $a_{ts}$}
}

\rput(0,2){
\psframe*[linecolor=white](-1.5,-2)(1.5,2)
\psframe(-1.5,-2)(1.5,2)
\rput(0,0){\scr $f$}
}

\rput(-6,-8){
\psline(0,0)(-8,5)(8,5)(0,0)
\psline(-6,5)(-6,10)\pscircle(-6,10.5){0.5}
\psline(-2,5)(-2,26)
\psline(6,5)(6,7)

\psline(0,0)(0,-2)
\rput(0,3){\scr $b_{ts}$}
}

\rput(16,5){\Tmap{\zeta_f}}

}

}

%%%%%%%%%%%%%%%%%%%%%%%%%%%%%%%%%%%%%%%%%%%%%%%%%%%%%%%%%%%%%%% bottom right

\rput(78,-91){

% \rput(-8,5){$T$}
% \rput(-4,5){$S$}
% \rput(0,5){$T$}
% \rput(4,5){$S$}
% \rput(8,5){$A$}

\rput(4,-13)
{

\rput(0,7){
\psline(0,0)(-6,5)(6,5)(0,0)
\psline(-4,5)(-4,10)
\psline(0,5)(0,8)\pscircle(0,8.5){0.5}
\psline(4,5)(4,10)
\psline(0,0)(0,-2)
\rput(0,3){\scr $a_{ts}$}
}

\rput(-6,0){
\psline(0,0)(-8,5)(8,5)(0,0)
\psline(-6,5)(-6,10)\pscircle(-6,10.5){0.5}
\psline(-2,5)(-2,17)
\psline(6,5)(6,7)
\psline(0,0)(0,-10)
\rput(0,3){\scr $a_{ts}$}
}

\rput(-6,-5){
\psframe*[linecolor=white](-1.5,-2)(1.5,2)
\psframe(-1.5,-2)(1.5,2)
\rput(0,0){\scr $f$}
}

}

}

\end{small}
\endpspicture\]

\vs{4}

Conversely suppose we have a pair $(\sigma_f, \tau_f)$ as below satisfying the interaction diagram.

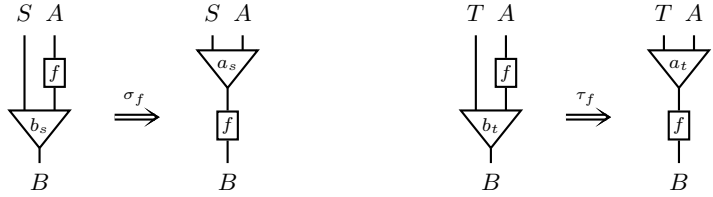
\[\psset{unit=0.1cm,labelsep=2pt,nodesep=3pt,linewidth=0.8pt}
\pspicture(0,50)(90,78)
\begin{small}

%%%%%%%%%%%%%%%%%%%%%%%%%%%%%%%%%%%%%%%%%%%% sigma

\rput(0,60){

\rput(0,4){
\rput(-2,14){$S$}
\rput(2,14){$A$}
}

\psline(0,0)(-4,5)(4,5)(0,0)
\psline(-2,5)(-2,15)
\psline(2,5)(2,15)

\rput(2,10){
\psframe*[linecolor=white](-1.5,-2)(1.5,2)
\psframe(-1.5,-2)(1.5,2)
\rput(0,0){\scr $f$}
}

\psline(0,0)(0,-2)
\rput(0,3){\scr $b_s$}
\rput(0,-4.5){$B$}

\rput(13,5){\Tmap{\sigma_f}}

\rput(25,0){

\rput(0,8){
\rput(-2,10){$S$}
\rput(2,10){$A$}
% \rput(0,-5){$A$}

\psline(0,0)(-4,5)(4,5)(0,0)
\psline(-2,5)(-2,7)
\psline(2,5)(2,7)
\psline(0,0)(0,-10)
\rput(0,3){\scr $a_s$}
}

% \rput(0,-4.5){$B$}

\rput(0,3){
\psframe*[linecolor=white](-1.5,-2)(1.5,2)
\psframe(-1.5,-2)(1.5,2)
\rput(0,0){\scr $f$}
}

\rput(0,-4.5){$B$}

}

}

%%%%%%%%%%%%%%%%%%%%%%%%%%%%%%%%%%%%%%%%%%%%%%%%%%%% tau

\rput(60,60){

\rput(0,0){

\rput(0,4){
\rput(-2,14){$T$}
\rput(2,14){$A$}
}

\psline(0,0)(-4,5)(4,5)(0,0)
\psline(-2,5)(-2,15)
\psline(2,5)(2,15)

\rput(2,10){
\psframe*[linecolor=white](-1.5,-2)(1.5,2)
\psframe(-1.5,-2)(1.5,2)
\rput(0,0){\scr $f$}
}

\psline(0,0)(0,-2)
\rput(0,3){\scr $b_t$}
\rput(0,-4.5){$B$}

\rput(13,5){\Tmap{\tau_f}}
}

\rput(25,0){

\rput(0,8){
\rput(-2,10){$T$}
\rput(2,10){$A$}
% \rput(0,-5){$A$}

\psline(0,0)(-4,5)(4,5)(0,0)
\psline(-2,5)(-2,7)
\psline(2,5)(2,7)
\psline(0,0)(0,-10)
\rput(0,3){\scr $a_t$}
}

% \rput(0,-4.5){$B$}

\rput(0,3){
\psframe*[linecolor=white](-1.5,-2)(1.5,2)
\psframe(-1.5,-2)(1.5,2)
\rput(0,0){\scr $f$}
}

\rput(0,-4.5){$B$}

}

}

%%%%%%%%%%%%%%%%%%%%%%%%%%%%%%%%%%%%%%%%%%%%%%%%%%%%%%%%%%%%%%%%%% end tau

\end{small}
\endpspicture\]

\noi We show that these result in a weak map of $TS$-algebras. First we form the 2-cell $\zeta_f$ as shown below.

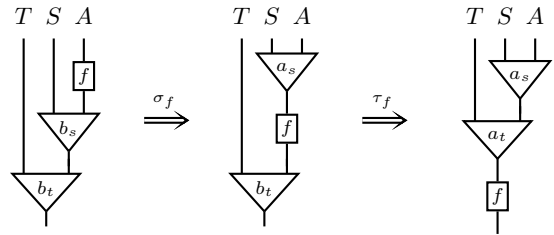
\[\psset{unit=0.1cm,labelsep=2pt,nodesep=3pt,linewidth=0.8pt}
\pspicture(0,-10)(70,27)
\begin{small}

%%%%%%%%%%%%%%%%%%%%%%%%%%%%%%%%%%%%%%%%%%%%%%%%%%%%%%%%%%%%%%% bottom left

% \rput(2,5){
% \pcline[doubleline=true,arrowinset=0.7,arrowlength=0.8, arrowsize=3.5pt 1.5]{-}(0,0)(0,-10)
% }

\rput(0,20){

\rput(0,0){

\rput(-2,5){$T$}
\rput(2,5){$S$}
\rput(6,5){$A$}

% \psline(-2,2)(0,0)(2,2)
% \psline(0,0)(0,-2)

\rput(4,-13)
{

\psline(0,0)(-4,5)(4,5)(0,0)
\psline(-2,5)(-2,15)
\psline(2,5)(2,15)

\rput(2,10){
\psframe*[linecolor=white](-1.5,-2)(1.5,2)
\psframe(-1.5,-2)(1.5,2)
\rput(0,0){\scr $f$}
}

\psline(0,0)(0,-2)
\rput(0,3){\scr $b_s$}
% \rput(0,-4.5){$B$}

\rput(-3,-8){
\psline(0,0)(-4.5,5)(4.5,5)(0,0)
\psline(-3,5)(-3,23)
\psline(3,5)(3,7)

\psline(0,0)(0,-2)
\rput(0,3){\scr $b_t$}
}

\rput(13,5){\Tmap{\sigma_f}}

}

%%%%%%%%%%%%%%%%%%%%%%%%%%%%%%%%%%%%%%%%%%%%%%%%%%%%%%%%%%%%%%% bottom middle

\rput(25,20){

\rput(4,-33){

\rput(4,8){
\rput(-4,0){
\rput(-2,10){$T$}
\rput(2,10){$S$}
\rput(6,10){$A$}
}

\psline(0,0)(-4,5)(4,5)(0,0)
\psline(-2,5)(-2,7)
\psline(2,5)(2,7)
\psline(0,0)(0,-10)
\rput(0,3){\scr $a_s$}
}

% \rput(0,-4.5){$B$}

\rput(4,3){
\psframe*[linecolor=white](-1.5,-2)(1.5,2)
\psframe(-1.5,-2)(1.5,2)
\rput(0,0){\scr $f$}
}

% \rput(0,-4.5){$B$}

\rput(1,-8){
\psline(0,0)(-4.5,5)(4.5,5)(0,0)
\psline(-3,5)(-3,23)
\psline(3,5)(3,7)

\psline(0,0)(0,-2)
\rput(0,3){\scr $b_t$}
}

}

}
}

\rput(46,-8){\Tmap{\tau_f}}

}

%%%%%%%%%%%%%%%%%%%%%%%%%%%%%%%%%%%%%%%%%%%%%%%%%%%%%%%%%%%%%%%  bottom right

% \rput(72,5){
% \pcline[doubleline=true,arrowinset=0.7,arrowlength=0.8, arrowsize=3.5pt 1.5]{-}(0,0)(0,-10)
% }

\rput(60,20){

\rput(0,7){

\rput(0,-7){
\rput(-2,5){$T$}
\rput(2,5){$S$}
\rput(6,5){$A$}
}

% \psline(-2,2)(0,0)(2,2)
% \psline(0,0)(0,-2)

\rput(4,-13)
{

\psline(0,0)(-4,5)(4,5)(0,0)
\psline(-2,5)(-2,8)
\psline(2,5)(2,8)

\psline(0,0)(0,-2)
\rput(0,3){\scr $a_s$}
% \rput(0,-4.5){$B$}

\rput(-3,-8){
\psline(0,0)(-4.5,5)(4.5,5)(0,0)
\psline(-3,5)(-3,16)
\psline(3,5)(3,7)

\psline(0,0)(0,-10)
\rput(0,3){\scr $a_t$}

\rput(0,-5){
\psframe*[linecolor=white](-1.5,-2)(1.5,2)
\psframe(-1.5,-2)(1.5,2)
\rput(0,0){\scr $f$}
}

}

}
}
}

\end{small}
\endpspicture\]

We check the axioms.  

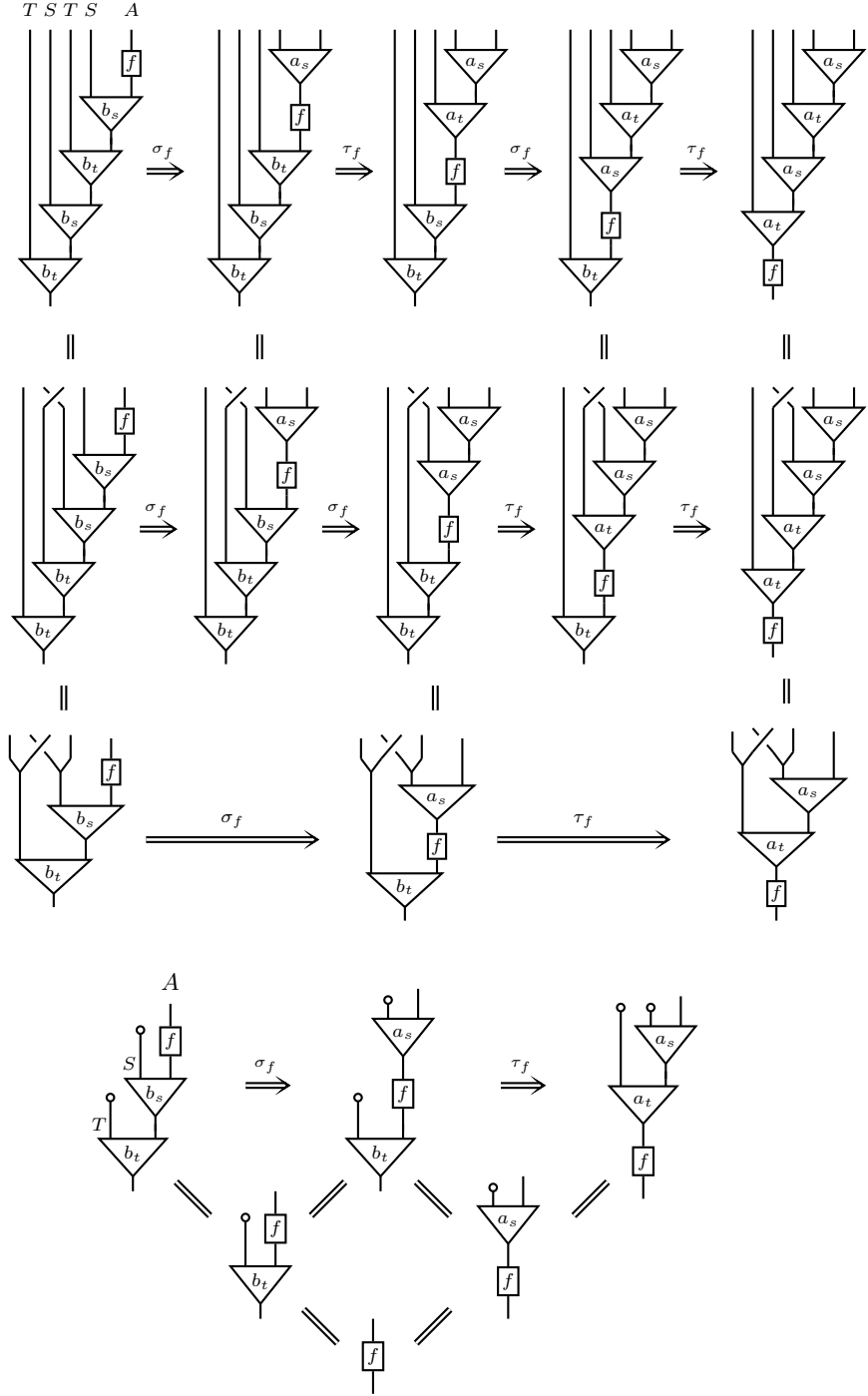
\[\psset{unit=0.09cm,labelsep=2pt,nodesep=3pt,linewidth=0.8pt}
\pspicture(-10,-60)(100,45)
\begin{small}

%%%%%%%%%%%%%%%%%%%%%%%%%%%%%%%%%%%%%%%%%%%%%%%%%%%%%%%%%%%% top left

\rput(0,70){

\rput(-8,5){\scr $T$}
\rput(-5,5){\scr $S$}
\rput(-2,5){\scr $T$}
\rput(1,5){\scr $S$}
\rput(7,5){\scr $A$}

% \psline(-2,2)(0,0)(2,2)
% \psline(0,0)(0,-2)

\rput(4,-13)
{

\psline(0,0)(-4.5,5)(4.5,5)(0,0)
\psline(-3,5)(-3,15)
\psline(3,5)(3,15)

\rput(3,10){
\psframe*[linecolor=white](-1.5,-2)(1.5,2)
\psframe(-1.5,-2)(1.5,2)
\rput(0,0){\scr $f$}
}

\psline(0,0)(0,-2)
\rput(0,3){\scr $b_s$}
% \rput(0,-4.5){$B$}

\rput(-3,-8){
\psline(0,0)(-4.5,5)(4.5,5)(0,0)
\psline(-3,5)(-3,23)
\psline(3,5)(3,7)

\psline(0,0)(0,-2)
\rput(0,3){\scr $b_t$}
}

\rput(-6,-16){
\psline(0,0)(-4.5,5)(4.5,5)(0,0)
\psline(-3,5)(-3,31)
\psline(3,5)(3,7)

\psline(0,0)(0,-2)
\rput(0,3){\scr $b_s$}
}

\rput(-9,-24){
\psline(0,0)(-4.5,5)(4.5,5)(0,0)
\psline(-3,5)(-3,39)
\psline(3,5)(3,7)

\psline(0,0)(0,-2)
\rput(0,3){\scr $b_t$}
}

\rput(8,-5){\msTmap{\sigma_f}}

\rput(-6,-29){
\pcline[doubleline=true,arrowinset=0.7,arrowlength=0.8, arrowsize=3.5pt 1.5]{-}(0,0)(0,-6)
}

}}

%%%%%%%%%%%%%%%%%%%%%%%%%%%%%%%%%%%%%%%%%%%%%%%%%%%%%%%%%%%% top left-middle

\rput(28,70){

% \rput(-8,5){\scr $T$}
% \rput(-5,5){\scr $S$}
% \rput(-2,5){\scr $T$}
% \rput(1,5){\scr $S$}
% \rput(7,5){\scr $A$}

% \psline(-2,2)(0,0)(2,2)
% \psline(0,0)(0,-2)

\rput(4,-13)
{

\rput(0,7){
\psline(0,0)(-4.5,5)(4.5,5)(0,0)
\psline(-3,5)(-3,8)
\psline(3,5)(3,8)
\psline(0,0)(0,-8)
\rput(0,3){\scr $a_s$}
}

\rput(0,2){
\psframe*[linecolor=white](-1.5,-2)(1.5,2)
\psframe(-1.5,-2)(1.5,2)
\rput(0,0){\scr $f$}
}

\rput(-3,-8){
\psline(0,0)(-4.5,5)(4.5,5)(0,0)
\psline(-3,5)(-3,23)
\psline(3,5)(3,7)
\psline(0,0)(0,-2)
\rput(0,3){\scr $b_t$}
}

\rput(-6,-16){
\psline(0,0)(-4.5,5)(4.5,5)(0,0)
\psline(-3,5)(-3,31)
\psline(3,5)(3,7)

\psline(0,0)(0,-2)
\rput(0,3){\scr $b_s$}
}

\rput(-9,-24){
\psline(0,0)(-4.5,5)(4.5,5)(0,0)
\psline(-3,5)(-3,39)
\psline(3,5)(3,7)

\psline(0,0)(0,-2)
\rput(0,3){\scr $b_t$}
}

\rput(8,-5){\msTmap{\tau_f}}

\rput(-6,-29){
\pcline[doubleline=true,arrowinset=0.7,arrowlength=0.8, arrowsize=3.5pt 1.5]{-}(0,0)(0,-6)
}

}}

%%%%%%%%%%%%%%%%%%%%%%%%%%%%%%%%%%%%%%%%%%%%%%%%%%%%%%%%%%%% top middle

\rput(54,70){

% \rput(-8,5){\scr $T$}
% \rput(-5,5){\scr $S$}
% \rput(-2,5){\scr $T$}
% \rput(1,5){\scr $S$}
% \rput(7,5){\scr $A$}

% \psline(-2,2)(0,0)(2,2)
% \psline(0,0)(0,-2)

\rput(4,-13)
{

\rput(0,7){
\psline(0,0)(-4.5,5)(4.5,5)(0,0)
\psline(-3,5)(-3,8)
\psline(3,5)(3,8)
\psline(0,0)(0,-2)
\rput(0,3){\scr $a_s$}
}

\rput(-3,-1){
\psline(0,0)(-4.5,5)(4.5,5)(0,0)
\psline(-3,5)(-3,16)
\psline(3,5)(3,7)
\psline(0,0)(0,-8)
\rput(0,3){\scr $a_t$}
}

\rput(-3,-6){
\psframe*[linecolor=white](-1.5,-2)(1.5,2)
\psframe(-1.5,-2)(1.5,2)
\rput(0,0){\scr $f$}
}

\rput(-6,-16){
\psline(0,0)(-4.5,5)(4.5,5)(0,0)
\psline(-3,5)(-3,31)
\psline(3,5)(3,7)

\psline(0,0)(0,-2)
\rput(0,3){\scr $b_s$}
}

\rput(-9,-24){
\psline(0,0)(-4.5,5)(4.5,5)(0,0)
\psline(-3,5)(-3,39)
\psline(3,5)(3,7)

\psline(0,0)(0,-2)
\rput(0,3){\scr $b_t$}
}

\rput(7,-5){\msTmap{\sigma_f}}

}}

%%%%%%%%%%%%%%%%%%%%%%%%%%%%%%%%%%%%%%%%%%%%%%%%%%%%%%%%%%%% top right-middle

\rput(80,70){

% \rput(-8,5){\scr $T$}
% \rput(-5,5){\scr $S$}
% \rput(-2,5){\scr $T$}
% \rput(1,5){\scr $S$}
% \rput(7,5){\scr $A$}

% \psline(-2,2)(0,0)(2,2)
% \psline(0,0)(0,-2)

\rput(4,-13)
{

\rput(0,7){
\psline(0,0)(-4.5,5)(4.5,5)(0,0)
\psline(-3,5)(-3,8)
\psline(3,5)(3,8)
\psline(0,0)(0,-2)
\rput(0,3){\scr $a_s$}
}

\rput(-3,-1){
\psline(0,0)(-4.5,5)(4.5,5)(0,0)
\psline(-3,5)(-3,16)
\psline(3,5)(3,7)
\psline(0,0)(0,-2)
\rput(0,3){\scr $a_t$}
}

\rput(-6,-9){
\psline(0,0)(-4.5,5)(4.5,5)(0,0)
\psline(-3,5)(-3,24)
\psline(3,5)(3,7)

\psline(0,0)(0,-8)
\rput(0,3){\scr $a_s$}
}

\rput(-6,-14){
\psframe*[linecolor=white](-1.5,-2)(1.5,2)
\psframe(-1.5,-2)(1.5,2)
\rput(0,0){\scr $f$}
}

\rput(-9,-24){
\psline(0,0)(-4.5,5)(4.5,5)(0,0)
\psline(-3,5)(-3,39)
\psline(3,5)(3,7)

\psline(0,0)(0,-2)
\rput(0,3){\scr $b_t$}
}

\rput(7,-5){\msTmap{\tau_f}}

\rput(-7,-29){
\pcline[doubleline=true,arrowinset=0.7,arrowlength=0.8, arrowsize=3.5pt 1.5]{-}(0,0)(0,-6)
}

}}

%%%%%%%%%%%%%%%%%%%%%%%%%%%%%%%%%%%%%%%%%%%%%%%%%%%%%%%%%%%% top right

\rput(107,70){

% \rput(-8,5){\scr $T$}
% \rput(-5,5){\scr $S$}
% \rput(-2,5){\scr $T$}
% \rput(1,5){\scr $S$}
% \rput(7,5){\scr $A$}

% \psline(-2,2)(0,0)(2,2)
% \psline(0,0)(0,-2)

\rput(4,-13)
{

\rput(0,7){
\psline(0,0)(-4.5,5)(4.5,5)(0,0)
\psline(-3,5)(-3,8)
\psline(3,5)(3,8)
\psline(0,0)(0,-2)
\rput(0,3){\scr $a_s$}
}

\rput(-3,-1){
\psline(0,0)(-4.5,5)(4.5,5)(0,0)
\psline(-3,5)(-3,16)
\psline(3,5)(3,7)
\psline(0,0)(0,-2)
\rput(0,3){\scr $a_t$}
}

\rput(-6,-9){
\psline(0,0)(-4.5,5)(4.5,5)(0,0)
\psline(-3,5)(-3,24)
\psline(3,5)(3,7)

\psline(0,0)(0,-2)
\rput(0,3){\scr $a_s$}
}

\rput(-9,-17){
\psline(0,0)(-4.5,5)(4.5,5)(0,0)
\psline(-3,5)(-3,32)
\psline(3,5)(3,7)

\psline(0,0)(0,-8)
\rput(0,3){\scr $a_t$}
}

\rput(-9,-21){
\psframe*[linecolor=white](-1.5,-2)(1.5,2)
\psframe(-1.5,-2)(1.5,2)
\rput(0,0){\scr $f$}
}

% \rput(11,-5){\Tmap{\sigma_f}}

\rput(-7,-29){
\pcline[doubleline=true,arrowinset=0.7,arrowlength=0.8, arrowsize=3.5pt 1.5]{-}(0,0)(0,-6)
}

}}

%%%%%%%%%%%%%%%%%%%%%%%%%%%%%%%%%%%%%%%%%%%%%%%%%%%%%%%%%%%% middle left

\rput(-1,17){

% \rput(-8,5){\scr $T$}
% \rput(-5,5){\scr $S$}
% \rput(-2,5){\scr $T$}
% \rput(1,5){\scr $S$}
% \rput(7,5){\scr $A$}

% \psline(-2,2)(0,0)(2,2)
% \psline(0,0)(0,-2)

\rput(0,-3){
\psline(-5,5)(-2,2)
\psline[border=2pt](-2,5)(-5,2)
}

\rput(4,-13)
{

\psline(0,0)(-4.5,5)(4.5,5)(0,0)
\psline(-3,5)(-3,15)
\psline(3,5)(3,15)

\rput(3,10){
\psframe*[linecolor=white](-1.5,-2)(1.5,2)
\psframe(-1.5,-2)(1.5,2)
\rput(0,0){\scr $f$}
}

\psline(0,0)(0,-2)
\rput(0,3){\scr $b_s$}
% \rput(0,-4.5){$B$}

\rput(-3,-8){
\psline(0,0)(-4.5,5)(4.5,5)(0,0)
\psline(-3,5)(-3,20)
\psline(3,5)(3,7)

\psline(0,0)(0,-2)
\rput(0,3){\scr $b_s$}
}

\rput(-6,-16){
\psline(0,0)(-4.5,5)(4.5,5)(0,0)
\psline(-3,5)(-3,28)
\psline(3,5)(3,7)

\psline(0,0)(0,-2)
\rput(0,3){\scr $b_t$}
}

\rput(-9,-24){
\psline(0,0)(-4.5,5)(4.5,5)(0,0)
\psline(-3,5)(-3,39)
\psline(3,5)(3,7)

\psline(0,0)(0,-2)
\rput(0,3){\scr $b_t$}
}

\rput(8,-5){\msTmap{\sigma_f}}

\rput(-6,-28){
\pcline[doubleline=true,arrowinset=0.7,arrowlength=0.8, arrowsize=3.5pt 1.5]{-}(0,0)(0,-6)
}

}}

%%%%%%%%%%%%%%%%%%%%%%%%%%%%%%%%%%%%%%%%%%%%%%%%%%%%%%%%%%%% middle left-middle

\rput(26,17){

% \rput(-8,5){\scr $T$}
% \rput(-5,5){\scr $S$}
% \rput(-2,5){\scr $T$}
% \rput(1,5){\scr $S$}
% \rput(7,5){\scr $A$}

\rput(0,-3){
\psline(-5,5)(-2,2)
\psline[border=2pt](-2,5)(-5,2)
}

% \psline(-2,2)(0,0)(2,2)
% \psline(0,0)(0,-2)

\rput(4,-13)
{

\rput(0,7){
\psline(0,0)(-4.5,5)(4.5,5)(0,0)
\psline(-3,5)(-3,8)
\psline(3,5)(3,8)
\psline(0,0)(0,-8)
\rput(0,3){\scr $a_s$}
}

\rput(0,2){
\psframe*[linecolor=white](-1.5,-2)(1.5,2)
\psframe(-1.5,-2)(1.5,2)
\rput(0,0){\scr $f$}
}

\rput(-3,-8){
\psline(0,0)(-4.5,5)(4.5,5)(0,0)
\psline(-3,5)(-3,20)
\psline(3,5)(3,7)
\psline(0,0)(0,-2)
\rput(0,3){\scr $b_s$}
}

\rput(-6,-16){
\psline(0,0)(-4.5,5)(4.5,5)(0,0)
\psline(-3,5)(-3,28)
\psline(3,5)(3,7)

\psline(0,0)(0,-2)
\rput(0,3){\scr $b_t$}
}

\rput(-9,-24){
\psline(0,0)(-4.5,5)(4.5,5)(0,0)
\psline(-3,5)(-3,39)
\psline(3,5)(3,7)

\psline(0,0)(0,-2)
\rput(0,3){\scr $b_t$}
}

\rput(8,-5){\msTmap{\sigma_f}}

}}

%%%%%%%%%%%%%%%%%%%%%%%%%%%%%%%%%%%%%%%%%%%%%%%%%%%%%%%%%%%% middle middle

\rput(53,17){

% \rput(-8,5){\scr $T$}
% \rput(-5,5){\scr $S$}
% \rput(-2,5){\scr $T$}
% \rput(1,5){\scr $S$}
% \rput(7,5){\scr $A$}

\rput(0,-3){
\psline(-5,5)(-2,2)
\psline[border=2pt](-2,5)(-5,2)
}

% \psline(-2,2)(0,0)(2,2)
% \psline(0,0)(0,-2)

\rput(4,-13)
{

\rput(0,7){
\psline(0,0)(-4.5,5)(4.5,5)(0,0)
\psline(-3,5)(-3,8)
\psline(3,5)(3,8)
\psline(0,0)(0,-2)
\rput(0,3){\scr $a_s$}
}

\rput(-3,-1){
\psline(0,0)(-4.5,5)(4.5,5)(0,0)
\psline(-3,5)(-3,13)
\psline(3,5)(3,7)
\psline(0,0)(0,-8)
\rput(0,3){\scr $a_s$}
}

\rput(-3,-6){
\psframe*[linecolor=white](-1.5,-2)(1.5,2)
\psframe(-1.5,-2)(1.5,2)
\rput(0,0){\scr $f$}
}

\rput(-6,-16){
\psline(0,0)(-4.5,5)(4.5,5)(0,0)
\psline(-3,5)(-3,28)
\psline(3,5)(3,7)

\psline(0,0)(0,-2)
\rput(0,3){\scr $b_t$}
}

\rput(-9,-24){
\psline(0,0)(-4.5,5)(4.5,5)(0,0)
\psline(-3,5)(-3,39)
\psline(3,5)(3,7)

\psline(0,0)(0,-2)
\rput(0,3){\scr $b_t$}
}

\rput(7,-5){\msTmap{\tau_f}}

\rput(-5,-28){
\pcline[doubleline=true,arrowinset=0.7,arrowlength=0.8, arrowsize=3.5pt 1.5]{-}(0,0)(0,-6)
}

}}

%%%%%%%%%%%%%%%%%%%%%%%%%%%%%%%%%%%%%%%%%%%%%%%%%%%%%%%%%%%% middle right-middle

\rput(79,17){

% \rput(-8,5){\scr $T$}
% \rput(-5,5){\scr $S$}
% \rput(-2,5){\scr $T$}
% \rput(1,5){\scr $S$}
% \rput(7,5){\scr $A$}

\rput(0,-3){
\psline(-5,5)(-2,2)
\psline[border=2pt](-2,5)(-5,2)
}

% \psline(-2,2)(0,0)(2,2)
% \psline(0,0)(0,-2)

\rput(4,-13)
{

\rput(0,7){
\psline(0,0)(-4.5,5)(4.5,5)(0,0)
\psline(-3,5)(-3,8)
\psline(3,5)(3,8)
\psline(0,0)(0,-2)
\rput(0,3){\scr $a_s$}
}

\rput(-3,-1){
\psline(0,0)(-4.5,5)(4.5,5)(0,0)
\psline(-3,5)(-3,13)
\psline(3,5)(3,7)
\psline(0,0)(0,-2)
\rput(0,3){\scr $a_s$}
}

\rput(-6,-9){
\psline(0,0)(-4.5,5)(4.5,5)(0,0)
\psline(-3,5)(-3,21)
\psline(3,5)(3,7)

\psline(0,0)(0,-8)
\rput(0,3){\scr $a_t$}
}

\rput(-6,-14){
\psframe*[linecolor=white](-1.5,-2)(1.5,2)
\psframe(-1.5,-2)(1.5,2)
\rput(0,0){\scr $f$}
}

\rput(-9,-24){
\psline(0,0)(-4.5,5)(4.5,5)(0,0)
\psline(-3,5)(-3,39)
\psline(3,5)(3,7)

\psline(0,0)(0,-2)
\rput(0,3){\scr $b_t$}
}

\rput(7,-5){\msTmap{\tau_f}}

}}

%%%%%%%%%%%%%%%%%%%%%%%%%%%%%%%%%%%%%%%%%%%%%%%%%%%%%%%%%%%% middle right

\rput(107,17){

% \rput(-8,5){\scr $T$}
% \rput(-5,5){\scr $S$}
% \rput(-2,5){\scr $T$}
% \rput(1,5){\scr $S$}
% \rput(7,5){\scr $A$}

\rput(0,-3){
\psline(-5,5)(-2,2)
\psline[border=2pt](-2,5)(-5,2)
}

% \psline(-2,2)(0,0)(2,2)
% \psline(0,0)(0,-2)

\rput(4,-13)
{

\rput(0,7){
\psline(0,0)(-4.5,5)(4.5,5)(0,0)
\psline(-3,5)(-3,8)
\psline(3,5)(3,8)
\psline(0,0)(0,-2)
\rput(0,3){\scr $a_s$}
}

\rput(-3,-1){
\psline(0,0)(-4.5,5)(4.5,5)(0,0)
\psline(-3,5)(-3,13)
\psline(3,5)(3,7)
\psline(0,0)(0,-2)
\rput(0,3){\scr $a_s$}
}

\rput(-6,-9){
\psline(0,0)(-4.5,5)(4.5,5)(0,0)
\psline(-3,5)(-3,21)
\psline(3,5)(3,7)

\psline(0,0)(0,-2)
\rput(0,3){\scr $a_t$}
}

\rput(-9,-17){
\psline(0,0)(-4.5,5)(4.5,5)(0,0)
\psline(-3,5)(-3,32)
\psline(3,5)(3,7)

\psline(0,0)(0,-8)
\rput(0,3){\scr $a_t$}
}

\rput(-9,-21){
\psframe*[linecolor=white](-1.5,-2)(1.5,2)
\psframe(-1.5,-2)(1.5,2)
\rput(0,0){\scr $f$}
}

% \rput(11,-5){\Tmap{\sigma_f}}

\rput(-7,-27){
\pcline[doubleline=true,arrowinset=0.7,arrowlength=0.8, arrowsize=3.5pt 1.5]{-}(0,0)(0,-6)
}

}}

%%%%%%%%%%%%%%%%%%%%%%%%%%%%%%%%%%%%%%%%%%%%%%%%%%%%%%%%%%%% bottom left

\rput(-3,-35){

% \rput(-8,5){\scr $T$}
% \rput(-5,5){\scr $S$}
% \rput(-2,5){\scr $T$}
% \rput(1,5){\scr $S$}
% \rput(7,5){\scr $A$}

% \psline(-2,2)(0,0)(2,2)
% \psline(0,0)(0,-2)

\rput(0,-3){
\psline(-5,5.5)(-2,2)
\psline[border=2pt](-2,5.5)(-5,2)

\rput(3,0){
\psline(-5,2)(-3.5,0)(-2,2)
\psline(-2,2)(-2,5.5)
}

\rput(-3,0){
\psline(-5,2)(-3.5,0)(-2,2)
\psline(-5,2)(-5,5.5)
}

}

\rput(4,-13)
{

\rput(-0.75,0){
\psline(0,0)(-5.5,5)(5.5,5)(0,0)
\rput(0,3){\scr $b_s$}
\psline(0,0)(0,-3)
}
\psline(-4.5,5)(-4.5,10)
\psline(3,5)(3,15)

\rput(3,10){
\psframe*[linecolor=white](-1.5,-2)(1.5,2)
\psframe(-1.5,-2)(1.5,2)
\rput(0,0){\scr $f$}
}

% \rput(0,-4.5){$B$}

\rput(-3,-8){

\rput(-2.5,0){
\psline(0,0)(-5.5,5)(5.5,5)(0,0)
\psline(0,0)(0,-2)
\rput(0,3){\scr $b_t$}
}

\psline(-7.5,5)(-7.5,18)
% \psline(3,5)(3,7)

}

% \rput(-6,-16){
% \psline(0,0)(-4.5,5)(4.5,5)(0,0)
% \psline(-3,5)(-3,28)
% \psline(3,5)(3,7)
% 
% \psline(0,0)(0,-2)
% \rput(0,3){\scr $b_t$}
% }
% 
% 
% \rput(-9,-24){
% \psline(0,0)(-4.5,5)(4.5,5)(0,0)
% \psline(-3,5)(-3,39)
% \psline(3,5)(3,7)
% 
% \psline(0,0)(0,-2)
% \rput(0,3){\scr $b_t$}
% }

\rput(7,0){
\pcline[doubleline=true,arrowinset=0.7,arrowlength=0.8, arrowsize=3.5pt 1.5]{->}(0,0)(28,0)\naput{\scr $\sigma_f$}
}

}}

%%%%%%%%%%%%%%%%%%%%%%%%%%%%%%%%%%%%%%%%%%%%%%%%%%%%%%%%%%%% bottom middle

\rput(49,-37){

% \rput(-8,5){\scr $T$}
% \rput(-5,5){\scr $S$}
% \rput(-2,5){\scr $T$}
% \rput(1,5){\scr $S$}
% \rput(7,5){\scr $A$}

% \psline(-2,2)(0,0)(2,2)
% \psline(0,0)(0,-2)

\rput(0,-1){
\psline(-5,5.5)(-2,2)
\psline[border=2pt](-2,5.5)(-5,2)

\rput(3,0){
\psline(-5,2)(-3.5,0)(-2,2)
\psline(-2,2)(-2,5.5)
}

\rput(-3,0){
\psline(-5,2)(-3.5,0)(-2,2)
\psline(-5,2)(-5,5)
}

}

\rput(4,-13)
{

\rput(0,5){

\rput(-0.75,0){
\psline(0,0)(-5.5,5)(5.5,5)(0,0)
\rput(0,3){\scr $a_s$}
\psline(0,0)(0,-8)
}
\psline(-4.5,5)(-4.5,7)
\psline(3,5)(3,12)
}

% \rput(0,-4.5){$B$}

\rput(-3,-8){

\rput(-2.5,0){
\psline(0,0)(-5.5,5)(5.5,5)(0,0)
\psline(0,0)(0,-2)
\rput(0,3){\scr $b_t$}
}

\psline(-7.5,5)(-7.5,20)
% \psline(3,5)(3,7)

}

\rput(-0.75,1){
\psframe*[linecolor=white](-1.5,-2)(1.5,2)
\psframe(-1.5,-2)(1.5,2)
\rput(0,0){\scr $f$}
}

% \rput(-6,-16){
% \psline(0,0)(-4.5,5)(4.5,5)(0,0)
% \psline(-3,5)(-3,28)
% \psline(3,5)(3,7)
% 
% \psline(0,0)(0,-2)
% \rput(0,3){\scr $b_t$}
% }
% 
% 
% \rput(-9,-24){
% \psline(0,0)(-4.5,5)(4.5,5)(0,0)
% \psline(-3,5)(-3,39)
% \psline(3,5)(3,7)
% 
% \psline(0,0)(0,-2)
% \rput(0,3){\scr $b_t$}
% }

% \rput(11,0){\Tmap{\sigma_f}}

\rput(7,2){
\pcline[doubleline=true,arrowinset=0.7,arrowlength=0.8, arrowsize=3.5pt 1.5]{->}(0,0)(28,0)\naput{\scr $\tau_f$}
}

}}

%%%%%%%%%%%%%%%%%%%%%%%%%%%%%%%%%%%%%%%%%%%%%%%%%%%%%%%%%%%% bottom right

\rput(104,-36){

% \rput(-8,5){\scr $T$}
% \rput(-5,5){\scr $S$}
% \rput(-2,5){\scr $T$}
% \rput(1,5){\scr $S$}
% \rput(7,5){\scr $A$}

% \psline(-2,2)(0,0)(2,2)
% \psline(0,0)(0,-2)

\rput(0,-1){
\psline(-5,5.5)(-2,2)
\psline[border=2pt](-2,5.5)(-5,2)

\rput(3,0){
\psline(-5,2)(-3.5,0)(-2,2)
\psline(-2,2)(-2,5.5)
}

\rput(-3,0){
\psline(-5,2)(-3.5,0)(-2,2)
\psline(-5,2)(-5,5)
}

}

\rput(4,-13)
{

\rput(0,5){

\rput(-0.75,0){
\psline(0,0)(-5.5,5)(5.5,5)(0,0)
\rput(0,3){\scr $a_s$}
\psline(0,0)(0,-3)
}
\psline(-4.5,5)(-4.5,7)
\psline(3,5)(3,12)
}

% \rput(0,-4.5){$B$}

\rput(-3,-3){

\rput(-2.5,0){
\psline(0,0)(-5.5,5)(5.5,5)(0,0)
\psline(0,0)(0,-8)
\rput(0,3){\scr $a_t$}
}

\psline(-7.5,5)(-7.5,15)
% \psline(3,5)(3,7)

}

\rput(-5.5,-7){
\psframe*[linecolor=white](-1.5,-2)(1.5,2)
\psframe(-1.5,-2)(1.5,2)
\rput(0,0){\scr $f$}
}

% \rput(-6,-16){
% \psline(0,0)(-4.5,5)(4.5,5)(0,0)
% \psline(-3,5)(-3,28)
% \psline(3,5)(3,7)
% 
% \psline(0,0)(0,-2)
% \rput(0,3){\scr $b_t$}
% }
% 
% 
% \rput(-9,-24){
% \psline(0,0)(-4.5,5)(4.5,5)(0,0)
% \psline(-3,5)(-3,39)
% \psline(3,5)(3,7)
% 
% \psline(0,0)(0,-2)
% \rput(0,3){\scr $b_t$}
% }

% \rput(11,0){\Tmap{\sigma_f}}

% \rput(7,0){
% \pcline[doubleline=true,arrowinset=0.7,arrowlength=0.8, arrowsize=3.5pt 1.5]{->}(0,0)(38,0)\naput{$\tau_f$}
% }

}}

\end{small}
\endpspicture\]

\[\psset{unit=0.1cm,labelsep=2pt,nodesep=3pt,linewidth=0.8pt}
\pspicture(0,10)(70,65)
\begin{small}

%%%%%%%%%%%%%%%%%%%%%%%%%%%%%%%%%%%%%%%%%%%%%%%%%%%%%% top left

\rput(0,60){

\rput(-3.5,-14){\scr $T$}
\rput(0.5,-6){\scr $S$}
\rput(6,5){$A$}

% \psline(-2,2)(0,0)(2,2)
% \psline(0,0)(0,-2)

\rput(4,-13)
{

\psline(0,0)(-4,5)(4,5)(0,0)
\psline(-2,5)(-2,11) \pscircle(-2,11.5){0.5}
\psline(2,5)(2,15)
\psline(0,0)(0,-2)
\rput(0,3){\scr $b_s$}

\rput(2,10){
\psframe*[linecolor=white](-1.5,-2)(1.5,2)
\psframe(-1.5,-2)(1.5,2)
\rput(0,0){\scr $f$}
}

% \rput(0,-4.5){$B$}

\rput(-3,-8){
\psline(0,0)(-4.5,5)(4.5,5)(0,0)
\psline(-3,5)(-3,10) \pscircle(-3,10.5){0.5}
\psline(3,5)(3,7)

\psline(0,0)(0,-2)
\rput(0,3){\scr $b_t$}
}

\rput(15,5){\Tmap{\sigma_f}}

}

\rput(6,-21){
\pcline[doubleline=true,arrowinset=0.7,arrowlength=0.8, arrowsize=3.5pt 1.5]{-}(0,0)(6,-6)
}

}

%%%%%%%%%%%%%%%%%%%%%%%%%%%%%%%%%%%%%%%%%%%%%%%%%%%%%%%%%%%%%%% top middle

\rput(33,60){

\rput(0,-13){

\rput(4,8){
\rput(-4,0){
% \rput(-2,10){$T$}
% \rput(2,10){$S$}
% \rput(6,10){$A$}
}

\psline(0,0)(-4,5)(4,5)(0,0)
\psline(-2,5)(-2,7) \pscircle(-2,7.5){0.5}
\psline(2,5)(2,9) 
\psline(0,0)(0,-10)
\rput(0,3){\scr $a_s$}
}

% \rput(0,-4.5){$B$}

\rput(4,3){
\psframe*[linecolor=white](-1.5,-2)(1.5,2)
\psframe(-1.5,-2)(1.5,2)
\rput(0,0){\scr $f$}
}

% \rput(0,-4.5){$B$}

\rput(1,-8){
\psline(0,0)(-4.5,5)(4.5,5)(0,0)
\psline(-3,5)(-3,10) \pscircle(-3,10.5){0.5}
\psline(3,5)(3,7)

\psline(0,0)(0,-2)
\rput(0,3){\scr $b_t$}
}

\rput(20,5){\Tmap{\tau_f}}

}

\rput(-3,-21){
\pcline[doubleline=true,arrowinset=0.7,arrowlength=0.8, arrowsize=3.5pt 1.5]{-}(0,0)(-6,-6)
}

\rput(5,-21){
\pcline[doubleline=true,arrowinset=0.7,arrowlength=0.8, arrowsize=3.5pt 1.5]{-}(0,0)(6,-6)
}

}

%%%%%%%%%%%%%%%%%%%%%%%%%%%%%%%%%%%%%%%%%%%%%%%%%%%%%%%%%%%%%%%  top right

\rput(68,60){

\rput(0,7){

\rput(0,-7){
% \rput(-2,5){$T$}
% \rput(2,5){$S$}
% \rput(6,5){$A$}
}

% \psline(-2,2)(0,0)(2,2)
% \psline(0,0)(0,-2)

\rput(4,-13)
{

\psline(0,0)(-4,5)(4,5)(0,0)
\psline(-2,5)(-2,7) \pscircle(-2,7.5){0.5}
\psline(2,5)(2,9)

\psline(0,0)(0,-2)
\rput(0,3){\scr $a_s$}
% \rput(0,-4.5){$B$}

\rput(-3,-8){
\psline(0,0)(-4.5,5)(4.5,5)(0,0)
\psline(-3,5)(-3,15) \rput(-3,15){\pscircle(0,0.5){0.5}}
\psline(3,5)(3,7)

\psline(0,0)(0,-10)
\rput(0,3){\scr $a_t$}

\rput(0,-5){
\psframe*[linecolor=white](-1.5,-2)(1.5,2)
\psframe(-1.5,-2)(1.5,2)
\rput(0,0){\scr $f$}
}

}

}
}

\rput(-3,-21){
\pcline[doubleline=true,arrowinset=0.7,arrowlength=0.8, arrowsize=3.5pt 1.5]{-}(0,0)(-6,-6)
}

}

%%%%%%%%%%%%%%%%%%%%%%%%%%%%%%%%%%%%%%%%%%%%%%%%%%%%%% middle left

\rput(14,35){

% \rput(-3.5,-14){\scr $T$}
% \rput(0.5,-6){\scr $S$}
% \rput(6,5){$A$}

% \psline(-2,2)(0,0)(2,2)
% \psline(0,0)(0,-2)

\rput(4,-13)
{

\psline(0,0)(-4,5)(4,5)(0,0)
\psline(-2,5)(-2,11) \pscircle(-2,11.5){0.5}
\psline(2,5)(2,15)
\psline(0,0)(0,-2)
\rput(0,3){\scr $b_t$}

\rput(2,10){
\psframe*[linecolor=white](-1.5,-2)(1.5,2)
\psframe(-1.5,-2)(1.5,2)
\rput(0,0){\scr $f$}
}

% \rput(0,-4.5){$B$}

% 
% \rput(-3,-8){
% \psline(0,0)(-4.5,5)(4.5,5)(0,0)
% \psline(-3,5)(-3,10) \pscircle(-3,10.5){0.5}
% \psline(3,5)(3,7)
% 
% \psline(0,0)(0,-2)
% \rput(0,3){\scr $b_t$}
% }

% \rput(15,5){\Tmap{\sigma_f}}

}

\rput(9,-13){
\pcline[doubleline=true,arrowinset=0.7,arrowlength=0.8, arrowsize=3.5pt 1.5]{-}(0,0)(6,-6)
}

}

%%%%%%%%%%%%%%%%%%%%%%%%%%%%%%%%%%%%%%%%%%%%%%%%%%%%%% middle right

\rput(47,36){

\rput(0,-14){

\rput(4,8){
\rput(-4,0){
% \rput(-2,10){$T$}
% \rput(2,10){$S$}
% \rput(6,10){$A$}
}

\psline(0,0)(-4,5)(4,5)(0,0)
\psline(-2,5)(-2,7) \pscircle(-2,7.5){0.5}
\psline(2,5)(2,9) 
\psline(0,0)(0,-10)
\rput(0,3){\scr $a_s$}
}

% \rput(0,-4.5){$B$}

\rput(4,3){
\psframe*[linecolor=white](-1.5,-2)(1.5,2)
\psframe(-1.5,-2)(1.5,2)
\rput(0,0){\scr $f$}
}

% \rput(0,-4.5){$B$}

% \rput(1,-8){
% \psline(0,0)(-4.5,5)(4.5,5)(0,0)
% \psline(-3,5)(-3,10) \pscircle(-3,10.5){0.5}
% \psline(3,5)(3,7)
% 
% \psline(0,0)(0,-2)
% \rput(0,3){\scr $b_t$}
% }

% \rput(20,5){\Tmap{\tau_f}}

}

\rput(-3,-14){
\pcline[doubleline=true,arrowinset=0.7,arrowlength=0.8, arrowsize=3.5pt 1.5]{-}(0,0)(-6,-6)
}

% \rput(6,-26){
% \pcline[doubleline=true,arrowinset=0.7,arrowlength=0.8, arrowsize=3.5pt 1.5]{-}(0,0)(5,-6)
% }

}

%%%%%%%%%%%%%%%%%%%%%%%%%%%%%%%%%%%%%% bottom

\rput(27,18){

% \rput(-3.5,-14){\scr $T$}
% \rput(0.5,-6){\scr $S$}
% \rput(6,5){$A$}

% \psline(-2,2)(0,0)(2,2)
% \psline(0,0)(0,-2)

\rput(4,-13)
{

% \psline(0,0)(-4,5)(4,5)(0,0)
% \psline(-2,5)(-2,11) \pscircle(-2,11.5){0.5}
\psline(2,5)(2,15)
% \psline(0,0)(0,-2)
% \rput(0,3){\scr $b_t$}

\rput(2,10){
\psframe*[linecolor=white](-1.5,-2)(1.5,2)
\psframe(-1.5,-2)(1.5,2)
\rput(0,0){\scr $f$}
}

% \rput(0,-4.5){$B$}

% 
% \rput(-3,-8){
% \psline(0,0)(-4.5,5)(4.5,5)(0,0)
% \psline(-3,5)(-3,10) \pscircle(-3,10.5){0.5}
% \psline(3,5)(3,7)
% 
% \psline(0,0)(0,-2)
% \rput(0,3){\scr $b_t$}
% }

% \rput(15,5){\Tmap{\sigma_f}}

}

% \rput(8,-20){
% \pcline[doubleline=true,arrowinset=0.7,arrowlength=0.8, arrowsize=3.5pt 1.5]{-}(0,0)(5,-6)
% }

}

\end{small}
\endpspicture\]

Finally we check that these assignations are inverse to each other.  We start with $\zeta_f$, construct individual weak maps $\sigma_f$ and $\tau_f$, compile them back into a weak $TS$-map and check that this is equal to $\zeta_f$. This is seen from the commutativity of the following diagram, where the top region is the multiplication axiom for a weak map of $TS$-algebras.

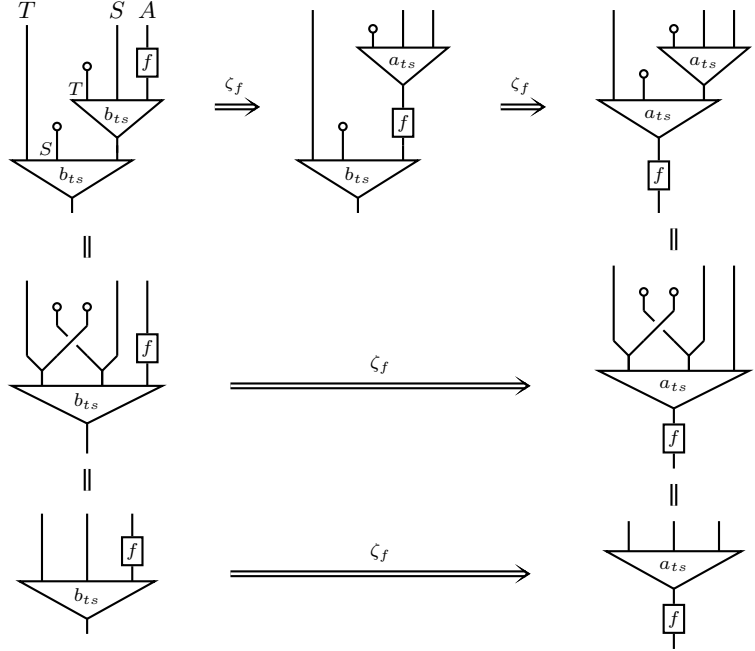
\[\psset{unit=0.1cm,labelsep=2pt,nodesep=3pt,linewidth=0.8pt}
\pspicture(0,-46)(80,41)
\begin{small}

%%%%%%%%%%%%%%%%%%%%%%%%%%%%%%%%%%%%%%%%% top left

\rput(0,35){

\rput(0,5){

\rput(0,-3){
\rput(-8,2){$T$}
\rput(4,2){$S$}
\rput(8,2){$A$}
}

\rput(-5.5,-19.5){\scr $S$}
\rput(-1.5,-11.5){\scr $T$}

}

\rput(4,-13)
{

\rput(0,0){
\psline(0,0)(-6,5)(6,5)(0,0)
\psline(-4,5)(-4,9)
\pscircle(-4,9.5){0.5}
\psline(0,5)(0,15)
\psline(4,5)(4,15)
\psline(0,0)(0,-2)
\rput(0,3){\scr $b_{ts}$}

% \psline(0,13)(-12,19)
% \psline[border=2pt](0,19)(-12,13)

}

\rput(4,10){
\psframe*[linecolor=white](-1.5,-2)(1.5,2)
\psframe(-1.5,-2)(1.5,2)
\rput(0,0){\scr $f$}
}

\rput(-6,-8){
\psline(0,0)(-8,5)(8,5)(0,0)
\psline(-6,5)(-6,23)
\psline(-2,5)(-2,9)
\pscircle(-2,9.5){0.5}
\psline(6,5)(6,7)
\psline(0,0)(0,-2)
\rput(0,3){\scr $b_{ts}$}
}

\rput(16,5){\Tmap{\zeta_f}}

}

}

%%%%%%%%%%%%%%%%%%%%%%%%%%%%%%%%%%%%%%%%%%%%%%%%%%%%%%%%%%%%%%% top middle

\rput(38,35){

% \rput(-8,5){$T$}
% \rput(-4,5){$S$}
% \rput(0,5){$T$}
% \rput(4,5){$S$}
% \rput(8,5){$A$}

\rput(4,-13)
{

\rput(0,7){
\psline(0,0)(-6,5)(6,5)(0,0)
\psline(-4,5)(-4,7)
\pscircle(-4,7.5){0.5}
\psline(0,5)(0,10)
\psline(4,5)(4,10)
\psline(0,0)(0,-8)
\rput(0,3){\scr $a_{ts}$}

% \rput(0,-4){
% \psline(0,13)(-12,19)
% \psline[border=2pt](0,19)(-12,13)
% }

}

\rput(0,2){
\psframe*[linecolor=white](-1.5,-2)(1.5,2)
\psframe(-1.5,-2)(1.5,2)
\rput(0,0){\scr $f$}
}

\rput(-6,-8){
\psline(0,0)(-8,5)(8,5)(0,0)
\psline(-6,5)(-6,25)
\psline(-2,5)(-2,9)
\pscircle(-2,9.5){0.5}
\psline(6,5)(6,7)

\psline(0,0)(0,-2)
\rput(0,3){\scr $b_{ts}$}
}

\rput(16,5){\Tmap{\zeta_f}}

}

}

%%%%%%%%%%%%%%%%%%%%%%%%%%%%%%%%%%%%%%%%%%%%%%%%%%%%%%%%%%%%%%% top right

\rput(78,35){

% \rput(-8,5){$T$}
% \rput(-4,5){$S$}
% \rput(0,5){$T$}
% \rput(4,5){$S$}
% \rput(8,5){$A$}

\rput(4,-13)
{

\rput(0,7){
\psline(0,0)(-6,5)(6,5)(0,0)
\psline(-4,5)(-4,7)
\pscircle(-4,7.5){0.5}
\psline(0,5)(0,10)
\psline(4,5)(4,10)
\psline(0,0)(0,-2)
\rput(0,3){\scr $a_{ts}$}

% \rput(0,-4){
% \psline(0,13)(-12,19)
% \psline[border=2pt](0,19)(-12,13)
% }

}

\rput(-6,0){
\psline(0,0)(-8,5)(8,5)(0,0)
\psline(-6,5)(-6,17)
\psline(-2,5)(-2,8)
\pscircle(-2,8.5){0.5}
\psline(6,5)(6,7)
\psline(0,0)(0,-10)
\rput(0,3){\scr $a_{ts}$}
}

\rput(-6,-5){
\psframe*[linecolor=white](-1.5,-2)(1.5,2)
\psframe(-1.5,-2)(1.5,2)
\rput(0,0){\scr $f$}
}

}

}

\rput(0,10){
\pcline[doubleline=true,arrowinset=0.7,arrowlength=0.8, arrowsize=3.5pt 1.5]{-}(0,0)(0,-5)
}

\rput(78,11){
\pcline[doubleline=true,arrowinset=0.7,arrowlength=0.8, arrowsize=3.5pt 1.5]{-}(0,0)(0,-5)
}

\rput(18,-11){
\pcline[doubleline=true,arrowinset=0.7,arrowlength=0.8, arrowsize=3.5pt 1.5]{->}(0,0)(42,0)\naput{\scr $\zeta_f$}
}

%%%%%%%%%%%%%%%%%%%%%%%%%%%%%%%%%%%%%%%%%%%%%%%%%%%%%%%%%%%%%%% left middle

\rput(0,-8){

\rput(0,0){

\rput(0,5){
% \rput(-8,5){$T$}
% \rput(-4,5){$S$}
% \rput(0,5){$T$}
% \rput(4,5){$S$}
% \rput(8,5){$A$}
}

\rput(-2,3){
\psline(-6,8)(-6,-2)
\psline(-2,2)(2,-2)
\psline[border=2pt](2,2)(-2,-2)
\psline(6,8)(6,-2)

\psline(2,2)(2,4)\pscircle(2,4.5){0.5}
\psline(-2,2)(-2,4)\pscircle(-2,4.5){0.5}

% \rput(6,-7){
% \psline(0,13)(-12,19)
% \psline[border=2pt](0,19)(-12,13)
% }

}

\rput(-6,-1){
\psline(-2,2)(0,0)(2,2)
\psline(0,0)(0,-2)
}

\rput(2,-1){
\psline(-2,2)(0,0)(2,2)
\psline(0,0)(0,-2)
}

\rput(-1,-8){

\rput(1,0){
\psline(0,0)(-10,5)(10,5)(0,0)
\psline(-6,5)(-6,7)
\psline(2,5)(2,7)
\psline(8,5)(8,19)
\psline(0,0)(0,-4)
\rput(0,3){\scr $b_{ts}$}
}

\rput(9,10){
\psframe*[linecolor=white](-1.5,-2)(1.5,2)
\psframe(-1.5,-2)(1.5,2)
\rput(0,0){\scr $f$}
}

}

}

}

%%%%%%%%%%%%%%%%%%%%%%%%%%%%%%%%%%%%%%%%%%%%%%%%%%%%%%%%%%%%%%% right middle

\rput(78,-8){

\rput(0,0){

\rput(0,5){
% \rput(-8,5){$T$}
% \rput(-4,5){$S$}
% \rput(0,5){$T$}
% \rput(4,5){$S$}
% \rput(8,5){$A$}
}

% \rput(-2,5){
% \psline(-6,2)(-6,-2)
% \psline(-2,2)(2,-2)
% \psline[border=2pt](2,2)(-2,-2)
% \psline(6,2)(6,-2)
% }

\rput(-2,5){
\psline(-6,8)(-6,-2)
\psline(-2,2)(2,-2)
\psline[border=2pt](2,2)(-2,-2)
\psline(6,8)(6,-2)

\psline(2,2)(2,4)\pscircle(2,4.5){0.5}
\psline(-2,2)(-2,4)\pscircle(-2,4.5){0.5}

% \rput(6,-7){
% \psline(0,13)(-12,19)
% \psline[border=2pt](0,19)(-12,13)
% }

}

\rput(-6,1){
\psline(-2,2)(0,0)(2,2)
\psline(0,0)(0,-2)
}

\rput(2,1){
\psline(-2,2)(0,0)(2,2)
\psline(0,0)(0,-2)
}

\rput(-1,-8){

\rput(1,2){
\psline(0,0)(-10,5)(10,5)(0,0)
\psline(-6,5)(-6,7)
\psline(2,5)(2,7)
\psline(8,5)(8,19)
\psline(0,0)(0,-8)
\rput(0,3){\scr $a_{ts}$}
}

\rput(1,-2){
\psframe*[linecolor=white](-1.5,-2)(1.5,2)
\psframe(-1.5,-2)(1.5,2)
\rput(0,0){\scr $f$}
}

}

}

}

%%%%%%%%%%%%%%%%%%%%%%%%%%%%%%%%%%%%%%%%%%%%%%%%%%%%%%%% bottom left

\rput(0,-37){

\rput(0,1){
% \rput(-4,10){$T$}
% \rput(0,10){$S$}
% \rput(4,10){$A$}
}

\rput(0,-5){
\psline(0,0)(-9,5)(9,5)(0,0)
\psline(-6,5)(-6,14)
\psline(0,5)(0,14)
\psline(6,5)(6,14)
\psline(0,0)(0,-2)

% \rput(0,11){
% \psline(-6,0)(0,4)
% \psline[border=2pt](0,0)(-6,4)
% }

\rput(6,9){
\psframe*[linecolor=white](-1.5,-2)(1.5,2)
\psframe(-1.5,-2)(1.5,2)
\rput(0,0){\scr $f$}
}

\rput(0,3){\scr $b_{ts}$}
}

\rput(0,16){
\pcline[doubleline=true,arrowinset=0.7,arrowlength=0.8, arrowsize=3.5pt 1.5]{-}(0,0)(0,-5)
}

\rput(78,14){
\pcline[doubleline=true,arrowinset=0.7,arrowlength=0.8, arrowsize=3.5pt 1.5]{-}(0,0)(0,-5)
}

\rput(18,1){
\pcline[doubleline=true,arrowinset=0.7,arrowlength=0.8, arrowsize=3.5pt 1.5]{->}(0,0)(42,0)\naput{\scr $\zeta_f$}
}

}

%%%%%%%%%%%%%%%%%%%%%%%%%%%%%%%%%%%%%%%%%%%%%%%%%%%%%%% right middle

\rput(78,-39){

\rput(0,1){
% \rput(-4,10){$T$}
% \rput(0,10){$S$}
% \rput(4,10){$A$}
}

\rput(0,1){
\psline(0,0)(-9,5)(9,5)(0,0)
\psline(-6,5)(-6,9)
\psline(0,5)(0,9)
\psline(6,5)(6,9)
\psline(0,0)(0,-8)

% \rput(0,7){
% \psline(-6,0)(0,4)
% \psline[border=2pt](0,0)(-6,4)
% }

\rput(0,-4){
\psframe*[linecolor=white](-1.5,-2)(1.5,2)
\psframe(-1.5,-2)(1.5,2)
\rput(0,0){\scr $f$}
}

\rput(0,3){\scr $a_{ts}$}
}

}

\end{small}
\endpspicture\]

Conversely start with $\sigma_f$ and $\tau_f$, make $\zeta_f$, and then extract individual weak $S$-map and $T$-map constraints out, and check that these are equal to the $\sigma_f$ and $\tau_f$ we started with. This is seen from the commutativity of the following diagrams; in each case the triangle commutes by the triangle axiom for a weak map.

\[\psset{unit=0.1cm,labelsep=2pt,nodesep=3pt,linewidth=0.8pt}
\pspicture(0,-37)(114,23)
\begin{small}

%%%%%%%%%%%%%%%%%%%%%%%%%%%%%%%%%%%%%%%%%%%%%%%%%%%%%%%%%%%%%%% top left

% \rput(2,5){
% \pcline[doubleline=true,arrowinset=0.7,arrowlength=0.8, arrowsize=3.5pt 1.5]{-}(0,0)(0,-10)
% }

\rput(0,20){

\rput(0,0){

% \rput(-2,5){$T}
\rput(2,5){$S$}
\rput(6,5){$A$}

% \psline(-2,2)(0,0)(2,2)
% \psline(0,0)(0,-2)

\rput(4,-13)
{

\psline(0,0)(-4,5)(4,5)(0,0)
\psline(-2,5)(-2,15)
\psline(2,5)(2,15)

\rput(2,10){
\psframe*[linecolor=white](-1.5,-2)(1.5,2)
\psframe(-1.5,-2)(1.5,2)
\rput(0,0){\scr $f$}
}

\psline(0,0)(0,-2)
\rput(0,3){\scr $b_s$}
% \rput(0,-4.5){$B$}

\rput(-3,-8){
\psline(0,0)(-4.5,5)(4.5,5)(0,0)
\psline(-3,5)(-3,20)
\rput(-3,20){\pscircle(0,0.5){0.5}}

\psline(3,5)(3,7)

\psline(0,0)(0,-2)
\rput(0,3){\scr $b_t$}
}

\rput(10,5){\msTmap{\sigma_f}}

}

%%%%%%%%%%%%%%%%%%%%%%%%%%%%%%%%%%%%%%%%%%%%%%%%%%%%%%%%%%%%%%% top middle

\rput(18,20){

\rput(4,-33){

\rput(4,8){
\rput(-4,0){
% \rput(-2,10){$T$}
\rput(2,10){$S$}
\rput(6,10){$A$}
}

\rput(0,-2){

\psline(0,0)(-4,5)(4,5)(0,0)
\psline(-2,5)(-2,9)
\psline(2,5)(2,9)
\psline(0,0)(0,-8)
\rput(0,3){\scr $a_s$}
}

}

% \rput(0,-4.5){$B$}

\rput(4,1.5){
\psframe*[linecolor=white](-1.5,-2)(1.5,2)
\psframe(-1.5,-2)(1.5,2)
\rput(0,0){\scr $f$}
}

% \rput(0,-4.5){$B$}

\rput(1,-8){
\psline(0,0)(-4.5,5)(4.5,5)(0,0)
\psline(-3,5)(-3,20)
\rput(-3,20){\pscircle(0,0.5){0.5}}
\psline(3,5)(3,7)

\psline(0,0)(0,-2)
\rput(0,3){\scr $b_t$}
}

}

}
}

\rput(34,-8){\msTmap{\tau_f}}

}

%%%%%%%%%%%%%%%%%%%%%%%%%%%%%%%%%%%%%%%%%%%%%%%%%%%%%%%%%%%%%%%  top right

% \rput(72,5){
% \pcline[doubleline=true,arrowinset=0.7,arrowlength=0.8, arrowsize=3.5pt 1.5]{-}(0,0)(0,-10)
% }

\rput(44,20){

\rput(0,7){

\rput(0,-7){
% \rput(-2,5){$T$}
\rput(2,5){$S$}
\rput(6,5){$A$}
}

% \psline(-2,2)(0,0)(2,2)
% \psline(0,0)(0,-2)

\rput(4,-13)
{

\psline(0,0)(-4,5)(4,5)(0,0)
\psline(-2,5)(-2,8)
\psline(2,5)(2,8)

\psline(0,0)(0,-2)
\rput(0,3){\scr $a_s$}
% \rput(0,-4.5){$B$}

\rput(-3,-8){
\psline(0,0)(-4.5,5)(4.5,5)(0,0)
\psline(-3,5)(-3,13)
\rput(-3,13){\pscircle(0,0.5){0.5}}
\psline(3,5)(3,7)

\psline(0,0)(0,-10)
\rput(0,3){\scr $a_t$}

\rput(0,-5){
\psframe*[linecolor=white](-1.5,-2)(1.5,2)
\psframe(-1.5,-2)(1.5,2)
\rput(0,0){\scr $f$}
}

}

}
}
}

%%%%%%%%%%%%%%%%%%%%%%%%%%%%%%%%%%%%%%%%%%%%%%%%%%%%%%%%%%%%%%%%%%%% sigma

\rput(0,-36){

\rput(11,5){
\pcline[doubleline=true,arrowinset=0.7,arrowlength=0.8, arrowsize=3.5pt 1.5]{->}(0,0)(26,0)\nbput{\scr $\sigma_f$}
}

\rput(1,30){
\pcline[doubleline=true,arrowinset=0.7,arrowlength=0.8, arrowsize=3.5pt 1.5]{-}(0,0)(0,-8) %\nbput{\scr $\sigma_f$}
}

\rput(45,30){
\pcline[doubleline=true,arrowinset=0.7,arrowlength=0.8, arrowsize=3.5pt 1.5]{-}(0,0)(0,-8) %\nbput{\scr $\sigma_f$}
}

\rput(26,33){
\pcline[doubleline=true,arrowinset=0.7,arrowlength=0.8, arrowsize=3.5pt 1.5]{-}(0,0)(13,-13)%\nbput{\scr $\sigma_f$}
}

%%%%%%%%%%%%%%%%%%%%%%%%%%%%%%%%%%%%%%%%%%%%%%%%%%%%%%%%% bottom left

\rput(1,0){

\rput(0,4){
\rput(-2,14){$S$}
\rput(2,14){$A$}
}

\psline(0,0)(-4,5)(4,5)(0,0)
\psline(-2,5)(-2,15)
\psline(2,5)(2,15)

\rput(2,10){
\psframe*[linecolor=white](-1.5,-2)(1.5,2)
\psframe(-1.5,-2)(1.5,2)
\rput(0,0){\scr $f$}
}

\psline(0,0)(0,-2)
\rput(0,3){\scr $b_s$}
\rput(0,-4.5){$B$}

}

% \rput(13,5){\Tmap{\sigma_f}}

%%%%%%%%%%%%%%%%%%%%%%%%%%%%%%%%%%%%%%%%%%%%%%%%%% bottom right

\rput(44,0){

\rput(0,8){
\rput(-2,10){$S$}
\rput(2,10){$A$}
% \rput(0,-5){$A$}

\psline(0,0)(-4,5)(4,5)(0,0)
\psline(-2,5)(-2,7)
\psline(2,5)(2,7)
\psline(0,0)(0,-10)
\rput(0,3){\scr $a_s$}
}

% \rput(0,-4.5){$B$}

\rput(0,3){
\psframe*[linecolor=white](-1.5,-2)(1.5,2)
\psframe(-1.5,-2)(1.5,2)
\rput(0,0){\scr $f$}
}

\rput(0,-4.5){$B$}

}

%%%%%%%%%%%%%%%%%%%%%%%%%%%%%%%%%%%%%%%%%%%%%%%%%%%%%%%%%

}

%%%%%%%%%%%%%%%%%%%%%%%%%%%%%%%%%%%%%%%%%%%%%%%%%%%% end sigma

%%%%%%%%%%%%%%%%%%%%%%%%%%%%%%%%%%%%%%%%%%%%%%%%%%%%%%% RH diagram

\rput(67,0){

%%%%%%%%%%%%%%%%%%%%%%%%%%%%%%%%%%%%%%%%%%%%%%%%%%%%%%%%%%%%%%% top left

% \rput(2,5){
% \pcline[doubleline=true,arrowinset=0.7,arrowlength=0.8, arrowsize=3.5pt 1.5]{-}(0,0)(0,-10)
% }

\rput(0,20){

\rput(0,0){

\rput(-2,5){$T$}
% \rput(2,5){$S$}
\rput(6,5){$A$}

% \psline(-2,2)(0,0)(2,2)
% \psline(0,0)(0,-2)

\rput(4,-13)
{

\psline(0,0)(-4,5)(4,5)(0,0)
\psline(-2,5)(-2,12)
\rput(-2,12){\pscircle(0,0.5){0.5}}

\psline(2,5)(2,15)

\rput(2,10){
\psframe*[linecolor=white](-1.5,-2)(1.5,2)
\psframe(-1.5,-2)(1.5,2)
\rput(0,0){\scr $f$}
}

\psline(0,0)(0,-2)
\rput(0,3){\scr $b_s$}
% \rput(0,-4.5){$B$}

\rput(-3,-8){
\psline(0,0)(-4.5,5)(4.5,5)(0,0)
\psline(-3,5)(-3,23)

\psline(3,5)(3,7)

\psline(0,0)(0,-2)
\rput(0,3){\scr $b_t$}
}

\rput(10,5){\msTmap{\sigma_f}}

}

%%%%%%%%%%%%%%%%%%%%%%%%%%%%%%%%%%%%%%%%%%%%%%%%%%%%%%%%%%%%%%% top middle

\rput(18,20){

\rput(4,-33){

\rput(4,8){
\rput(-4,0){
\rput(-2,10){$T$}
% \rput(2,10){$S$}
\rput(6,10){$A$}
}

\rput(0,-2){
\psline(0,0)(-4,5)(4,5)(0,0)
\psline(-2,5)(-2,7)
\rput(-2,7){\pscircle(0,0.5){0.5}}

\psline(2,5)(2,9)
\psline(0,0)(0,-8)
\rput(0,3){\scr $a_s$}
}

}

% \rput(0,-4.5){$B$}

\rput(4,1.5){
\psframe*[linecolor=white](-1.5,-2)(1.5,2)
\psframe(-1.5,-2)(1.5,2)
\rput(0,0){\scr $f$}
}

% \rput(0,-4.5){$B$}

\rput(1,-8){
\psline(0,0)(-4.5,5)(4.5,5)(0,0)
\psline(-3,5)(-3,23)
\psline(3,5)(3,7)

\psline(0,0)(0,-2)
\rput(0,3){\scr $b_t$}
}

}

}
}

\rput(34,-8){\msTmap{\tau_f}}

}

%%%%%%%%%%%%%%%%%%%%%%%%%%%%%%%%%%%%%%%%%%%%%%%%%%%%%%%%%%%%%%%  top right

% \rput(72,5){
% \pcline[doubleline=true,arrowinset=0.7,arrowlength=0.8, arrowsize=3.5pt 1.5]{-}(0,0)(0,-10)
% }

\rput(43,20){

\rput(0,7){

\rput(0,-7){
\rput(-2,5){$T$}
% \rput(2,5){$S$}
\rput(6,5){$A$}
}

% \psline(-2,2)(0,0)(2,2)
% \psline(0,0)(0,-2)

\rput(4,-13)
{

\rput(0,-1){
\psline(0,0)(-4,5)(4,5)(0,0)
\psline(-2,5)(-2,7)
\rput(-2,7){\pscircle(0,0.5){0.5}}

\psline(2,5)(2,9)

\psline(0,0)(0,-2)
\rput(0,3){\scr $a_s$}
% \rput(0,-4.5){$B$}
}

\rput(-3,-8){
\psline(0,0)(-4.5,5)(4.5,5)(0,0)
\psline(-3,5)(-3,16)
\psline(3,5)(3,7)

\psline(0,0)(0,-10)
\rput(0,3){\scr $a_t$}

\rput(0,-5){
\psframe*[linecolor=white](-1.5,-2)(1.5,2)
\psframe(-1.5,-2)(1.5,2)
\rput(0,0){\scr $f$}
}

}

}
}
}

%%%%%%%%%%%%%%%%%%%%%%%%%%%%%%%%%%%%%%%%%%%%%%%%%%%%%%%%%%%%%%%%%%%% tau

\rput(0,-36){

\rput(11,5){
\pcline[doubleline=true,arrowinset=0.7,arrowlength=0.8, arrowsize=3.5pt 1.5]{->}(0,0)(26,0)\nbput{\scr $\tau_f$}
}

\rput(1,30){
\pcline[doubleline=true,arrowinset=0.7,arrowlength=0.8, arrowsize=3.5pt 1.5]{-}(0,0)(0,-8) %\nbput{\scr $\sigma_f$}
}

\rput(44,30){
\pcline[doubleline=true,arrowinset=0.7,arrowlength=0.8, arrowsize=3.5pt 1.5]{-}(0,0)(0,-8) %\nbput{\scr $\sigma_f$}
}

\rput(20,33){
\pcline[doubleline=true,arrowinset=0.7,arrowlength=0.8, arrowsize=3.5pt 1.5]{-}(0,0)(-13,-13)%\nbput{\scr $\sigma_f$}
}

%%%%%%%%%%%%%%%%%%%%%%%%%%%%%%%%%%%%%%%%%%%%%%%%%%%%%%%%% bottom left

\rput(1,0){

\rput(0,4){
\rput(-2,14){$T$}
\rput(2,14){$A$}
}

\psline(0,0)(-4,5)(4,5)(0,0)
\psline(-2,5)(-2,15)
\psline(2,5)(2,15)

\rput(2,10){
\psframe*[linecolor=white](-1.5,-2)(1.5,2)
\psframe(-1.5,-2)(1.5,2)
\rput(0,0){\scr $f$}
}

\psline(0,0)(0,-2)
\rput(0,3){\scr $b_t$}
\rput(0,-4.5){$B$}

% \rput(13,5){\Tmap{\tau_f}}
}

% \rput(13,5){\Tmap{\sigma_f}}

%%%%%%%%%%%%%%%%%%%%%%%%%%%%%%%%%%%%%%%%%%%%%%%%%% bottom right

\rput(44,0){

\rput(0,8){
\rput(-2,10){$T$}
\rput(2,10){$A$}
% \rput(0,-5){$A$}

\psline(0,0)(-4,5)(4,5)(0,0)
\psline(-2,5)(-2,7)
\psline(2,5)(2,7)
\psline(0,0)(0,-10)
\rput(0,3){\scr $a_t$}
}

% \rput(0,-4.5){$B$}

\rput(0,3){
\psframe*[linecolor=white](-1.5,-2)(1.5,2)
\psframe(-1.5,-2)(1.5,2)
\rput(0,0){\scr $f$}
}

\rput(0,-4.5){$B$}

}

%%%%%%%%%%%%%%%%%%%%%%%%%%%%%%%%%%%%%%%%%%%%%%%%%%%%%%%%%

}

%%%%%%%%%%%%%%%%%%%%%%%%%%%%%%%%%%%%%%%%%%%%%%%%%%%% end sigma

%%%%%%%%%%%%%%%%%%%%%%%%%%%%%%%%%%%%%% end RH diagram

}

\end{small}
\endpspicture\]

\end{proof}

%%%%%%%%%%%%%%%%%%%%%%%%%%%%%%%%%%%%%%%%%%%%%%%%% 
% bib
%%%%%%%%%%%%%%%%%%%%%%%%%%%%%%%%%%%%%%%%%%%%%%%%%

% \bibliography{../macros/bib2112}

\end{document}